\documentclass{article}

\usepackage{anysize}
\marginsize{1.25in}{1.25in}{1in}{1in}

\usepackage{amsmath}
\usepackage{amssymb}
\usepackage{amsthm}
\usepackage{mathrsfs}
\usepackage[all]{xy}
\usepackage{graphicx}
\usepackage{pinlabel}

\theoremstyle{plain} \newtheorem{theorem}{Theorem}[section]
\theoremstyle{plain} \newtheorem{proposition}[theorem]{Proposition}
\theoremstyle{plain} \newtheorem{lemma}[theorem]{Lemma}
\theoremstyle{plain} \newtheorem{corollary}[theorem]{Corollary}
\theoremstyle{plain} \newtheorem{conjecture}[theorem]{Conjecture}
\theoremstyle{definition} \newtheorem{definition}[theorem]{Definition}
\theoremstyle{remark} \newtheorem{remark}[theorem]{Remark}
\theoremstyle{remark} \newtheorem{example}[theorem]{Example}
\theoremstyle{remark} \newtheorem{fact}[theorem]{Fact}
\theoremstyle{remark} \newtheorem{warning}[theorem]{WARNING}
\theoremstyle{remark} 

\newcommand{\hyperbox}{\mathcal{H}}
\newcommand{\C}{\mathcal{C}}
\newcommand{\F}{\mathcal{F}}
\newcommand{\A}{\mathfrak{A}}

\newcommand{\hfbold}{\mathbf{HF}}

\newcommand{\incl}{\mathcal{I}}
\newcommand{\proj}{\mathcal{P}}
\newcommand{\destab}{\mathcal{D}}
\newcommand{\Link}{\mathcal{L}}
\newcommand{\Mink}{\mathcal{M}}

\DeclareMathOperator{\rk}{rk}

\newcommand{\boldalpha}{\mbox{\boldmath $\alpha$}}
\newcommand{\boldbeta}{\mbox{\boldmath $\beta$}}

\title{Framed Floer Homology}
\date{}
\author{Tye Lidman}

\begin{document}
\maketitle

\abstract{For any three-manifold presented as surgery on a framed link $(L,\Lambda)$ in an integral homology sphere $Y$, Manolescu and Ozsv\'ath construct a hypercube of chain complexes whose homology calculates the Heegaard Floer homology of $Y_\Lambda(L)$.  This carries a natural filtration that exists on any hypercube of chain complexes; we study the $E_2$ page of the associated spectral sequence, called Framed Floer homology.  One purpose of this paper is to show that Framed Floer homology is an invariant of oriented framed links, but not an invariant of the surgered manifold.  We discuss how this relates to an attempt at a combinatorial proof of the invariance of Heegaard Floer homology.  We also show that Framed Floer homology satisfies a surgery exact triangle analogous to that of Heegaard Floer homology.  This setup leads to a completely combinatorial construction of a surgery exact triangle in Heegaard Floer homology.  Finally, we study the Framed Floer homology of two-component links which surger to $S^2 \times S^1 \# S^2 \times S^1$ and have Property 2R.}


\section{Introduction}
Spectral sequences have become a pervasive tool in the study of invariants of various low-dimensional objects in topology and their interrelations.  A prime example of this is the link surgery spectral sequence constructed by Ozsv\'ath and Szab\'o in \cite{hfbranched}.  This has the reduced Khovanov homology of a link $L$ as its $E_2$ term and converges to $\widehat{HF}(\Sigma(\bar{L}))$, the Heegaard Floer homology of the double cover of $S^3$ branched over the mirror of $L$; the methods in their paper have been propagated to many additional results.

Their idea is to present $\widehat{HF}(\Sigma(\bar{L}))$ as the iterated mapping cone of maps induced by four-dimensional surgery cobordisms.  This essentially produces a hypercube of chain complexes in the sense of \cite{hflz}.  Therefore, one may induce a filtration on this complex as in \cite{hfbranched}; this is a filtration determined by the position in the hypercube (often called the height in Khovanov homology \cite{drorkhovanov}) which we call the $\varepsilon$-filtration.  The corresponding spectral sequence is thus the desired one to study.  Ozsv\'ath and Szab\'o prove that the complex $(E_1,d_1)$ is in fact the reduced Khovanov chain complex for $L$.

In another direction, given a framed, oriented link, $(\vec{L},\Lambda)$, in an integral homology sphere $Y$, Manolescu and Ozsv\'ath construct the link surgery formula (not to be confused with the link surgery spectral sequence).  The surgery formula is a chain complex, $\C^\circ(\hyperbox,\Lambda)$, where $\hyperbox$ is a complete system (a collection of Heegaard data constructed from $L$ (Section~\ref{reviewsection}), whose homology is isomorphic to a completed version of the mod 2 Heegaard Floer homology of $\Lambda$-surgery on $L$, $\mathbf{HF}^\circ(Y_{\Lambda}(L))$, for each flavor $\circ$ \cite{hflz}.  Here $\Lambda$ is determined by a collection of integral surgery coefficients for the components of $L$.  A flavor of Heegaard Floer homology refers to one of the many versions of the theory, $\circ = +,-,{ }_{\widehat{\quad}}, \infty$; this essentially amounts to a choice of base ring for the chain complex.  

The link surgery formula is presented as a hypercube of chain complexes, which in some sense is also created by an iterated mapping cone construction.  This generalizes the integer surgeries formula for knots of Ozsv\'ath-Szab\'o \cite{hfkz}.  Since we will work with the Manolescu-Ozsv\'ath theory, all Floer homology coefficients will be calculated mod 2 and all $U$ variables will be formally completed; in other words, we work over $\mathbb{F}[[U]]$ instead of $\mathbb{Z}[U]$, where $\mathbb{F} = \mathbb{Z}/2\mathbb{Z}$.  We therefore use their notation of $\mathbf{HF}^\circ$ for the mod 2, $U$-completed groups.  It is useful to note that $\widehat{HF} \cong \widehat{\mathbf{HF}}$.

Since any hypercube of chain complexes is naturally equipped with the $\varepsilon$-filtration, it makes sense to study the corresponding spectral sequence on the link surgery formula complex.  For example, this spectral sequence is used to compute $HF^\infty(Y,\mathfrak{s}_0;\mathbb{F})$ for torsion Spin$^c$ structures $\mathfrak{s}_0$ in \cite{hftriple}.  The goal of this paper is to study the $E_2$ term of this spectral sequence in the context of framed, oriented links $(\vec{L},\Lambda)$, which we call Framed Floer homology.  Framed Floer homology will always refer to the hat-flavor, and thus will be denoted by $\widehat{FFH}(\hyperbox,\Lambda)$, where again $\hyperbox$ is the Heegaard data used to construct the link surgery formula.  Like Heegaard Floer homology, $\widehat{FFH}$ decomposes by Spin$^c$ structures on $Y_{\Lambda}(L)$.  

\begin{remark}
One can also define versions of Framed Floer homology for the other flavors; in the case of $+$ and $-$, this does not give much insight into understanding the corresponding Heegaard Floer homology.  This is because the $E_\infty$ page of the $\varepsilon$-spectral sequence is not necessarily isomorphic to the homology of the link surgery formula if we are not working over a field.  On the other hand, we expect $FFH^\infty$ to be completely determined by $H_1(Y_\Lambda(L))$ and thus do not focus on it.  This is why we restrict our attention to $\widehat{FFH}$.  
\end{remark} 

Framed Floer homology will later be rephrased as the homology of a complex whose chain groups correspond to the Heegaard Floer homology of large surgeries on sublinks of $L$ and whose differential is calculated by induced maps coming from two-handle cobordisms between the surgered manifolds.

\subsection{Framed Floer Homology and Invariance}
The first natural question is to what level of invariance Framed Floer homology satisfies.  The main goal is to address this.  

\begin{theorem}\label{invariancetheorem} The stable isomorphism-type of Framed Floer homology, $\widehat{FFH}(\hyperbox,\Lambda,\mathfrak{s})$, is an invariant of an oriented, framed link $(\vec{L},\Lambda)$ in $Y$ and a choice of Spin$^c$ structure on $Y_\Lambda(L)$.  
\end{theorem}

Here, {\em stable isomorphism} means an isomorphism up to tensoring with copies of $H^*(S^1)$, where the exact number of such copies is made precise in Theorem~\ref{surgerytheorem}.  To avoid this concern in the introduction, we will work with basic systems (see Section~\ref{cobordismrephrasesection} for details).  In this case there are no factors of $H^*(S^1)$.  The general statements can be adjusted as necessary.  

\begin{remark}
From now on, when we refer to a framed link, $L$, we will assume that $L$ comes with a fixed orientation.  We will often not decorate this link $L$ with $\rightarrow$ even though it is oriented, as the decorated notation will be used later in the link surgery formula to denote any arbitrary orientation on $L$.  We expect that the Framed Floer homology should actually be independent of the choice of orientation of the link.    
\end{remark}

By Theorem~\ref{invariancetheorem}, we may correctly speak of $\widehat{FFH}(L,\Lambda,\mathfrak{s})$ (we will omit the underlying $Y$ from the notation).  
While gradings are a powerful tool in Heegaard Floer homology, we will not make use of them; we do remark though that all of the isomorphisms in our theorems do respect the relative-gradings.

\begin{remark}\label{linksurgerygeneralization}
The link surgery formula generalizes to nullhomologous framed links in arbitrary three-manifolds, $Y$, as long as one only considers the Spin$^c$ structures on $Y_{\Lambda}(L)$ which restrict to be torsion on $Y$.  In fact, we will always assume that our links are nullhomologous and the Spin$^c$ structures we work with on the ambient manifold $Y$ are torsion.  In this context, the Framed Floer homology of the empty link in $\mathbb{T}^3$ has rank 6 (Proposition 1.9 in \cite{absgraded}).  On the other hand, the Framed Floer homology of the 0-framed Borromean rings in $S^3$, which surgers to $\mathbb{T}^3$, has rank 8 (this can be deduced from Theorem 1.3 in \cite{hftriple} and Theorem~\ref{largesurgeries}).  
For this reason, when comparing Framed Floer homologies, we will only consider framed links in the same ambient manifold. 
\end{remark}

Since Heegaard Floer homology is an invariant of three-manifolds, Ozsv\'ath asked if the $E_2$ page of their spectral sequence (reduced Khovanov homology) was also an invariant of the double branched cover of the link.  Watson gives two links with the same cover, but different Khovanov homology \cite{liambranched}.  Based on the above remark, we would like to ask if Framed Floer homology is an invariant of the surgered manifold when restricted to links in a fixed ambient manifold.  

\begin{theorem}\label{noninvariancetheorem}
$\widehat{FFH}(L,\Lambda)$ is NOT an invariant of the three-manifold $Y_\Lambda(L)$.
\end{theorem}

The counterexample we will provide is given by comparing the disjoint union of the left- and right-handed trefoils to a link obtained by handlesliding one component over the other. 

Another way of studying the three-manifold invariance of Heegaard Floer homology has been proposed by Manolescu.  `Define' the Heegaard Floer homology of a manifold $Y$ to be the homology of the link surgery formula for a framed link $(L,\Lambda)$ such that $Y = S^3_\Lambda(L)$ (each $Y$ will always have such a description).  To show from this definition that Heegaard Floer homology is a three-manifold invariant one must prove that the homology of the link surgery formula is independent of the surgery presentation.  Manolescu, Ozsv\'ath, and Thurston have used the link surgery formula to give a completely combinatorial description of Heegaard Floer homology \cite{hflzcombo}.  Therefore, one could hopefully tailor such a proof to fit into this combinatorial framework as well.  We would thus like to understand how non-invariant Framed Floer homology really is, since this may give some insight into the obstructions to methods of such a proof.  We will also see later how in some cases this non-invariance can be studied on the level of four-manifolds.  
  
Recall that if two homeomorphic three-manifolds can be presented by surgery on oriented, framed links $(L_1,\Lambda_1)$ and $(L_2,\Lambda_2)$ in $S^3$, then these two pairs are related by changes in orientation of components, link isotopies, handleslides, and (de)stabilizations.  Due to Theorem~\ref{noninvariancetheorem}, we know that Framed Floer homology cannot be invariant under all such moves.  In particular, this implies that if a quasi-isomorphism between the surgery complexes for framed links presenting the same manifold preserves the $\varepsilon$-filtration, it cannot induce an isomorphism on either of the first two pages of the $\varepsilon$-filtration spectral sequence.  Since most of the usual maps that one writes down between link surgery complexes are $\varepsilon$-filtered, the construction of such explicit maps will be somewhat unnatural.  Most likely, such a map will not respect the $\varepsilon$-filtration.  

On the other hand, we will establish invariance under stabilizations.  Recall that a stabilization of $(L',\Lambda')$ is obtained by adding a $\pm 1$-framed, geometrically split unknot $U$ to $(L',\Lambda')$. 

\begin{proposition}\label{stabilizationinvariance}
If $(L,\Lambda)$ is obtained by a stabilization of $(L',\Lambda')$, and $\hyperbox$, $\hyperbox'$ are complete systems for $L, L'$ respectively, then there are isomorphisms
\[
H_*(\C^\circ(\hyperbox,\Lambda,\mathfrak{s})) \cong H_*(\C^\circ(\hyperbox',\Lambda',\mathfrak{s})) \text{ and } \widehat{FFH}(L,\Lambda,\mathfrak{s}) \cong \widehat{FFH}(L',\Lambda',\mathfrak{s}).
\] 
\end{proposition}     

\begin{remark}
The proof of this proposition works by making an initial choice of $\hyperbox$ based on $\hyperbox'$.  We then invoke the invariance of the link surgery formula.  However, the proof that the link surgery formula is invariant under the choice of complete system subtly makes use of the three-manifold invariance of Heegaard Floer homology.  We do suspect that this assumption can be avoided when restricting to the combinatorial setting.  
\end{remark}

Another natural invariance property of mod 2 Heegaard Floer homology for three-manifolds is that, gradings aside, it is preserved under orientation-reversal (Proposition 2.5 in \cite{hfpa}).  Recall that $-S^3_\Lambda(L) \cong S^3_{-\Lambda}(\bar{L})$.  Therefore, the analogous question to ask is in what capacity this might hold for $\widehat{FFH}(L,\Lambda)$ and $\widehat{FFH}(\bar{L},-\Lambda)$ for framed links in $S^3$.  
  
\begin{proposition}\label{nonmirrortheorem}
There exists a framed link $(L,\Lambda)$ in $S^3$ such that the ranks of $\widehat{FFH}(L,\Lambda)$ and $\widehat{FFH}(\bar{L},-\Lambda)$ are not equal.
\end{proposition}

\subsection{Properties of Framed Floer Homology}
While Theorem~\ref{noninvariancetheorem} may convince the reader that Framed Floer homology is not worth studying, there are useful properties that it does satisfy.  In analogy with Heegaard Floer homology, we have a K\"unneth formula.

\begin{proposition}\label{kunnethformula}
If $(L_1,\Lambda_1)$ and $(L_2,\Lambda_2)$ are framed, oriented links in $Y_1$ and $Y_2$ respectively, then  
\[
\widehat{FFH}(L_1 \coprod L_2, \Lambda_1 \oplus \Lambda_2, \mathfrak{s}_1 \# \mathfrak{s}_2) \cong \widehat{FFH}(L_1,\Lambda_1,\mathfrak{s}_1) \otimes \widehat{FFH}(L_2,\Lambda_2,\mathfrak{s}_2) .
\]
\end{proposition}

It also turns out that $\widehat{FFH}$ satisfies one of the more prominent and powerful properties of Heegaard Floer homology: a surgery exact triangle.  Exact triangles for both Heegaard Floer homology and Framed Floer homology will come simultaneously from a more general setup.

\begin{theorem} \label{sestheorem}
Fix a framed link $(L',\Lambda')$ and a nullhomologous knot $K$ in $Y_{\Lambda'}(L')$, which we have isotoped to sit outside of the surgery region.  Choose a complete system of hyperboxes, $\hyperbox$, for $L= L' \cup K \subseteq S^3$ and let $\hyperbox' = \hyperbox|_{L'}$ be the restricted complete system for $L'$.  Denote by $\Lambda^r$ the framing on $L$ with $\Lambda^r|_{L'} = \Lambda'$ and $\Lambda^r|_K = r$.  Then, there is a short exact sequence of $\varepsilon$-filtered complexes 
\[
0 \longrightarrow \C^\circ(\hyperbox',\Lambda') \longrightarrow \C^\circ(\hyperbox,\Lambda^0) \longrightarrow \C^\circ(\hyperbox,\Lambda^1) \longrightarrow 0
\]
which also induces a short exact sequence on the $E_0$ and on the $E_1$ pages of the respective $\varepsilon$-spectral sequence.
\end{theorem}

Recall that a {\em surgery exact triangle} is a long exact sequence relating the Floer homology groups of three different framed surgeries on a knot in a three-manifold (see, for example, Section 9 of \cite{hfpa}).  

By choosing $M = S^3_{\Lambda'}(L')$, Theorem~\ref{sestheorem} gives a long exact sequence for a nullhomologous knot in any three-manifold $M$ (compare to Theorem 1.7 of \cite{hfpa})
\[
\ldots \longrightarrow \hfbold^\circ(M) \longrightarrow \hfbold^\circ (M_0(K)) \longrightarrow \hfbold^\circ(M_1(K)) \longrightarrow 
\hfbold^\circ(M) \longrightarrow \ldots
\]
The analogous notion of a surgery exact triangle between Framed Floer homology groups is clear.  Taking homology of the $E_1$ pages in Theorem~\ref{sestheorem} gives  
\begin{corollary}\label{ffles} There is a surgery exact triangle:
\[
\ldots \longrightarrow \widehat{FFH}(L',\Lambda') \longrightarrow \widehat{FFH}(L,\Lambda^0) \longrightarrow \widehat{FFH}(L,\Lambda^1) \longrightarrow \widehat{FFH}(L',\Lambda') \longrightarrow \ldots
\]
\end{corollary}

The proof of Theorem~\ref{sestheorem} will be completely algebraic.  We may therefore apply the combinatorial link surgery formula to obtain

\begin{corollary}\label{combinatorialles}
Fix a nullhomologous knot $K$ in $Y$.  There is a combinatorial construction and proof of a surgery exact triangle for $\mathbf{HF}^\circ(Y)$, $\mathbf{HF}^\circ(Y_0(K))$, and $\mathbf{HF}^\circ(Y_1(K))$.  
\end{corollary}

\subsection{Framed Floer Homology and Property 2R}
\begin{definition}
We will say that a two-component link $L$ in $S^3$ has {\em Property 2R} if for any $\Lambda$-surgery on $L$ which yields $S^2 \times S^1 \# S^2 \times S^1$, the pair $(L,\Lambda)$ can be related to the 0-framed two-component unlink, $V$, by only handleslides (no stabilizations allowed).
\end{definition} 

Note that if $L$ cannot surger to $S^2 \times S^1 \# S^2 \times S^1$ it automatically has Property 2R.  This is a generalization of the classical Property R - a knot has Property R if it is the unknot or it does not surger to $S^2 \times S^1$.  Since Gabai proved in \cite{gabaifoliations3} that all knots have Property R, the following conjecture has been previously proposed.  

\begin{conjecture}
All two-component links have Property 2R.
\end{conjecture}

It is easy to see that if surgery on $L$ gives $S^2 \times S^1 \# S^2 \times S^1$, then the linking of the two-components of $L$ is trivial and the framing must be 0.  However, there is more that can be deduced about $L$.  For example, Gompf, Scharlemann, and Thompson have shown that the knot with minimal genus occurring in a counterexample to Property 2R (if one exists) must {\em not} be fibered \cite{gompfscharlemannthompson}.  Furthermore, they construct a family of links which they expect to be counterexamples.  Motivated by this, we utilize the non-invariance of Floer homology to try to create a potential obstruction to links having Property 2R.  
 
Let $\mathfrak{s}_0$ denote the unique torsion Spin$^c$ structure on $S^2 \times S^1 \# S^2 \times S^1$.  To put the following proposition in perspective, we first point out that 
\[
\rk \widehat{FFH}(V,\mathbf{0}) = \rk \widehat{FFH}(V,\mathbf{0},\mathfrak
{s}_0) = \rk \widehat{HF}(S^2 \times S^1 \# S^2 \times S^1) = 4.
\] 
This follows, for instance, by Proposition~\ref{kunnethformula} and Remark~\ref{knotsvanish}.  We will show that every link with Property 2R must satisfy this condition. 

\begin{proposition}\label{property2r}
Let $L$ be a two-component link in $S^3$ which surgers to $S^2 \times S^1 \# S^2 \times S^1$.  If the rank of $\widehat{FFH}(L,\mathbf{0},\mathfrak{s}_0)$ is not 4, then $L$ does not have Property 2R.  
\end{proposition}

It would be interesting to calculate the Framed Floer homology for the potential counterexamples mentioned above.  However, at this time, it is not clear how to perform this calculation.   

\subsection{Outline}
The paper is organized as follows.  In Section~\ref{reviewsection} we collect the relevant ideas and notation from the Manolescu-Ozsv\'ath link surgery formula for our purposes.  In Section~\ref{filteredsection} we review the algebraic machinery for filtrations that we will employ, as well as give the definition of Framed Floer homology.  We give an alternate interpretation of Framed Floer homology in Section~\ref{cobordismrephrasesection}.  Section~\ref{linkinvariancesection} is devoted to answering the question about framed link invariance.   The K\"unneth formula will be proven in Section~\ref{kunnethsection}.  Section~\ref{stabilizationsection} will be where we discuss three-manifold invariance from the perspective of the link surgery formula.  We will compute some examples in Section~\ref{noninvariancesection} to see that Framed Floer homology is not a three-manifold invariant.  We compare the Framed Floer homologies of a particular framed link and its mirror in Section~\ref{mirrorsection}.  In Section~\ref{surgeryexactsection} the exact sequences are constructed.  Section~\ref{property2rsection} addresses the relationship between Framed Floer homology and Property 2R.  The final section is devoted to the discussion of some possible future directions.

\section*{Acknowledgments}
I would like to thank Ciprian Manolescu for suggesting the problem of invariance via the link surgery formula and encouraging me to work on this project, as well as his wonderful insights as an advisor.  He especially helped with the material in Section~\ref{surgeryexactsection}.  I would also like to thank Liam Watson for interesting discussions and showing interest in Framed Floer homology.  Finally, I am indebted to Eamonn Tweedy for correcting my understanding of filtered chain complexes and spectral sequences.


\section{The Surgery Formula}\label{reviewsection}
In this section, we review the link surgery formula of Manolescu-Ozsv\'ath \cite{hflz}; this will be the underlying object to study for the remainder of the paper.  This is by no means a self-contained summary, so we refer the reader to \cite{hftriple} for another recap (or \cite{hflzcombo} for a construction with grid diagrams).  Recall that $\mathbf{HF}^\circ$ represents Heegaard Floer homology with mod 2 coefficients and the $U$-variable completed.  We assume the reader has familiarity with the material on Heegaard Floer homology from \cite{hfinvariance} and \cite{hfl}.  

Throughout, $L = L_1 \cup \ldots \cup L_{\ell}$ will be a fixed oriented, $\ell$-component link in an integer homology sphere $Y$.  We will also use $|L|$ to denote the number of components of $L$.  $\vec{M}$ will refer to an arbitrary orientation on the sublink $M \subseteq L$.  Furthermore, if a sublink is instead decorated by $+$ or has no vector decoration (resp. $-$), this means that all of the components are oriented in a manner compatible with (resp. opposite to) the fixed orientation of $L$.  Let $I_{\pm}(L,\vec{M})$ denote the set of indices of components in $\vec{M}$ which are oriented compatibly/oppositely with $L$.   Finally, all of our modules will always be over $\mathbb{F}$-algebras. 

\subsection{Hyperboxes of Chain Complexes}
\begin{definition} An $n$-\emph{dimensional hyperbox of size }$\mathbf{d} = (d_1,\ldots,d_n)$ is the subset of $\mathbb{N}^n$
\[
\mathbb{E}(\mathbf{d}) = \{(\varepsilon_1,\ldots,\varepsilon_n): 0 \leq \varepsilon_i \leq d_i\}. 
\]
If $\mathbb{E}(\mathbf{d})= \{0,1\}^n$, then we say $\mathbb{E}(\mathbf{d})$ is a \emph{hypercube}.  The \emph{length} of $\varepsilon$ is given by 
\[
\| \varepsilon \| = \sum \varepsilon_i.
\] 
\end{definition}
We give hyperboxes the partial-order determined by $\varepsilon' \leq \varepsilon$ if and only if each $\varepsilon'_i \leq \varepsilon_i$.  
We say that $\varepsilon$ and $\varepsilon'$ are {\em neighbors} if they differ by an element of $\{0,1\}^n$.  
\begin{example} Let $L$ be an $\ell$-component link.  We define an $\ell$-dimensional hypercube indexed by the sublinks of $L$ as follows.   The elements of the hypercube are given by $\ell$-tuples $\varepsilon(M)=\varepsilon_1(M)\ldots\varepsilon_\ell(M)$, where each $\varepsilon_i(M)$ satisfies 
\[
\varepsilon_i(M)  = \left\{
\begin{array}{rl}
1 & \text{if } L_i \subseteq M\\
0 & \text{if } L_i \nsubseteq M
\end{array} \right. 
\]
Therefore, $\varepsilon(M') \leq \varepsilon(M)$ if and only if $M' \subseteq M$.  
\end{example}

\begin{definition} Fix $\mathbb{E}(\mathbf{d})$.  An $n$\emph{-dimensional hyperbox of chain complexes for }$\mathbb{E}(\mathbf{d})$ is a collection of $\mathbb{Z}$-graded modules, $C_*^\varepsilon$, for each $\varepsilon \in \mathbb{E}(\mathbf{d})$, with maps $D^{\varepsilon'}_{\varepsilon}: C_*^{\varepsilon} \longrightarrow C_{*-1+\|\varepsilon'\|}^{\varepsilon + \varepsilon'}$ for all $\varepsilon' \in \{0,1\}^n$ satisfying
\begin{equation}\label{hypercuberelation}
\sum_{\tau \leq \varepsilon' \in \{0,1\}^n} \!\!\!\!\!\!\!
D_{\varepsilon + \tau}^{\varepsilon' - \tau} \circ D_{\varepsilon}^{\tau} = 0.
\end{equation}
Here, it is implicit that maps which originate from or land outside the hyperbox are zero.  We will usually omit the subscript of the $D^{\varepsilon'}_{\varepsilon}$ when the domain is clear.   
\end{definition}

\begin{remark}
For each $\varepsilon$, we have that $(C^\varepsilon,D^{\mathbf{0}}_\varepsilon)$ forms a chain complex.  These will be called {\em vertex complexes}.  We will often denote the differential for a vertex complex by $\partial$.   
\end{remark}

Suppose that a hyperbox of chain complexes is in fact defined over a hypercube $\mathbb{E}(\mathbf{d}) = \{0,1\}^n$.  Note that by defining $C_{tot,*} = \sum_{\varepsilon} C_{*+\|\varepsilon\|}^\varepsilon$ and $D = \sum_{\varepsilon} D^\varepsilon$ we get an authentic chain complex, called the {\em total complex}. 
We will often use $C$ to actually refer to $C_{tot}$.  

The link surgery formula will associate to a framed link $(L,\Lambda) \subseteq Y$ a hypercube of chain complexes $\C^-(\hyperbox,\Lambda)$ whose total complex has homology isomorphic to $\mathbf{HF}^-(Y_\Lambda(L))$.

\subsection{The $\mathfrak{A}$-Complexes}
We now make our notation for framings precise.  Choose a basis for $H_1(Y-L) \cong \mathbb{Z}^{\ell}$ given by the oriented meridians of $L$.  The framing, $\Lambda$, will be denoted by a symmetric, integer-valued $\ell \times \ell$-matrix.  The off-diagonal entries $\Lambda_{i,j}$ are given by the linking numbers of the components $L_i$ and $L_j$.  The diagonal entries $\Lambda_{i,i}$ correspond to the surgery coefficients.  Therefore, it is easy to see how each row vector, $\Lambda_i$, is a representation of $L_i$ as an element of $H_1(Y-L)$ in the chosen basis of oriented meridians.

Note that the inclusion of a sublink $M$ induces a natural map from $H_1(Y-L)$ to $H_1(Y-M)$.  We will use the notation $\Lambda|_{M}$ for the restriction of the framing $\Lambda$ to the sublink $M$ - this is the obvious $|M| \times |M|$ submatrix of $\Lambda$.

\begin{definition} Let $\mathbb{H}(L)_i$ be $\frac{lk(L_i,L-L_i)}{2} + \mathbb{Z}$.  The affine lattice $\mathbb{H}(L)$ over $H_1(Y-L)$ is defined by 
\[
\mathbb{H}(L) = \bigoplus^\ell_{i=1} \mathbb{H}(L)_i.
\] 
\end{definition}
The elements of $\mathbb{H}(L)$ correspond to Spin$^c$ structures on $Y$ relative to $L$ (Section 3.2 of \cite{hfl}).  Furthermore, there is a natural identification between Spin$^c$ structures on $Y_\Lambda(L)$ and the $\mathbb{H}(L)/H(L,\Lambda)$, where this notation means quotienting the lattice by the action of each row vector, $\Lambda_i$, on $\mathbb{H}(L)$ (Section 3.7 of \cite{hfl}).   
We will often not distinguish between an equivalence class $[\mathbf{s}] \in \mathbb{H}(L)/H(L,\Lambda)$ and the corresponding Spin$^c$ structure, $\mathfrak{s}$, that one obtains in $Y_{\Lambda}(L)$ via this identification.   

With this in mind, we would like a way to restrict our relative Spin$^c$ structures to sublinks of $L$.
\begin{definition} The map $\psi^{\vec{M}}:\mathbb{H}(L) \longrightarrow \mathbb{H}(L-M)$ is given by $\psi^{\vec{M}}(\mathbf{s}) =  \mathbf{s} - [\vec{M}]/2$.
\end{definition}

We will work with multi-pointed Heegaard diagrams for links, $\hyperbox = (\Sigma_g,\boldalpha,\boldbeta,\mathbf{z},\mathbf{w})$, where $\mathbf{z} =  \{z_1,\ldots,z_m\}$ and $\mathbf{w} = \{w_1,\ldots,w_k\}$ with $k \geq m$.  We will assume that $z_i$ and $w_i$ will always be on the same component for $1 \leq i \leq m$.  Therefore, we will refer to these additional $w_j$ ($j>m$) as free basepoints.  This is just an ordering convention that does not restrict any generality.  Finally, we require that all of our Heegaard diagrams are generic (the $\boldalpha$ and $\boldbeta$ curves intersect transversely) and satisfy the admissibility condition as in Definition 4.1 of \cite{hflz}.

When defining Floer complexes, there are lots of options for base rings to work over - for example, polynomial rings over $\ell$, $k$, and even just one $U$ variable can be found in the literature.  Therefore, we need a rule for declaring which base ring we are working with. 

\begin{definition} A \emph{coloring} with $p$ colors is a surjective function $\tau: \mathbf{z} \cup \mathbf{w} \longrightarrow \{1,\ldots,p\}$ such that all points on the same component of the link get mapped to the same `color'.  A coloring is \emph{maximal} if there are exactly $\ell+(k - m)$ colors.
\end{definition}

Given a coloring $\tau$ with $p$ colors, we will soon see that our Floer complexes are defined over power series rings with $p$ variables.  Colorings will also tell us which variables in our base ring will be counted by disks crossing over specified basepoints.  Note that a coloring gives a coloring of the link components as well.  Let $\tau_i = \tau(w_j)$, where $w_j$ is on the component $L_i$.  

The first step towards our construction of a hypercube of chain complexes for our framed link is to build the $\C^\varepsilon$.  Let's suppose we have a fixed colored Heegaard diagram, $\hyperbox^L = (\Sigma,\boldalpha,\boldbeta,\mathbf{z},\mathbf{w},\tau)$, for $L$.  Let $\mathbb{T}_\alpha = \alpha_1 \times \ldots \times \alpha_{g+k-1}$ and similarly for $\mathbb{T}_\beta$.  There is an absolute Alexander grading $A_i$ on $\mathbb{T}_\alpha \cap \mathbb{T}_\beta$ satisfying 
\[
A_{i}(x) - A_{i}(y) = \sum_{z_j \in L_i} n_{z_j}(\phi) - \sum_{w_j \in L_i} n_{w_j}(\phi),
\] 
for any $\phi \in \pi_2(x,y)$, $1 \leq i \leq \ell$.  

We may add the points $\pm \infty$ to $\mathbb{H}(L)_i$ to obtain  $\overline{\mathbb{H}}(L)_i$.  Set $\overline{\mathbb{H}}(L) = \oplus_i \overline{\mathbb{H}}(L)_i$.  We extend $\psi$ in the obvious way to $\overline{\mathbb{H}}(L)$.  For what follows, we take the necessary conventions with $\pm \infty$ to make the definition make sense, such as $a - \infty < 0$ and $(\infty + a) - (\infty+b) = a-b$.    

\begin{definition} Let $\mathbf{s}=(s_1,\ldots,s_\ell) \in \overline{\mathbb{H}}(L)$.  We will define a complex $\mathfrak{A}^-(\hyperbox^L,\mathbf{s})$ freely generated by $\mathbb{T}_{\alpha} \cap \mathbb{T}_{\beta}$ over $\mathbb{F}[[U_1,\ldots, U_p]]$ and equipped with the differential
\begin{align*}
\partial_{\mathbf{s}}(x) =\sum_{y \in \mathbb{T}_\alpha \cap \mathbb{T}_\beta} \sum_{\substack{{\phi \in \pi_2(x,y)} \\{\mu(\phi)=1}}} \#(\mathcal{M}(\phi)/\mathbb{R})  \cdot & U_{\tau_1}^{E^{1}_{s_1}(\phi)} \cdots U_{\tau_\ell}^{E^{\ell}_{s_\ell}(\phi)} \\
&\cdot U^{n_{w_{m+1}}(\phi)}_{\tau(w_{m+1})} \cdots U^{n_{w_k}(\phi)}_{\tau(w_k)} y,
\end{align*}
where 
\[
E^i_s(\phi) = (A_i(x)-s) \vee 0 - (A_i(y)-s) \vee 0 + \sum_{\tau(w_j) = \tau_i} n_{w_j}(\phi)
\]
and $x \vee y = \max\{x,y\}$.  Here, $\mathcal{M}(\phi)$ is counting holomorphic disks in Sym$^{g+k-1}(\Sigma)$.    
\end{definition}

\begin{remark}
Observe that if $s_i = \infty$, then $E^i_{s_i}(\phi) = \sum_{\tau(w_j) = \tau_i} n_{w_j}(\phi)$.  A similar statement holds for $-\infty$.  A special case of this is that the Heegaard diagram for the empty link, $\hyperbox^\emptyset$, represents $Y$, and thus there is only one $\mathfrak{A}^-$-complex, namely $\mathfrak{A}^-(\hyperbox^\emptyset) = \mathbf{CF}^-(\hyperbox^\emptyset)$.  
\end{remark}

\subsection{Hyperboxes of Heegaard Diagrams}
The key idea that we will work with is moving curves around on a fixed surface by isotopies and handleslides which do not cross over the basepoints.  We will call any set of curves on the same punctured surface $(\Sigma,\mathbf{z},\mathbf{w})$ {\em strongly equivalent} it they are related by only handleslides and isotopies that do not cross $\mathbf{z}$ or $\mathbf{w}$.  

Suppose we have a hyperbox of size $\mathbf{d} \in \mathbb{N}^n$.  Let $\mathbf{d}^\circ = (d_1^\circ,\ldots,d_n^\circ)$, where $d_i^\circ = (d_i-1) \vee 0$.  With this, set $n^\circ$ to be the number of $i$ such that $d_i \neq 0$.  

\begin{definition}
An \emph{empty $\beta$-hyperbox} of size $\mathbf{d} \in \mathbb{N}^n$ on the punctured surface $(\Sigma,\mathbf{z},\mathbf{w})$ is a collection of strongly equivalent curves $\boldbeta^{\varepsilon}$ for each $\varepsilon \in \mathbb{E}(\mathbf{d})$ as well as a map $\tau:\mathbf{z} \cup \mathbf{w} \longrightarrow \{1,\ldots,p\}$ for some $p$ which satisfies the following:  
for the $n^\circ$-dimensional unit hypercube, $\mathbb{E}_{n^\circ}$, and each fixed $\varepsilon \in \mathbb{E}(\mathbf{d}^\circ)$, the Heegaard diagram for the unlink in a connect sum of $S^2 \times S^1$'s,  $(\Sigma,\boldbeta^{\varepsilon'},\boldbeta^{\varepsilon''},\mathbf{z},\mathbf{w})$ with $\varepsilon', \varepsilon'' \in \varepsilon + \mathbb{E}_{n^\circ}$, admits $\tau$ as a coloring.   
\end{definition}

\begin{definition}
Suppose that $\tau$ is a coloring with $b$ free basepoints.  A \emph{filling} of an empty $\beta$-hyperbox consists of a choice of elements $\Theta_{\varepsilon,\varepsilon'} \in \mathfrak{A}_{b/2 + \|\varepsilon - \varepsilon'\| - 1}^-((\Sigma,\boldbeta^{\varepsilon},\boldbeta^{\varepsilon'},\mathbf{z},\mathbf{w}),\mathbf{0})$ for all neighbors $\varepsilon,\varepsilon' \in \mathbb{E}(\mathbf{d})$ with $\varepsilon < \varepsilon'$ satisfying: \\
1) if $\|\varepsilon - \varepsilon'\| = 1$, this is a cycle representing a generator of homology in the maximal degree, \\
2) the sum over all polygon counts in the various multi-diagrams consisting of $\boldbeta^{\varepsilon} = \boldbeta^{\varepsilon_0}, \boldbeta^{\varepsilon_1},\ldots, \boldbeta^{\varepsilon_r} = \boldbeta^{\varepsilon'}$ where $\| \varepsilon_{i+1} - \varepsilon_{i} \| = 1$ with fixed vertices $\Theta_{\varepsilon_i,\varepsilon_{i+1}}$ must vanish (Equation 50 in \cite{hflz}).  \\
An empty $\beta$-hyperbox with a choice of filling is called a $\beta${\em-hyperbox}.
\end{definition}

Manolescu and Ozsv\'ath show that every empty $\beta$-hyperbox admits a filling (Lemma 6.6 in \cite{hflz}).  It is clear that an analogous notion for $\alpha$-hyperboxes exists as well.  

\begin{definition}
A collection of functions $r_i:\{1,\ldots,d_i\} \longrightarrow \{\alpha,\beta\}$ for $d_i \in \mathbf{d}$ is called a set of \emph{bipartition functions}.  
\end{definition}

For each $\varepsilon \in \mathbb{E}(\mathbf{d})$, we set $\varepsilon_\alpha$ and $\varepsilon_\beta$ to be 
\[
\varepsilon^i_\alpha = \# (r_i^{-1}(\alpha) \cap \{1,\ldots, \varepsilon_i\}) \]
\[
\varepsilon^i_\beta = \# (r_i^{-1}(\beta) \cap \{1,\ldots, \varepsilon_i\}). 
\]

These $\varepsilon_\alpha$ and $\varepsilon_\beta$ naturally define new hyperboxes $\mathbb{E}(\mathbf{d}^\alpha)$ and $\mathbb{E}(\mathbf{d}^\beta)$.    

\begin{definition}
A \emph{hyperbox of Heegaard diagrams for a link} $\vec{L}$ consists of an $n$-dimensional hyperbox of size $\mathbf{d}$, a collection of bipartition functions $r_i$ containing a $\beta$-hyperbox of size $\mathbf{d}^\beta$ and an $\alpha$-hyperbox of size $\mathbf{d}^\alpha$ on a punctured surface $(\Sigma,\mathbf{z},\mathbf{w})$, and a function $\tau: \mathbf{z} \cup \mathbf{w} \longrightarrow \{1,\ldots,p\}$ for some $p$.  For each $\varepsilon \in \mathbb{E}(\mathbf{d})$, the Heegaard diagram 
\[
\hyperbox_\varepsilon = (\Sigma,\boldalpha^\varepsilon,\boldbeta^\varepsilon,\mathbf{z},\mathbf{w})
\]  
must represent a Heegaard diagram for $\vec{L}$ such that $\tau$ is a coloring.  
\end{definition}  

\begin{definition}
A \emph{hyperbox for the pair} $(L,\vec{M})$, where $M$ is an $n$-component sublink, is an $n$-dimensional hyperbox of Heegaard diagrams for $L-M$ and a choice of ordering for the components of $M$.  
\end{definition}

\begin{definition}  We say that $\hyperbox$ and $\hyperbox'$, two hyperboxes of Heegaard diagrams of size $\mathbf{d}$ for $\vec{M}$ with the same underlying Heegaard surface are \emph{surface isotopic} if there is a single self-diffeomorphism of the underlying Heegaard surface isotopic to the identity, taking $\hyperbox_\varepsilon$ to $\hyperbox'_\varepsilon$ for each $\varepsilon \in \mathbb{E}(\mathbf{d})$ and is supported away from $M$.  In other words, the basepoints, curves, and colorings are preserved at each $\varepsilon$ by the map.  If these are instead hyperboxes for pairs $(L,\vec{M})$, then they are {\em surface isotopic} if the underlying hyperboxes for $L-M$ are surface isotopic and the orderings of $M$ are the same.  
\end{definition}

Given a Heegaard diagram, $\hyperbox$, for $L$ and a choice of oriented sublink $\vec{M}$, we will construct a Heegaard diagram for $L - \vec{M}$.

\begin{definition}
The Heegaard diagram, $r_{\vec{M}}(\hyperbox)$, is the Heegaard diagram for $L-M$ defined by the following method.  If $z_i$ is on a component $L_j$ with $j \in I_+(L,\vec{M})$, then remove $z_i$ from $\hyperbox$.  If $w_i$ is on a component $L_j$ with $j \in I_-(L,\vec{M})$, then we remove $w_i$ and relabel the corresponding $z_i$ as $w_i$.  Restricting the original coloring gives the coloring for $r_{\vec{M}}(\hyperbox)$.
\end{definition}


\begin{remark}
We can apply the restriction operator $r$ to entire hyperboxes by simply doing this to each vertex in the cube.  Note that a filling for an empty hyperbox is a filling for the restricted empty hyperbox.   
\end{remark}

Let $\hyperbox^{L,\vec{M}}$ be a hyperbox of size $\mathbf{d}$ for the pair $(L,\vec{M})$ and suppose $M'' \subseteq M' \subseteq \vec{M}$.  Then, we consider a subhyperbox spanned by the two corners of $\hyperbox^{L,\vec{M}}$ given by  $(\varepsilon_1(M'') \cdot d_1,\ldots, \varepsilon_n(M'') \cdot d_n)$ and $(\varepsilon_1(M') \cdot d_1,\ldots, \varepsilon_n(M') \cdot d_n)$.  We denote this by $\hyperbox^{L,\vec{M}}(M'',M')$.  For notational purposes, we will use $\hyperbox^M$ for $\hyperbox^{M,\emptyset}$.  

The idea for a complete system is now becoming more clear; one would like to be able to relate the complexes $\mathfrak{A}^-(\hyperbox^L,\mathbf{s})$ to $\mathfrak{A}^-(\hyperbox^{L-M},\psi^{\vec{M}}(\mathbf{s}))$-complexes for its sublinks by studying some moves on the diagram level.  This requires some compatibility between the Heegaard diagrams.

\begin{definition}
Fix a link $L$.  Suppose that for a collection of hyperboxes of Heegaard diagrams $\hyperbox^{L',\vec{M}}$ for the pairs $(L',\vec{M})$ with $\vec{M} \subseteq L' \subseteq L$ and every $M' \subseteq \vec{M}$ (the orientation of $M'$ is induced from $\vec{M}$), $\hyperbox^{L',\vec{M}}(\emptyset,M')$ is surface isotopic to $r_{\vec{M}-M'}(\hyperbox^{L',M'})$ and $\hyperbox^{L',\vec{M}}(M',M)$ is surface isotopic to $\hyperbox^{L' - M',\vec{M} - M'}$.  This collection along with a choice of such isotopies form a \emph{complete pre-system of hyperboxes for} $L$.         
\end{definition}

As in \cite{hflz}, instead of complete pre-systems, we will need to work with {\em complete systems} in this paper.  In order to be a complete system, there is an additional homotopical condition given in Definition 6.27 in \cite{hflz} on the paths traced out by the basepoints via the isotopies.  We will ignore this technicality throughout, since we will hardly ever explicitly work with these isotopies. 

\begin{proposition}[Manolescu-Ozsv\'ath]
Complete systems of hyperboxes always exist for any $\vec{L}$. 
\end{proposition}

\subsection{The Link Surgery Theorem}\label{surgerytheoremstatement}

Fix a complete system of hyperboxes for $L$.  To define our hypercube of chain complexes, we define the chain groups to be $\C^{\varepsilon(M)} = \prod_{\mathbf{s} \in \mathbb{H}(L)} \C^{\varepsilon(M)}_{\mathbf{s}}$, where 
\[
\C^{\varepsilon(M)}_{\mathbf{s}} = \mathfrak{A}^-(\hyperbox^{L-M},\psi^{M}(\mathbf{s})).
\]
Note that the chain groups do not depend on the choice of framing that we put on our link.  

Let's now relate $\mathfrak{A}^-(\hyperbox^M,\mathbf{s})$ to $\mathfrak{A}^-(\hyperbox^{M-M'},\psi^{\vec{M'}}(\mathbf{s}))$ for $\vec{M'}$ a non-empty sublink of $M$.  This is done in essentially two steps.  First, there is a map $\incl$ derived from taking $\hyperbox^M$ to $r_{\vec{M'}}(\hyperbox^M)$; this comes from some multiplication by powers of $U$.  Next, there will be a map $\destab$ determined by transforming $r_{\vec{M'}}(\hyperbox^M)$ into $\hyperbox^{M-M'}$.  Such a transformation exists because of the compatibility conditions for a complete system.

Fix an oriented sublink $\vec{M}$.  We define a map $p^{\vec{M}}: \overline{\mathbb{H}}(L) \longrightarrow \overline{\mathbb{H}}(L)$ componentwise by  
\begin{equation*}
p^{\vec{M}}_i(s)  = \left\{
\begin{array}{rl}
+\infty & \text{ if } i \in I_+(L,\vec{M}) \\
-\infty & \text{ if } i \in I_-(L,\vec{M}) \\
s & \text{ if } i \neq \pm \infty.  
\end{array} \right. 
\end{equation*}

\begin{definition}
Let $\vec{M} \subseteq L$ and suppose that $s_i \neq \mp \infty$ if $i \in I_{\pm}(L,\vec{M})$.  The \emph{inclusion} $\incl^{\vec{M}}_{\mathbf{s}}:\mathfrak{A}^-(\hyperbox^L,\mathbf{s}) \longrightarrow \mathfrak{A}^-(\hyperbox^L,p^{\vec{M}}(\mathbf{s}))$ is given by the formula
\[
\incl^{\vec{M}}_{\mathbf{s}}(x) = \prod_{i \in I_+(L,\vec{M})} \!\!\!\!\!U_{\tau_i}^{(A_{i}(x) - s_i) \vee 0} \!\!\!\!\!\prod_{j \in I_-(L,\vec{M})} \!\!\!\!\! U_{\tau_j}^{(s_j - A_{j}(x)) \vee 0} x.
\]
\end{definition}

\begin{example}
In the integer surgery formula for knots (Theorem 1.1 of \cite{hfkz}), the inclusions $\incl_s^{+ K}$ and $\incl_s^{-K}$ correspond to the inclusions of $CFK^\infty\{i \vee (j-s) \leq 0\}$ into $CFK^\infty\{i \leq 0\}$ and $CFK^\infty\{j \leq s\}$ respectively.  
\end{example}

There is then an identification between $\mathfrak{A}^-(\hyperbox^L,p^{\vec{M}}(\mathbf{s}))$ and $\mathfrak{A}^-(r_{\vec{M}}(\hyperbox^L),\psi^{\vec{M}}(\mathbf{s}))$, since setting an $s_i$ to $\infty$ (resp. $-\infty$) has the same effect on the $\mathfrak{A}^-$-complex as ignoring $z$ (resp. $w$).  However, this is not the diagram for $\hyperbox^{L-M}$ that we used to define $\C^{\varepsilon(M)}$.  Therefore, we need a way to connect the two to define the $D^{\varepsilon}$ maps.  The important observation is that $\hyperbox^{L-M}$ and $r_{\vec{M}}(\hyperbox^L)$ are related by a sequence of isotopies and handleslides as determined by the complete system.  These Heegaard moves, as well as the identification mentioned, will induce the \emph{destabilization} map, $\destab_{p^{\vec{M}}(\mathbf{s})}^{\vec{M}}$, between the $\mathfrak{A}^-$-complexes.  

Since we will not be making explicit use of the destabilization maps, we do not give the details and instead refer the reader to Section 7.2 of \cite{hflz}.  We remark that when $\vec{M}$ is a knot, the destabilization maps are quasi-isomorphisms given by counting holomorphic triangles (Example 7.2 in \cite{hflz}).  In general, destabilizations by sublinks with more components will be given by higher polygon counts which represent higher dimensional analogues of chain homotopies (and are not necessarily quasi-isomorphisms).  They help relate the $|M|!$ maps obtained by destabilizing the components of $\vec{M}$ individually in a specified order.  

The composition, 
\[
\Phi_s^{\vec{M}} = \destab_{p^{\vec{M}}(s)}^{\vec{M}} \circ \incl_{\mathbf{s}}^{\vec{M}}
\]
is what we will use to build the $D^{\varepsilon}$ in the hypercube of chain complexes.  Finally, we set $\Phi_{\mathbf{s}}^{\emptyset}$ to be the differential $\partial_{\mathbf{s}}$ on the $\mathfrak{A}^-$-complexes, as we still need a $D^{\mathbf{0}}$ in any hyperbox of chain complexes.  

Now we are ready to write down our hypercube of chain complexes explicitly.  Recall that we have already constructed our chain groups $C^{\varepsilon(M)}$ for all sublinks $M$.  For notation, we index the elements of $C^{\varepsilon(M)}$ by using $(\mathbf{s},x)$ to denote $x \in \mathfrak{A}^-(\hyperbox^{L-M},\psi^M(\mathbf{s}))$.  Define $\Lambda_{L,\vec{M}}$ to be $\sum_{i \in I_-(L,\vec{M})} \Lambda_i$.  The differentials $D^\varepsilon$ are given by   
\[
D_{\varepsilon(N)}^{\varepsilon(M)}(\mathbf{s},x) = \bigoplus_{\vec{M}}(s+\Lambda_{L,\vec{M}},\Phi_{\psi^{N}(\mathbf{s})}^{\vec{M}}(x)) \text{ for } (\mathbf{s},x) \in C^{\varepsilon(N)}.
\]  
Here, the notation $\vec{M}$ in the summand means that we are summing over the $2^{|M|}$ orientations for our fixed sublink $M$.  

\begin{definition}
The {\em link surgery formula}, $\C^-(\hyperbox,\Lambda)$, is the total complex of the hypercube of chain complexes $(\C,D)$.  
\end{definition}

Because $[\mathbf{s}] = [\mathbf{s} + \Lambda_i] \in \mathbb{H}(L)/H(L,\Lambda)$ for all $i$, $\C^-(\hyperbox,\Lambda)$ splits into subcomplexes corresponding to Spin$^c$ structures on $Y_{\Lambda}(L)$.  Therefore, we define $\C^-(\hyperbox,\Lambda,\mathfrak{s})$ to be the subcomplex consisting of the $\mathfrak{A}^-(\hyperbox^{L-M},\psi^{M}(\mathbf{s}))$ such that $[\mathbf{s}]$ corresponds to $\mathfrak{s}$.

\begin{theorem}[Manolescu-Ozsv\'ath Link Surgery Theorem, Theorem 7.7 of \cite{hflz}]\label{surgerytheorem} Let $\hyperbox$ be a complete system of hyperboxes for $L$ with $p$ colors and $k$ basepoints of type $w$.  $\C^-(\hyperbox,\Lambda,\mathfrak{s})$, as defined above, is a hypercube of chain complexes.  All $U_i$'s act the same on $H_*(\C^-(\hyperbox,\Lambda,\mathfrak{s}))$.  Finally, $H_*(\C^-(\hyperbox,\Lambda,\mathfrak{s})) \cong \mathbf{HF}^-(Y_\Lambda(L),\mathfrak{s}) \otimes H^*(\mathbb{T}^{k-p})$ as a relatively-graded $\mathbb{F}[[U]]$-module.
\end{theorem}

As mentioned before, the relative gradings will not really come into play in this paper.  For the definition of the relative grading, see Section 7.4 of \cite{hflz}.

\begin{remark} Usually, one sets $U=0$ to define $\widehat{HF}$ theories.  The way that this is done for $\widehat{\C}(\hyperbox,\Lambda)$ is to choose some $U_i$ and set this equal to 0 at the chain level; in other words $\widehat{\C}(\hyperbox,\Lambda) = \C^-(\hyperbox,\Lambda)/U_i \cdot \C^-(\hyperbox,\Lambda)$.  The theorem in fact implies that the choice of $i$ does not affect the overall outcome on homology.  
For $\C^\infty(\hyperbox,\Lambda)$, one formally inverts all $U_i$ variables.  Finally, for $\C^+(\hyperbox,\Lambda)$, we simply take the quotient complex $\C^\infty(\hyperbox,\Lambda)/\C^-(\hyperbox,\Lambda)$.  If we want to leave the flavor unspecified, then we use the standard notation $\C^\circ(\hyperbox,\Lambda)$.  The analogous construction holds for the $\mathfrak{A}^\circ(\hyperbox^M,\mathbf{s})$.  The link surgery theorem also holds for the other flavors.
\end{remark}

\begin{remark} In order to show that $\C^-(\hyperbox,\Lambda)$ is in fact a hypercube of chain complexes, one must understand how the destabilizations and inclusions interact when applied to different sublinks to satisfy (\ref{hypercuberelation}).  Suppose that $\vec{M_1}$ and $\vec{M_2}$ are disjoint and that $\mathbf{s}$ has $s_i = \pm \infty$ when $\pm L_i$ is a compatibly oriented component of $\vec{M_1}$.  Then we have from Lemma 7.4 in \cite{hflz}:
\begin{equation}\label{destabinclcommute}
\incl_{\psi^{\vec{M_1}}(\mathbf{s})}^{\vec{M_2}} \circ \destab_{\mathbf{s}}^{\vec{M_1}} = \destab_{p^{\vec{M_2}}(\mathbf{s})}^{\vec{M_1}} \circ \incl_{\mathbf{s}}^{\vec{M_2}}.
\end{equation}
This equation will play a key role when we use the link surgery formula to construct surgery exact triangles.
\end{remark}

Furthermore, the link surgery formula satisfies a certain type of naturality in the sense that certain subcomplexes correspond to surgery on sublinks of $L$.  Fix a complete system $\hyperbox$ for $L$ and a sublink $L'$.  We define the complete system $\hyperbox|_{L'}$ for $L'$ to be given by only considering the hyperboxes of Heegaard diagrams for pairs $(L'',M)$, where $L'' \subseteq L'$.  

Let's consider $\mathbb{H}(L)/H(L,\Lambda|_{L'})$, where now we are only quotienting out by the action of $\Lambda_i$ with $L_i \subseteq L$.  The map $\psi^{L-L'}$ descends to  
\[
\psi^{L-L'}: \mathbb{H}(L)/H(L,\Lambda|_{L'}) \longrightarrow \mathbb{H}(L')/H(L',\Lambda|_{L'}).
\]
Furthermore, given an equivalence class $[\mathbf{t}] \in \mathbb{H}(L)/H(L,\Lambda|_{L'})$, one can define a subcomplex of $\C^\circ(\hyperbox,\Lambda)$,
\[
\C^\circ(\hyperbox,\Lambda)^{L',\mathbf{t}} = \bigoplus_{L-L' \subseteq M \subseteq L} \prod_{\{\mathbf{s} \in \mathbb{H}(L)|[\mathbf{s}] = [\mathbf{t}]\}} \mathfrak{A}^\circ(\hyperbox^{L-M},\psi^M(\mathbf{s})). 
\]  

\begin{remark}\label{restrictedsystems} Given a complete system $\hyperbox$ for $L$ and a sublink $L'$ such that $L-L'$ is nullhomologous in $Y_{\Lambda|_{L'}}(L')$, there is an identification of the subcomplex $\C^\circ(\hyperbox,\Lambda)^{L',\mathbf{t}}$ with $\C^\circ(\hyperbox|_{L'},\Lambda|_{L'},\psi^{L-L'}([\mathbf{t})])$ (see Section 11.1 of \cite{hflz} for the general case).  
\end{remark} 

This is a very key concept which will allow us to calculate the Heegaard Floer homology from surgeries on a link by understanding the surgeries on the various sublinks.

\subsection{Notation and Conventions}
With this machinery in mind, we would like to suppress most of this to the background and introduce some notation to simplify the expressions.  

Fix an oriented, framed link $(L,\Lambda)$ in an integral homology sphere $Y$ and a complete system of hyperboxes $\hyperbox$ for $L$.  Most importantly, we would like to abbreviate the notation for the vertex complexes.  If $\C$ is a hypercube of chain complexes, then we will take $\varepsilon$ to represent $\C^\varepsilon$.   Recall that for a sublink $M \subseteq L$, $\varepsilon_i(M)$ is 1 if the $i$th component of $L$ is in $M$ and 0 otherwise.  Therefore, for $\mathbf{s} \in \mathbb{H}(L)$ we use
\[
\varepsilon(M)_{\mathbf{s}} = \mathfrak{A}^\circ(\hyperbox^{L-M},\psi^M(\mathbf{s})).
\] 
Notice that this means that we are not applying the $\psi$ maps in our notation for the index $\mathbf{s}$.  We may also omit the $\mathbf{s}$ in the $\Phi$ maps when it is clear what the domain should be.  Finally, we use $[\varepsilon(M)_{\mathbf{s}}]$ to represent the homology of the $\mathfrak{A}$-complex,  $H_*(\varepsilon(M)_{\mathbf{s}},\partial_{\psi^M(\mathbf{s})})$.


\section{Framed Floer Homology and Filtered Algebra}\label{filteredsection}
\subsection{Filtered Algebra}
We collect some facts about filtered complexes that we will repeatedly reference.  Throughout the paper, $(\C,\F,\partial)$ will be a filtered chain complex of modules over an $\mathbb{F}$-algebra with filtration levels bounded; in our setup, we will always have a decomposition $\partial = \partial^0 + \partial^1 + \ldots + \partial^n$, where $\F(x) - \F(\partial^j(x)) = j$.  The associated spectral sequence will have pages $E_i$, denoted $E_i(\C)$ or $E^{\C}_i$, with differentials denoted $d_i$.  We will use $E^j_i$ to denote the terms living in filtration level $j$.  Recall that $H_*(E_i,d_i) \cong E_{i+1}$.  Furthermore, if a chain map $f:\C_1 \longrightarrow \C_2$ between filtered complexes does not raise the filtration levels, it is said to be {\em filtered} and induces chain maps $f_i:(E^{\C_1}_i,d^{\C_1}_i) \longrightarrow (E^{\C_2}_i,d^{\C_2}_i)$ for all $i$ (see, for example, \cite{usersguide}). 

\begin{fact}\label{filteredisoisiso} Given a filtered chain map $f: (\C_1,\F_1) \longrightarrow (\C_2,\F_2)$, for each $i \geq 0$ there exist filtrations $\F(i)$ on the mapping cone of $f$, denoted $Cone(f)$, such that $E_i(Cone(f),\F(i)) \cong Cone(f_i)$.  This tells us that over a field, the rank of $f_\infty$ is equal to the rank of $f_*$, the induced map on homology.  More generally, if some $f_i$ induces isomorphisms on the $E_i$ pages, then all subsequent $f_r$ are isomorphisms for $r \geq i$, since a bijective chain map, in this case $f_i:(E_i(\C_1),d_i) \longrightarrow (E_i(\C_2),d_i)$, always has the property that the induced map on homology, now $f_{i+1}$, is an isomorphism.  In this case, $E_{i+1}(Cone(f),\F(i))$ is acyclic and thus $Cone(f)$ is acyclic.  In particular, $f$ is a quasi-isomorphism.  We will heavily rely on this fact.
\end{fact}

\subsection{Framed Floer Homology and the $\varepsilon$-filtration}
Now we define the filtration on $\C^\circ(\hyperbox,\Lambda)$ that we would like to study.  

\begin{definition}
Let $\C$ be an $n$-dimensional hypercube of chain complexes.  The \emph{$\varepsilon$-filtration} on $\C$ is defined by 
\[
\F(x) = n - \| \varepsilon \| \text{ for } x \in \C^{\varepsilon}.
\]
The induced spectral sequence is called the {\em $\varepsilon$-spectral sequence}.  
\end{definition} 

The {\em depth} of a filtered complex is the largest difference in the filtration levels of two non-zero elements.  If $k$ is greater than the depth of the filtration, the $k$th differential in the spectral sequence, $d_k$, must vanish. Therefore, the induced spectral sequence from the $\varepsilon$-filtration has all differentials $d_k$ vanish for $k$ greater than the dimension of the hypercube.  

\begin{definition}
Define an {\em $\varepsilon$-filtered quasi-isomorphism} to be an $\varepsilon$-filtered chain map between the total complexes of two hypercubes of chain complexes which induces isomorphisms between the $E_1$ pages of the $\varepsilon$-spectral sequence.  It is necessarily a quasi-isomorphism on the total complexes by Fact~\ref{filteredisoisiso}.
\end{definition}
Given a hypercube of chain complexes, $\C$, the notation $E_i(\C)$ will always refer to the $\varepsilon$-filtration, unless noted otherwise.  In this setting, there is always a canonical isomorphism between $(E_0(\C),d_0)$ and $(\C,\partial)$, so we will often not distinguish between the two.    

\begin{definition}
Choose a complete system $\hyperbox$ for a framed link $(L,\Lambda)$ in $Y$.  The \emph{Framed Floer complex} is the $(E_1,d_1)$ complex of the $\varepsilon$-spectral sequence for $\widehat{\C}(\hyperbox,\Lambda)$.  The {\em Framed Floer homology} is the homology of the Framed Floer complex, or equivalently, the $E_2$ page of the $\varepsilon$-spectral sequence.  These are denoted by $\widehat{FFC}(\hyperbox,\Lambda)$ and $\widehat{FFH}(\hyperbox,\Lambda)$ respectively, again omitting the underlying three-manifold.  
\end{definition}

\begin{remark}
Since $\widehat{\C}(\hyperbox,\Lambda)$ is defined by choosing some $U_i$ and setting it to 0, we have to be careful here.  We need to show, similar to the link surgery theorem, the independence of this choice of $i$ to make Framed Floer homology well-defined.  This issue will be addressed in Proposition~\ref{variableactions} for $\widehat{\C}(\hyperbox,\Lambda)$ and Theorem~\ref{largesurgeries} for $\widehat{\mathfrak{A}}$-complexes.  Thus, we will suppress which $U_i$ we are setting to 0.  
\end{remark}

\begin{remark}
Due to the splitting of $\C^\circ(\hyperbox,\Lambda,\mathfrak{s})$ over Spin$^c$ structures on $Y_\Lambda(L)$, there is a splitting 
\[
\widehat{FFH}(L,\Lambda) = \bigoplus_{\mathfrak{s} \in \text{Spin}^c(Y_{\Lambda}(L))} \widehat{FFH}(L,\Lambda,\mathfrak{s}).
\]
\end{remark}

As we will be constantly working with the $\varepsilon$-spectral sequence, let's study the $E_1$ page for $\C^\circ(\hyperbox,\Lambda)$.  This is given by 
\[
E^j_1 = \bigoplus_{|M| = n-j} \; \prod_{\mathbf{s} \in \mathbb{H}(L)} H_*(\A^\circ(\hyperbox^{L-M},\psi^{M}(\mathbf{s})),\partial_{\psi^M(\mathbf{s})}),
\]
since $\| \varepsilon(M) \| = |M|$.  

Notice that with any non-empty link, the $E_1$ page will be infinitely generated.  However, there are some special cases where this will not happen when restricted to a single Spin$^c$ structure; these cases are usually well-suited for computations.  In fact, $\C^\circ(\hyperbox,\Lambda,\mathfrak{s})$ will be finitely generated if and only if $\Lambda$ is identically 0 (all pairwise linking numbers are 0 and each component has framing 0).

\begin{remark} \label{knotsvanish}
Because of the depth, if $\hyperbox$ is a complete system for a knot $K$ in $Y$, we have 
\[
\widehat{FFH}(\hyperbox,\lambda,\mathfrak{s}) = E_2(\widehat{C}(\hyperbox,\lambda,\mathfrak{s})) \cong E_{\infty}(\widehat{C}(\hyperbox,\lambda,\mathfrak{s})) \cong \widehat{HF}(Y_\lambda(K),\mathfrak{s})
\]
for all $\lambda \in \mathbb{Z}$ and $\mathfrak{s} \in \text{Spin}^c(Y_\lambda(K))$.  This relies on the fact that $\widehat{\C}(\hyperbox,\lambda,\mathfrak{s})$ is defined over a field.  The final isomorphism above does not necessarily hold for all flavors.  An example where they disagree is $FFH^-(T_L,-1)$, the $-$ flavor of the Framed Floer homology of $-1$-surgery on the left-handed trefoil in $S^3$ (this is an exercise for the reader with the integer surgery formula for knots, Theorem 1.1 of \cite{hfkz}).  For this reason, we do not study $FFH^{\pm}$.  On the other hand, we expect that $FFH^\infty(\hyperbox,\Lambda)$ is completely determined by $H_1(Y;\mathbb{Z})$ (see Lemma 5.1 of \cite{hftriple} for the case of torsion Spin$^c$ structures when $\Lambda \equiv \mathbf{0}$).  For these reasons, we restrict to the case of $\widehat{FFH}$ for this paper.
\end{remark}

\section{What is Framed Floer Homology?}\label{cobordismrephrasesection}
We now take a moment to sketch a more geometric interpretation of what Framed Floer homology really is.  For more details on the constructions in this section, see Section 10 of \cite{hflz}.  

In order to prove the link surgery theorem, Manolescu and Ozsv\'ath work heavily with a special kind of complete system called a \emph{basic system} which they construct for any oriented link $L$ (Section 6.7 in \cite{hflz}).  We will make use of them in this section and Section~\ref{kunnethsection}, but only review the properties relevant for our discussion.  Heegaard diagrams in a basic system will have exactly one $z$ basepoint on any given link component and there are the same number of $w$ basepoints as components of $L$.  The diagram $\hyperbox^L$ in a basic system is always maximally colored, so it will have exactly $\ell$ colors.  Recall that a maximally colored Heegaard diagram for $L$ always has $\ell \leq p$.  Thus, basic systems are defined over the ``smallest" base rings possible for a maximally colored diagram, $\mathbb{F}[[U_1,\ldots,U_\ell]]$.  Finally, for $M \subseteq L$, the Heegaard diagram $\hyperbox^M$ is obtained from $\hyperbox^L$ by removing $z$ basepoints.  In particular, there is a natural identification between the generators of $\mathfrak{A}^\circ(\hyperbox^L,\mathbf{s})$ and $\mathfrak{A}^\circ(\hyperbox^M,\mathbf{s'})$ for any $\mathbf{s}$ and $\mathbf{s'}$.


From the definitions, it may seem as though Framed Floer homology is an arbitrary object to study.  However, the objects involved in Framed Floer homology are actually comprised of elementary pieces of the Heegaard Floer homology package.  We now give a suitable rephrasing of Framed Floer homology in terms of these more familiar pieces.  Recall that the chain groups for $\widehat{FFC}$ are comprised of terms of the form  $H_*(\widehat{\mathfrak{A}}(\hyperbox^{L-L'},\psi^{L'}(\mathbf{s})),\partial_{\psi^{L'}(\mathbf{s})})$.    

\begin{theorem}[Theorem 10.1 in \cite{hflz}]\label{largesurgeries} 
For each $\mathbf{s} \in \mathbb{H}(L)$, for sufficiently large framings $\widetilde{\Lambda}$ such that there there exists a quasi-isomorphism of complexes of $\mathbb{F}[[U_1,\ldots,U_n]]$-modules  
\[
\mathbf{CF}^\circ(Y_{\widetilde{\Lambda}|_{L'}}(L),\psi^{L'}([\mathbf{s}])) \cong \mathfrak{A}^\circ(\hyperbox^{L-L'},\psi^{L'}(\mathbf{s})),
\]
where $\mathbf{CF}^\circ(Y_{\widetilde{\Lambda}|_{L'}}(L'),\psi^{L'}([\mathbf{s}]))$ is the Floer chain complex for a multi-pointed Heegaard diagram for $Y_{\widetilde{\Lambda}|_{L'}}(L')$ with Spin$^c$ structure obtained by restricting the Spin$^c$ structure $[\mathbf{s}]$ on $Y_{\widetilde{\Lambda}}(L)$.  Here, all $U_i$'s act the same on $\mathbf{CF}^\circ$.  
\end{theorem}
Thus, the terms in the $E_1$ page of the $\varepsilon$-spectral sequence are determined by the Heegaard Floer homologies of large surgeries on $L-L'$ in various Spin$^c$ structures.    

\begin{remark}\label{variablesacttrivially}
By Theorem~\ref{largesurgeries}, all of the $U_i$ variables must act the same on $H_*(\mathfrak{A}^\circ)$, since they do on Heegaard Floer homology.  Theorem 4.10 in \cite{hflz} tells us that the stable isomorphism-type of $H_*(\widehat{\mathfrak{A}})$ is independent of the colored Heegaard diagram chosen.  Therefore, for any complete system, we will have that all $U_i$ act the same on $H_*(\mathfrak{A}^\circ)$.  Thus, the isomorphism-type of  $H_*(\widehat{\mathfrak{A}})$ is independent of the variable we are quotienting by.  Finally, we note that $H_*(\widehat{\mathfrak{A}})$ is always finite-dimensional over $\mathbb{F}$.  
\end{remark}

Now, to study the $d_1$ differential, observe that the terms in the differential on $\mathcal{C}$ which lower the $\varepsilon$-filtration by precisely 1 are the maps $\Phi^{\pm K}_\mathbf{s}$.  Therefore, the $d_1$-differential of some $[(\mathbf{s},x)] \in H_*(\mathfrak{A}^\circ(\hyperbox^{L-L'},\psi^{L'}(\mathbf{s})),\partial_{\psi^{L'}(\mathbf{s})})$ is given by 
\[
d_1(\mathbf{s},[x]) = \sum_{\vec{K} \subseteq L-L'} (\mathbf{s} + \Lambda_{\vec{K},\vec{L}}, (\Phi^{\vec{K}}_{\psi^{L'}(\mathbf{s})})_*([x])),
\]
where we are summing over all orientations of the individual components of $L-L'$.

We would now like to rephrase the entire $(E_1,d_1)$ complex in this setting.  

Suppose that $W$ is a cobordism from $Y$ to $Y'$ equipped with a Spin$^c$ structure $\mathfrak{t}$ such that $\mathfrak{t}|_Y = \mathfrak{s}$ and $\mathfrak{t}|_{Y'} = \mathfrak{s}'$.  In \cite{hfsmooth4}, Ozsv\'ath and Szab\'o construct a map
\[
f^\circ_{W,\mathfrak{t}}:\mathbf{CF}^\circ(Y,\mathfrak{s}) \longrightarrow \mathbf{CF}^\circ(Y',\mathfrak{s}'),  
\]
which induces a map on homology, $F^\circ_{W,\mathfrak{t}}$.  
Observe that given a two-handle addition from $Y$ to $Y_n(K)$, we may flip the direction and reverse orientation of this cobordism to obtain a cobordism $W:Y_n(K) \longrightarrow Y$.  We will call this a {\em reversed 2-handle addition} cobordism.  

\begin{proposition}[Theorem 10.2 of \cite{hflz}]\label{mappingconeidentification} Let $L' \subseteq L$ and choose $\vec{K}$ an oriented component of $L'$.  Fix $\mathbf{s} \in \mathbb{H}(L')$.  For $\widetilde{\Lambda} >> 0$, there exists a Spin$^c$ structure $\mathfrak{t}_{\vec{K}}$ on the reversed 2-handle addition cobordism 
\[
W : Y_{\widetilde{\Lambda}|_{L'}}(L') \longrightarrow Y_{\widetilde{\Lambda}|_{L'-K}}(L'-K)
\] 
extending $[\mathbf{s}]$ and $\psi^{\vec{K}}([\mathbf{s}])$, such that there exists an $\varepsilon$-filtered quasi-isomorphism between the one-dimensional hypercubes of chain complexes 
\[
\xymatrix{
\mathbf{CF}^\circ(Y_{\widetilde{\Lambda}|_{L'}}(L'),[\mathbf{s}]) \ar[r]^{\hspace{-.7cm} f^\circ_{W,\mathfrak{t}_{\vec{K}}}} & \mathbf{CF}^\circ(Y_{\widetilde{\Lambda}|_{L'-K}}(L'-K),[\psi^{\vec{K}}(\mathbf{s})])
}
\]
and
\[
\xymatrix{
\mathfrak{A}^\circ(\hyperbox^{L'},\mathbf{s}) \ar[r]^{\hspace{-.6cm}\Phi^{\vec{K}}_{\mathbf{s}}} & \mathfrak{A}^\circ(\hyperbox^{L'-K},\psi^{\vec{K}(\mathbf{s})}).
}
\]
\end{proposition}

This tells us in fact that our $d_1$ differential, after the identifications of the above discussion, is given by the induced maps on Heegaard Floer homology arising from the cobordism maps obtained by a reversed 2-handle addition.  Furthermore, after appropriate truncations (see the proof of Proposition 10.10 in \cite{hflz}), Manolescu and Ozsv\'ath extend the isomorphisms of Proposition~\ref{mappingconeidentification} to be compatible in the following sense.  While there are quasi-isomorphisms between the mapping cones $\Phi^{\pm K}$ and $f_{W,\mathfrak{t}_{\pm K}}$, one requires them to extend to the mapping cones of $\Phi^{+K} + \Phi^{-K}$ and $f^\circ_{W,\mathfrak{t}_{K}} + f^\circ_{W,\mathfrak{t}_{-K}}$.  These quasi-isomorphisms are also compatible after summing over various values of $\mathbf{s}$ and different choices of $K$.  

To quickly summarize, modulo some truncations, the chain groups of the Framed Floer complex are Heegaard Floer homologies of large surgeries on sublinks and the differential is given by cobordism maps.  We also point out that Framed Floer homology can be computed by only counting bigons and triangles.  

This is in fact what one sees in the construction of the spectral sequence from the reduced Khovanov homology of a link to the Heegaard Floer homology of the double-branched cover of its mirror.  The Heegaard Floer homology is the homology of a hypercube of chain complexes, but when looking at the $E_1$ page, one has a complex determined solely by the Heegaard Floer homologies of certain surgeries whose differential is  determined by 2-handle cobordism maps.

\section{Framed Link Invariance of $\widehat{FFH}$}\label{linkinvariancesection}
In our definition of Framed Floer homology, there were essentially two choices for input.  We would like to show that the isomorphism-type is independent of these choices.  The first is the complete system that we use to construct $\C^-(\hyperbox,\Lambda)$.  The other choice is the $U_i$ variable used to form $\widehat{\C}(\hyperbox,\Lambda) = \C^-(\hyperbox,\Lambda)/U_i \cdot \C^-(\hyperbox,\Lambda)$.  We show the independence of variable choice first.  

\subsection{Independence of the Variable in Framed Floer Homology}
In order to do this, we must understand what happens in general when we take a hypercube of chain complexes and set $U_i$ to 0.  Let's work with hypercubes of chain complexes, $\C$, defined over $\mathbb{F}[[U_1,\ldots,U_n]]$.  Recall that $E^j_r(\C)$ consists of the elements in filtration level $j$.  In particular, $d_1 : E^j_r(\C) \longrightarrow E^{j-1}_r(\C)$.   Remember that the $(E_0,d_0)$-complex for $\C$ consists of the vertex complexes $(\C^\varepsilon,\partial)$.  We use the notation 
\[
D^1(x) = \sum_{\| \varepsilon \| = 1} \destab^{\varepsilon}(x), \text{ for } x \in \C,  
\]     
since these are the terms that are underlying the $d_1$ differentials.  For convenience, we will use $d_0$ and $d_1$ as the differentials in the $\varepsilon$-spectral sequence for any hypercube of chain complexes - it will be clear from the context which hypercube we will be working with.  

Note that quotienting $\C$ by the ideal generated by $U_i$ is still naturally a hypercube of chain complexes.  If $\C$ is a hypercube of chain complexes of dimension $q$, the $\varepsilon$-filtration can be extended to the mapping cone for $U_j : \C \longrightarrow \C$, denoted $Cone(U_j)$, in the obvious way such that the depth is still $q$.    

We will denote elements in the mapping cone complex as pairs $< x,x' >$ to distinguish them from the pair $(x,x')$ living in the hypercube $\C \oplus \C$.  Finally, $\tilde{x}$ means the class of the cycle $x$ in $E_1(C/U_i)$, and thus $d_1(\tilde{x}) = \widetilde{D^1(x)}$.  We instead use $[< x,x' >]$ to represent an element of $E_1(Cone(U_j))$.        

The following is similar to Lemma 10.12 in \cite{hflz}.  

\begin{lemma}\label{variableactionsgeneral}
Suppose that $\C$ is a complex of free $\mathbb{F}[[U_1,\ldots,U_n]]$-modules and fix indices $i$ and $j$ between $1$ and $n$.  Then the following hold: \\
1) there are $\varepsilon$-filtered quasi-isomorphisms between the hypercubes of $\mathbb{F}$-chain complexes $\C/(U_i,U_j)$ and $Cone(U_j:\C/(U_i) \longrightarrow \C/(U_i))$, \\
2) if $U_j$ acts trivially on $E_1(\C/(U_i))$, there is an $\mathbb{F}$-vector space isomorphism between $E_2(\C/(U_i,U_j))$ and $E_2(\C/(U_i)) \otimes_{\mathbb{F}} H^*(S^1)$.   
\end{lemma}  
\begin{proof}
We first prove $1)$.  There is a chain map $q$ from $Cone(U_j)$ to $\C/(U_i,U_j)$ given by sending the first $\C/(U_i)$ factor to 0 and applying the obvious quotient map from the second factor of $\C/(U_i)$ to $\C/(U_i,U_j)$, namely sending each copy of $\mathbb{F}[[U_1,\ldots,U_n]]/U_i$ in $\C/(U_i)$ to the quotient $\mathbb{F}[[U_1,\ldots,U_n]]/(U_i,U_j)$.  This is clearly $\varepsilon$-filtered, so it suffices to check that it induces $\mathbb{F}$-isomorphisms on the $E_1$ pages.  However, this is now a question about chain complexes in general since we just want a quasi-isomorphism on the level of the vertex complexes.  The reader can check that the map $q$ between chain complexes $Cone(U_j)$ and $C/(U_i,U_j)$ is always a quasi-isomorphism between vertex complexes.      

To prove part $2)$, it now suffices to establish the splitting 
\[
E_2(Cone(U_j)) \cong E_2(\C/(U_i)) \oplus E_2(\C/(U_i)). 
\]
We will define an explicit $\mathbb{F}$-linear bijective chain map 
\[
\Delta:(E_1(\C/(U_i)) \oplus E_1(\C/(U_i)),d_1 \oplus d_1) \longrightarrow (E_1(Cone(U_j)),d_1).  
\]
Choose an $\mathbb{F}$-basis for $E_1(\C/(U_i))$ as follows.  First, pick a basis for the kernel of $d^n_1:E^n_1(\C/(U_i)) \longrightarrow E^{n-1}_1(\C/(U_i))$, denoted $\{{\tilde{x}^-}_\alpha\}$.  Choose $\{{\tilde{x}^+}_\alpha\}$ to extend $\{{\tilde{x}^-}_\alpha\}$ to a basis for $E^n_1(\C/(U_i))$.  Now, we would like to build the basis for $E^{n-1}_1(\C/(U_i))$.  We choose a basis for the image of $d^n_1:E^n_1(\C/(U_i)) \longrightarrow E^{n-1}_1(\C/(U_i))$ by simply applying $d^n_1$ to the basis $\{{\tilde{x}^+}_\alpha\}$.  We now want to extend to a basis for $E^{n-1}_1(\C/(U_i))$.  This is done by first extending  $\{d^n_1({\tilde{x}^+}_\alpha)\}$ to a basis for the kernel of $d^{n-1}_1:E^{n-1}_1(\C/(U_i))\longrightarrow E^{n-2}_1(\C/(U_i))$.  Again, we extend to the rest of $E^{n-1}_1$.  This process can be repeated to obtain a basis for all of $E_1(\C/(U_i))$, since the kernel of $d^i_1$ contains the image of $d^{i+1}_1$.  After completing this process, the elements of the basis for $E_1(\C/(U_i))$ will be denoted by $\tilde{x}$ and $d_1(\tilde{x})$ accordingly.  

To make things particularly explicit, we choose a representative $x$ from each $\tilde{x}$ and then use $D^1(x)$ as the corresponding representative of $d_1(\tilde{x})$.  We extend linearly to obtain representatives for each element of $E_1(\C/(U_i))$.  In particular, if $\tilde{x} = 0$, we have that $x = 0$.  We will use $y$ to denote a linear combination of basis vectors of type $x$ and of type $D^1(x)$; thus, for each element of $E_1(\C/(U_i))$, we have a unique representative of the form $y$.  

Now we are ready to analyze $E_1(Cone(U_j))$.  For each representative $x$ of $\tilde{x}$ we have chosen, define $\tau(x)$ to be some choice of $x'$ such that $d_0(x') = U_j \cdot x$.  We can do this since $U_j \cdot E_1(\C/(U_i)) = 0$ by assumption, and thus $U_j \cdot x$ must be a boundary.  Furthermore, we define $\tau(D^1(x))$ to be $D^1(\tau(x))$.  Note that 
\[
d_0(D^1(\tau(x))) = D^1(d_0(\tau(x))) = D^1(U_j \cdot x) = U_j \cdot D^1(x),  
\]
so $\tau(D^1(x))$ has the desired property.  Here we are making use of the fact that $D^1$ and $d_0 = \partial$ commute in a hypercube of chain complexes.  
Again, we extend $\tau$ linearly.  It is clear by construction that the elements $< y_1,\tau(y_1) +y_2>$, where each $y_i$ is one of the chosen representatives in $E_1(\C_i)$, are cycles in $E_0(Cone(U_j))$.  Define the map $\Delta$ from $E_1(\C_i) \oplus E_1(\C_i)$ to $E_1(Cone(U_j))$ by 
\[
\Delta: (\tilde{y_1},\tilde{y_2}) \mapsto [< y_1, \tau(y_1) + y_2 >],
\] 
where we define the map via the representatives $y$ that we have chosen for the classes in $E_1(\C_i)$; it is well-defined since there is precisely one $y$ for each class.  By construction, this map is linear.  

We can also see that it is injective.  This is simply because if $[< y_1,\tau(y_1) + y_2 >]$ is a trivial class, $y_1$ must be a boundary.  Therefore, by the linearity of our constructions, $y_1$ and $\tau(y_1)$ must be 0.  In particular, $y_2$ must also be a boundary.  

For surjectivity, we would like to see that we can express every element of  $E_1(Cone(U_j))$ in the form $[< y_1,\tau(y_1) + y_2 > ]$.  Choose a cycle $< a_1,a_2 >$.  In particular, $a_1$ is a cycle, so it is homologous in $E_1(\C_i)$ to an element $y_1$, say $a_1 + y_1 = d_0(b)$.  Thus, $< a_1, a_2 >$ is homologous to $< y_1, a_2 + U_j \cdot b >$, since the difference is $d_0 < b,0 > = < d_0(b), U_j \cdot b >$.  This can be decomposed into a sum of $< y_1,\tau(y_1) >$ and $< 0, a_2 + U_j \cdot b + \tau(y_1) >$.  It suffices to prove that $a_2 + U_j \cdot b + \tau(y_1)$ is a $d_0$-cycle and thus $< 0, a_2 + U_j \cdot b + \tau(y_1) >$ will be homologous to some $< 0, y_2 >$ in $E_1(Cone(U_j))$.  However, we can chase the definitions to see that $d_0(a_2) = U_j \cdot a_1$ and $d_0(U_j \cdot b) = U_j \cdot(a_1 + y_1)$ and $d_0(\tau(y_1)) = U_j \cdot y_1$.  Clearly this sums to 0.  Therefore, $\Delta$ is an isomorphism of $\mathbb{F}$-vector spaces.  

The lemma will be complete if we can show 
\[
\Delta(d_1(\tilde{y_1},\tilde{y_2})) = d_1(\Delta(\tilde{y_1}, \tilde{y_2})),
\]   
since a bijective chain map is a quasi-isomorphism.  This is simply a matter of unpacking the definitions in the constructions. 
\begin{align*}
\Delta(d_1(\tilde{y_1},\tilde{y_2})) =& \Delta[(\widetilde{D^1(y_1)},\widetilde{D^1(y_2)})] \\=& [ < D^1(y_1), \tau(D^1(y_1)) + D^1(y_2) > ] \\
=& [ < D^1(y_1), D^1(\tau(y_1)) + D^1(y_2) > ] \\=&d_1([< y_1, \tau(y_1) + y_2 > ]) \\=& d_1(\Delta(\tilde{y_1},\tilde{y_2})). \qedhere
\end{align*} 
\end{proof}

\begin{proposition}\label{variableactions}
The isomorphism-type of Framed Floer homology is independent of the choice of variable we set to 0.  
\end{proposition}
\begin{proof}
Fix a complete system $\hyperbox$ for $L$ with $p$ colors and consider the complex $\C^-(\hyperbox,\Lambda,\mathfrak{s})$.  If $p=1$, then there is only one possible $U$ to set to 0; therefore, assume $p \geq 2$.  Fix $i$ with $1 \leq i \leq p$.  We will use $\widehat{FFH}(\hyperbox,\Lambda,\mathfrak{s},i)$ to denote the Framed Floer homology where we set $U_i$ to 0; this is the $E_2$ page of the quotient complex $\C^-(\hyperbox,\Lambda,\mathfrak{s})/(U_i)$.  
Note that $E_1(\C^-(\hyperbox,\Lambda,\mathfrak{s})/(U_i))$ is comprised of the modules  $H_*(\mathfrak{A}^-(\hyperbox^{L'},\mathbf{s})/(U_i))$.  By Remark~\ref{variablesacttrivially}, all $U_j$ act by 0 on $E_1(\C^-(\hyperbox,\Lambda,\mathfrak{s})/(U_i))$.  We may thus apply Lemma~\ref{variableactionsgeneral} to see that 
\[
E_2(\C^-(\hyperbox,\Lambda,\mathfrak{s})/(U_i,U_j)) \cong \widehat{FFH}(\hyperbox,\Lambda,\mathfrak{s},i) \otimes H^*(S^1). 
\]  
However, we could switch $i$ and $j$ in the construction and obtain 
\[
E_2(\C^-(\hyperbox,\Lambda,\mathfrak{s})/(U_i,U_j)) \cong \widehat{FFH}(\hyperbox,\Lambda,\mathfrak{s},j) \otimes H^*(S^1).
\]
Doing this construction for all $j$ shows that the Framed Floer homologies are independent of the variable we set to 0.  
\end{proof}

\subsection{Invariance of Framed Floer Homology for Complete Systems}
To complete the proof of Theorem~\ref{invariancetheorem}, we simply repeat the proof of invariance of complete systems for the link surgery formula (Theorem 7.7 of \cite{hflz}), but keeping track of the $\varepsilon$-filtration.  The reader is strongly encouraged to read the original proof of Manolescu-Ozsv\'ath if they are interested in the details of the invariance of Framed Floer homology.  One advantage of the Framed Floer homology setting is that we do not need to use vertical truncations as in their proof since this is only necessary for relating the actions of the different $U_i$ variables, which we have done in Proposition~\ref{variableactions}.  It is important to point out that we did have to use their invariance theorem to obtain Proposition~\ref{variableactions}, and thus we have subtly relied on vertical truncations for this proof. 

Before proceeding, we need another filtered algebra lemma.

\begin{lemma}\label{minusgiveshat} Suppose that $\C$ is a hypercube of chain complexes over $\mathbb{F}[[U_1,\ldots,U_n]]$.  Let $f:\C_1 \longrightarrow \C_2$ be an $\varepsilon$-filtered chain map which induces isomorphisms between $E_1(\C_1)$ and $E_1(\C_2)$.  Then, $f$ induces isomorphisms between $E_2(\C_1/(U_i))$ and $E_2(\C_2/(U_i))$ for any $i$.     
\end{lemma} 
\begin{proof}
For $j = 1,2$, consider the $\varepsilon$-filtered short exact sequences 
\[
\xymatrix{
0 \ar[r] & \C_j^- \ar[r]^{U_i} & \C_j^- \ar[r] & \C_j/(U_i) \ar[r] & 0.
}
\]
We now obtain long exact sequences on the $E_1$ pages.  The 5-lemma implies that $f$ induces an isomorphism between $E_1(\C_1/(U_i))$ and $E_1(\C_2/(U_i))$.  Since $f$ was filtered, Fact~\ref{filteredisoisiso} completes the proof.
\end{proof}

In light of Lemma~\ref{minusgiveshat}, we will work with $\C^-(\hyperbox,\Lambda)$ instead of $\widehat{\C}(\hyperbox,\Lambda)$.  Recall that stable isomorphism means an isomorphism up to tensoring with factors of $H^*(S^1)$.

\begin{proof}[Proof of Theorem~\ref{invariancetheorem}]
Fix a basic system, $\hyperbox_b$, for $L$ and a Spin$^c$ structure, $\mathfrak{s}$, on $Y_{\Lambda}(L)$.  This will be our reference system as everything else will be compared to it.  We first consider the case of a complete system, $\hyperbox$, where $\hyperbox^L$ is maximally colored.  Denote the number of $z$ basepoints on $\hyperbox^L$ by $m$ and the number of $w$ basepoints by $k$; finally, $\ell$ is the number of components of $L$.  For convenience, we are assuming that $L_1 \subseteq L$ is colored by $U_1$ in all of the complete systems that we are comparing and thus $\widehat{\C}$ will always be $\C^-/U_1 \cdot \C^-$ in this proof.  By Proposition~\ref{variableactions}, this is sufficient.   

The idea is that we can relate $\hyperbox_b$ to $\hyperbox$ by a sequence of moves on complete systems (isotoping the Heegaard diagrams, changing the coloring, etc.).  Each complete system move induces a map on the respective surgery formulas, and the composition gives an explicit map from $\C^-(\hyperbox,\Lambda,\mathfrak{s})$ to  
\[
\C^-(\hyperbox_b,\Lambda,\mathfrak{s}) \otimes H^*(S^1)^{\otimes m-\ell}.
\]
We will study the $\varepsilon$-filteredness of each move, which also induce maps on the level of $\widehat{\C}$.  

By Lemma~\ref{minusgiveshat}, it suffices to establish that each step gives a stable isomorphism between the $E_1$ pages of the surgery formulas that we are comparing.  
  
To change $\hyperbox_b$ to $\hyperbox$, we must first make it so that we have the same number of $z$ and $w$ basepoints as well as the same number of colors.  This involves two different types of stabilization moves.  We first apply $k-m$ \emph{neo-chromatic free index 0/3 stabilizations} to $\hyperbox_b$ to obtain a complete system $\overline{\hyperbox_b}$.  A free index 0/3 stabilization consists of altering each Heegaard diagram in $\hyperbox_b$ as in Figure~\ref{freestabilization}. 

\begin{figure}
\labellist
\large
\pinlabel $\Sigma$ at 67 95
\pinlabel $\Sigma$ at 510 95
\pinlabel $\alpha$ at 550 215
\pinlabel $\beta$ at 750 215
\pinlabel $w$ at 650 165
\endlabellist

\begin{center}
\includegraphics[scale=.36]{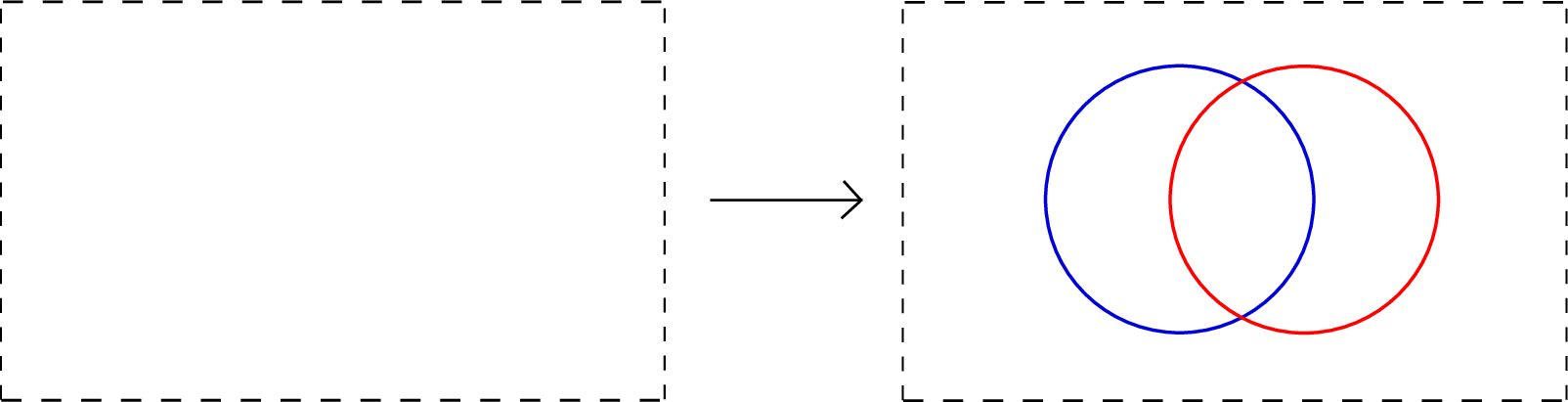}
\end{center}
\caption{A free index 0/3 stabilization} \label{freestabilization}
\end{figure}
Being neo-chromatic means that each $w$ that we are adding gets its own color.  Now there will be $\ell + k - m = p$ colors.  These are the only moves that will change the base rings which our complexes are defined over.  However, we still only have $\ell + k - m$ basepoints of type $w$ and $\ell$ basepoints of type $z$, which does not match with $\hyperbox$.    

To fix this, we do $m-\ell$ \emph{index 0/3 link stabilizations}.  An index 0/3 link stabilization consists of applying Figure~\ref{linkstabilization} to each Heegaard diagram in the complete system.  The coloring agrees with that on the original Heegaard diagrams and is determined for the new basepoints, since they lie on the same component as the $z_i$.  This move takes us from $\overline{\hyperbox_b}$ to the next complete system, $\widetilde{\hyperbox_b}$.  We will come back to study these two stabilization moves later.     

\begin{figure}
\labellist
\large
\pinlabel $\Sigma$ at 67 95
\pinlabel $\Sigma$ at 510 95
\pinlabel $\alpha$ at 550 215
\pinlabel $\beta$ at 750 215
\pinlabel $z$ at 149 165
\pinlabel $z$ at 585 165
\pinlabel $w'$ at 650 168
\pinlabel $z'$ at 716 168
\endlabellist

\begin{center}
\includegraphics[scale=.36]{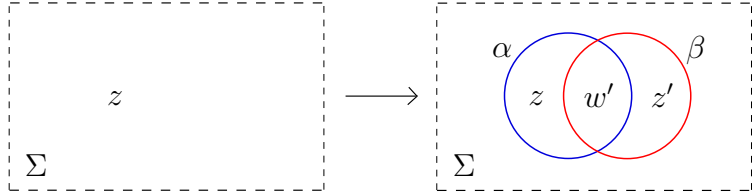}
\end{center}
\caption{An index 0/3 link stabilization} \label{linkstabilization}
\end{figure}

When we define our Floer complexes for $\widetilde{\hyperbox_b}$, these will also be defined as $\mathbb{F}[[U_1,\ldots,U_p]]$-modules, since its colorings have $p$ colors.  By Proposition 6.29 in \cite{hflz}, there is now a sequence of moves (isotopies, handleslides, global shifts, etc.) relating $\widetilde{\hyperbox_b}$ to $\hyperbox$ which will not change the colorings of the complete systems and thus not affect the base ring.  Let's study the effects of these moves.  Note that we may need to relabel the basepoints, but this is hardly a concern.  

We can consider the various standard Heegaard moves on the level of complete systems, such as surface isotopies, handleslides, and stabilizations; these induce maps of the form 
\[
\mathfrak{A}^-(\hyperbox^M, \mathbf{s}) \longrightarrow \mathfrak{A}^-(\widetilde{\hyperbox_b}^{M-M'},\psi^{\vec{M'}}(\mathbf{s})) \otimes H^*(S^1)^{\otimes m-\ell} \text{ for } \vec{M'} \subseteq M, 
\]
which implies they are $\varepsilon$-filtered.  Each move induces stable  quasi-isomorphisms on the $(E_0,d_0)$ complexes.  For example, Theorem 4.10 in \cite{hflz} shows that surface isotopies induce triangle-counting chain homotopy equivalences of complexes of $\mathbb{F}[[U_1,\ldots,U_p]]$-modules from $\mathfrak{A}^-(\hyperbox^M,\mathbf{s})$ to $\mathfrak{A}^-(\widetilde{\hyperbox_b}^M,\mathbf{s})$.  Again, Fact~\ref{filteredisoisiso} implies these give isomorphisms at all subsequent pages of the spectral sequence.  

The final set of moves necessary to relate the hyperboxes of Heegaard diagrams arise from changing the underlying hyperboxes of Heegaard diagrams, such as their size.  These are global shifts and elementary enlargements/contractions.  Lemmas 6.15 and 6.16 in \cite{hflz} give the desired quasi-isomorphisms, which again preserve the $\varepsilon$-filtration and induce quasi-isomorphisms on the associated graded level.    

By Lemma~\ref{minusgiveshat}, we have shown that $\widehat{FFH}(\hyperbox,\Lambda,\mathfrak{s})$ and $\widehat{FFH}(\widetilde{\hyperbox_b},\Lambda,\mathfrak{s})$ are isomorphic as $\mathbb{F}[[U_1,\ldots,U_p]]$-modules.  What remains is to understand the effect of the stabilizations that went from $\hyperbox_b$ to $\widetilde{\hyperbox_b}$.  We want to show that this is an isomorphism on the $E_1$ pages of $\C^-$ after tensoring with some factors of $H^*(S^1)$.  The result will then follow from Lemma~\ref{minusgiveshat}.    

First, we study the neo-chromatic free index 0/3 stabilizations from $\hyperbox_b$ to $\overline{\hyperbox_b}$.  We would like to think of $\C^-(\hyperbox_b,\Lambda,\mathfrak{s})$ as an $\mathbb{F}[[U_1,\ldots,U_p]]$-module so it can be compared to $\C^-(\overline{\hyperbox_b},\Lambda,\mathfrak{s})$.  This is done as follows.  For each $i > \ell$, in any Heegaard diagram, the basepoint $w_i$ is in the same region as some other $w_j$ after removing the $\alpha$ and $\beta$ curves that were added by the stabilization.  In this case, $U_i$ acts as $U_{\tau(w_j)}$.  Now, Proposition 6.20 and Lemma 10.12 in \cite{hflz} give us the usual filtered map which establishes an isomorphism of $\mathbb{F}[[U_1,\ldots,U_p]]$-modules
\[
E_1(\C^-(\overline{\hyperbox_b},\Lambda,\mathfrak{s})) \cong E_1(\C^-(\hyperbox_b,\Lambda,\mathfrak{s})).
\]     
For index 0/3 link stabilizations, we use the $\varepsilon$-filtered quasi-isomorphism given by Proposition 6.20 in \cite{hflz} inducing   
\[
E_1(\C^-(\widetilde{\hyperbox_b},\Lambda,\mathfrak{s})) \cong E_1(\C^-(\overline{\hyperbox_b},\Lambda,\mathfrak{s})) \otimes H^*(\mathbb{T}^{m-\ell})
\]
as $\mathbb{F}[[U_1,\ldots,U_p]]$-modules.  As a result, we have that $\widehat{FFH}(\hyperbox_b,\Lambda,\mathfrak{s})$ and $\widehat{FFH}(\hyperbox,\Lambda,\mathfrak{s})$ are stably isomorphic; this completes the proof in the case that $\hyperbox^L$ is maximally colored.    

Therefore, assume that $\hyperbox$ is a complete system such that $\hyperbox^L$ is not maximally colored.  There exists a complete system $\hyperbox'$ with $\hyperbox'^L$ maximally colored, such that $\hyperbox$ can be obtained by a sequence of \emph{elementary coloring changes} from $\hyperbox'$.  An elementary coloring change from $\tau$ ($p$ colors) to $\tau'$ ($p-1$ colors) is given by post-composing $\tau$ with a surjective map from $\{1,\ldots,p\}$ to $\{1,\ldots,p-1\}$.  Without loss of generality, we only require one elementary coloring change and that this map sends $p$ to $1$ (if $p$ is sent elsewhere, we instead set that variable to 0 for the definition of $\widehat{\C}$).  We know that on the level of complexes this elementary color change results in setting the variables $U_1$ and $U_p$ to be equal.  This can be reinterpreted as an $\varepsilon$-filtered isomorphism of hypercubes of chain complexes
\[
\C^-(\hyperbox,\Lambda,\mathfrak{s})/(U_1)  \cong \C^-(\hyperbox',\Lambda,\mathfrak{s})/(U_1,U_p). 
\]
We apply Lemma~\ref{minusgiveshat} to see that 
\[
\widehat{FFH}(\hyperbox,\Lambda,\mathfrak{s}) \cong \widehat{FFH}(\hyperbox',\Lambda,\mathfrak{s}) \otimes H^*(S^1).
\]
\end{proof}

Now, the symbols $\widehat{FFH}(L,\Lambda)$ are well-defined up to stable isomorphism for a given oriented, framed link $(L,\Lambda)$ in $Y$.  We will also use the similar notation $\C^\circ(L,\Lambda)$, even though this is only an invariant up to stable {\em quasi}-isomorphism.  Also, the default module structure on Framed Floer homology will always be $\mathbb{F}$.         

\begin{warning}
For the remainder of the paper, unless otherwise specified, complete systems will always have $\hyperbox^L$ maximally colored with exactly one $z$ and one $w$ basepoint for each component and no free $w$ (basic systems satisfy this).  Therefore, the link surgery formula gives an authentic isomorphism of $H_*(\C^\circ(\hyperbox,\Lambda))$ with $\mathbf{HF}^\circ(Y_\Lambda(L))$.  We do this so that we no longer need to keep track of stable isomorphisms (factors of $H^*(S^1)$); we can compare groups directly to see whether or not they are isomorphic to prove our theorems.  We remind the reader that the appropriate statements of the theorems can be adjusted for general complete systems by simply keeping track of the colorings and thus the number of factors of $H^*(S^1)$.  We have done everything prior in this generality so that Framed Floer homology can be calculated combinatorially, which will be discussed later.
\end{warning}

\begin{remark}[Manolescu]\label{framedfilteredinvariants}
It is interesting to point out that the Framed Floer homology groups are in no way canonically isomorphic as we had to make many basis choices in Lemma~\ref{variableactionsgeneral}.  On the other hand, Framed Floer homology does carry additional structure - the proof of Theorem~\ref{invariancetheorem} actually shows that the rank in each filtration level is another collection of invariants of $(L,\Lambda,\mathfrak{s})$.  

One can also use this additional filtered information to distinguish the Framed Floer homologies of framed links that surger to the same manifold.  It would be interesting to see whether two framed links representing homeomorphic three-manifolds can have the same total rank of Framed Floer homology but be distinguished by the filtered information.  We will neither discuss this nor gradings further, so for us, an isomorphism of Framed Floer homologies will just be determined by total rank.    
\end{remark}

\section{The K\"unneth Formula}\label{kunnethsection}
\subsection{A K\"unneth Formula for $\widehat{\mathfrak{A}}$-Complexes}
We now prove the K\"unneth formula for $\widehat{FFH}$.  As a first step, we will establish such a formula for the $\widehat{\mathfrak{A}}$-complexes.  For $i= 1,2$, let $L_i$ be a link in $Y_i$ and choose $\mathbf{s}_i$ in $\mathbb{H}(L_i)$.  Fix maximally colored Heegaard diagrams for $L_i$.  We construct $\hyperbox^{L_i}$ by performing a neo-chromatic free index 0/3 stabilization (see the proof of Theorem~\ref{invariancetheorem}), which adds the basepoint $w^i$ as well as a new variable, $U_n$ for $i=1$ and $U_m$ for $i=2$.   We will denote the Heegaard surfaces by $\Sigma_i$ and the curve sets by $\boldalpha_i$ and $\boldbeta_i$.  To form $\hyperbox^{L_1 \coprod L_2}$, we connect the region of each $\Sigma_i - \boldalpha_i -\boldbeta_i$ containing the $w^i$ by a tube and then delete $w^1$ as in Figure~\ref{kunnethtube}.  
\begin{figure}
\labellist
\small
\pinlabel $\Sigma_1$ at 30 290
\pinlabel $\Sigma_2$ at 457 290
\pinlabel $\Sigma_1\#\Sigma_2$ at 133 125
\pinlabel $\alpha_1$ at 50 210
\pinlabel $\beta_1$ at 178 210
\pinlabel $\alpha_1$ at 140 45
\pinlabel $\beta_1$ at 266 45
\pinlabel $\alpha_2$ at 405 45
\pinlabel $\beta_2$ at 537 45
\pinlabel $\alpha_2$ at 480 210
\pinlabel $\beta_2$ at 607 210
\pinlabel $w^1$ at 115 230
\pinlabel $w^2$ at 545 230
\pinlabel $w^2$ at 472 63
\endlabellist
\begin{center}
\includegraphics[scale=.45]{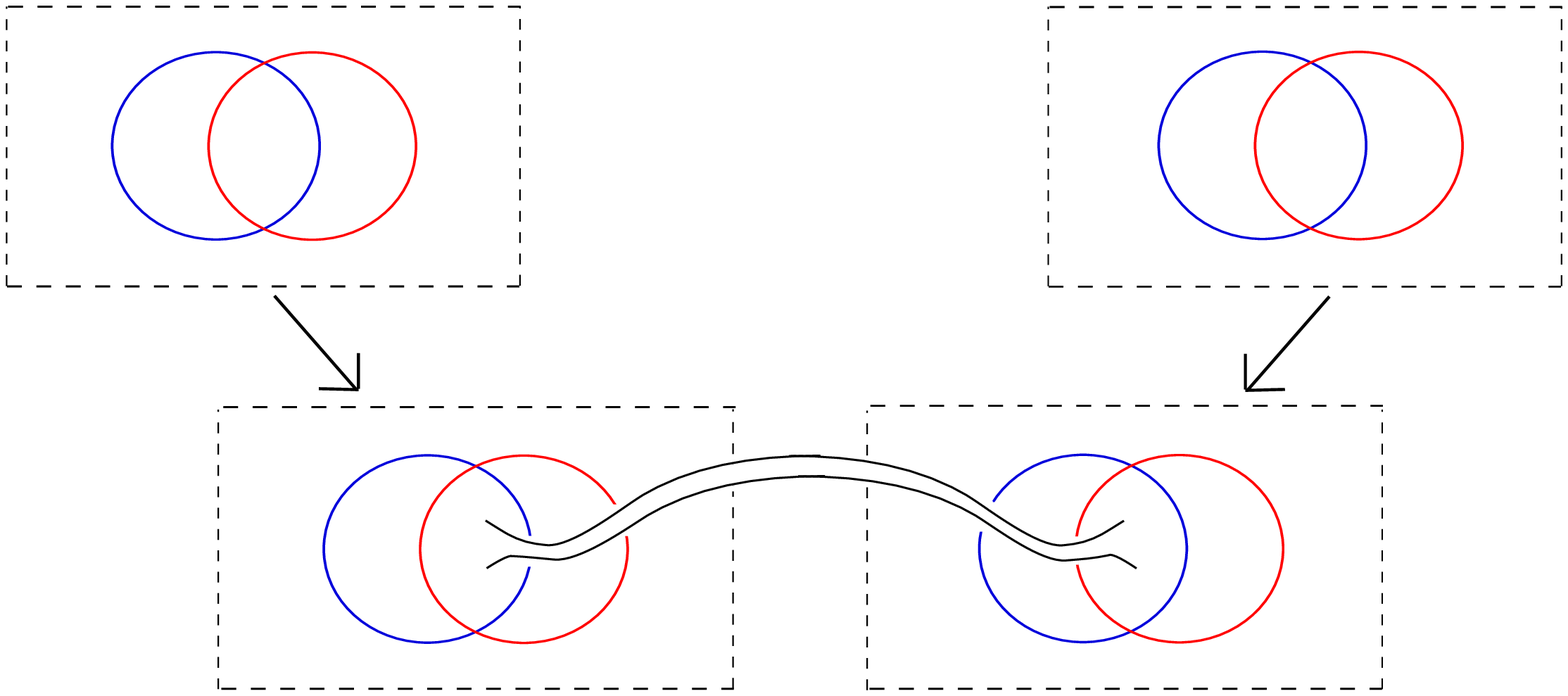}
\end{center}
\caption{Constructing the Heegaard diagram for $\hyperbox^{L_1 \coprod L_2}$} \label{kunnethtube}
\end{figure}
It is easy to see that this forms a suitable link diagram for $L_1 \coprod L_2$.  By keeping $\hyperbox^{L_1\coprod L_2}$ maximally colored, the factors of $H^*(S^1)$ will also remain unchanged.  

For $\widehat{\mathfrak{A}}(\hyperbox^{L_1},\mathbf{s}_1)$, we think of this as a complex of free-modules over $\mathbb{F}[[U_1,\ldots,U_{n-1}]]$, while $\widehat{\mathfrak{A}}(\hyperbox^{L_2},\mathbf{s}_2)$ is a complex over $\mathbb{F}[[U_{n+1},\ldots,U_{m-1}]]$; here we have chosen $U_n$ and $U_m$ to be the variables set to 0 respectively.  We set $\mathbf{s} = (\mathbf{s}_1,\mathbf{s}_2)$ in $\mathbb{H}(L_1 \coprod L_2)$.  We have that $\widehat{\mathfrak{A}}(\hyperbox^{L_1 \coprod L_2},\mathbf{s})$ is naturally a module over $\mathbb{F}[[U_1,\ldots,U_{n-1},U_{n+1},\ldots,U_{m-1}]]$ - we have kept $w^2$ colored by $m$ and set $U_m$ to 0.  We denote this ring by $\mathcal{R}$.   

\begin{lemma}\label{vertexkunneth} Under the above choices, there exists a quasi-isomorphism of $\mathcal{R}$-complexes 
\[
\widehat{\mathfrak{A}}(\hyperbox^{L_1},\mathbf{s}_1) \otimes_{\mathbb{F}} \widehat{\mathfrak{A}}(\hyperbox^{L_2},\mathbf{s}_2) \cong \widehat{\mathfrak{A}}(\hyperbox^{L_1 \coprod L_2},\mathbf{s}).
\]
\end{lemma}
\begin{proof}
We follow the argument for the proof of the K\"unneth formula for $\widehat{HF}$ (Proposition 6.1 in \cite{hfpa}).  The obvious map 
\[
\widehat{\mathfrak{A}}(\hyperbox^{L_1},\mathbf{s}_1) \otimes_{\mathbb{F}} \widehat{\mathfrak{A}}(\hyperbox^{L_2},\mathbf{s}_2) \longrightarrow \widehat{\mathfrak{A}}(\hyperbox^{L_1 \coprod L_2},\mathbf{s})  
\]
given by sending $\mathbf{x} \otimes \mathbf{y}$ to $\mathbf{xy}$ is an isomorphism of chain complexes.  The basic idea is that since we have added the free basepoint $w^2$ in all of the Heegaard diagrams we are considering, only disks with $n_{w^2}(\phi) = 0$ can contribute to the differential on $\widehat{\mathfrak{A}}(\hyperbox^{L_1},\mathbf{s}_1)$.  In other words, no projection of $\phi$ onto $\Sigma_1 \# \Sigma_2$ crosses the connect-sum region, since otherwise we will pick up a power of $U_m$, which we have set to 0.  The moduli space of holomorphic disks $\phi$ in the symmetric product of $\Sigma$ appearing in $\widehat{\mathfrak{A}}(\hyperbox^{L_1 \coprod L_2},\mathbf{s})$ that have $n_{w^2}(\phi) = 0$, which we denote $\mathcal{M}(\phi)$, is identified with $\mathcal{M}(\phi_1) \times \mathcal{M}(\phi_2)$, for $\phi_1$ and $\phi_2$ living in symmetric products of $\Sigma_1$ and $\Sigma_2$ respectively.  Since our differential only counts disks where the moduli spaces have dimension one, one of the $\mathcal{M}(\phi_i)$ must be constant.  This is exactly saying that the differential splits as $\partial_{L_1,L_2} = \partial_{L_1}\otimes id + id \otimes \partial_{L_2}$.
given by sending $\mathbf{x} \otimes \mathbf{y}$ to $\mathbf{xy}$ is an isomorphism of chain complexes.  The basic idea is that since we have added the free basepoint $w^2$ in all of the Heegaard diagrams we are considering, only disks with $n_{w^2}(\phi) = 0$ can contribute to the differential on $\widehat{\mathfrak{A}}(\hyperbox^{L_1},\mathbf{s}_1)$.  In other words, no projection of $\phi$ onto $\Sigma_1 \# \Sigma_2$ crosses the connect-sum region, since otherwise we will pick up a power of $U_m$, which we have set to 0.  The moduli space of holomorphic disks $\phi$ in the symmetric product of $\Sigma$ appearing in $\widehat{\mathfrak{A}}(\hyperbox^{L_1 \coprod L_2},\mathbf{s})$ that have $n_{w^2}(\phi) = 0$, which we denote $\mathcal{M}(\phi)$, is identified with $\mathcal{M}(\phi_1) \times \mathcal{M}(\phi_2)$, for $\phi_1$ and $\phi_2$ living in symmetric products of $\Sigma_1$ and $\Sigma_2$ respectively.  Since our differential only counts disks where the moduli spaces have dimension one, one of the $\mathcal{M}(\phi_i)$ must be constant.  This is exactly saying that the differential splits as $\partial_{L_1,L_2} = \partial_{L_1}\otimes id + id \otimes \partial_{L_2}$.
\end{proof}

\subsection{Extending the K\"unneth Formula to Hypercubes}
We let $\Lambda$ denote the sum of the two framings, $\Lambda_1$ and $\Lambda_2$, on the framed links $L_1$ and $L_2$ respectively.  

\begin{proof}[Proof of Proposition~\ref{kunnethformula}]
We would like to ``hypercube-ify" the argument of Lemma~\ref{vertexkunneth}.  First, we take basic systems for $L_1$ and $L_2$, and apply free index 0/3 stabilizations at the level of complete systems (see Section 6.8 of \cite{hflz}) as before to obtain $\hyperbox_1$ and $\hyperbox_2$ respectively.  Next, for each Heegaard diagram $\hyperbox_1^{M_1}$ and $\hyperbox_2^{M_2}$, we construct the new Heegaard diagram $\hyperbox_{1,2}^{M_1 \coprod M_2}$ by removing the basepoint in the attaching region on the $L_1$-side and connect sum the two surfaces as in Lemma~\ref{vertexkunneth}; more generally, the complete systems $\hyperbox_1$ and $\hyperbox_2$ can be used to to give a complete system $\hyperbox_{1,2}$ for $L_1 \coprod L_2$.  

By applying the maps defined on the level of $\widehat{\mathfrak{A}}$-complexes in Lemma~\ref{vertexkunneth}, we have chain complex isomorphisms
\[
\widehat{\mathfrak{A}}(\hyperbox^{L'_1}_1,\mathbf{s}_1) \otimes_{\mathbb{F}} \widehat{\mathfrak{A}}(\hyperbox^{L'_2}_2,\mathbf{s}_2) \cong 
\widehat{\mathfrak{A}}(\hyperbox^{L'_1 \coprod L'_2}_{1,2},\mathbf{s}).
\]
This gives an identification
\begin{equation}
\widehat{\C}(\hyperbox_1,\Lambda_1,\mathfrak{s}_1) \otimes_{\mathbb{F}} \widehat{\C}(\hyperbox_2,\Lambda_2,\mathfrak{s}_2) \cong \widehat{\C}(\hyperbox_{1,2},\Lambda,\mathfrak{s}_1 \# \mathfrak{s}_2).
\end{equation}
on the level of $(E_0,d_0)$ complexes, and thus on the $E_1$ page, but the identification does not a priori commute with $d_1$.  Again, we have chosen $U_n$ and $U_m$ to act trivially for $\widehat{\C}(\hyperbox_1,\Lambda,\mathfrak{s}_1)$ and $\widehat{\C}(\hyperbox_2,\Lambda,\mathfrak{s}_2)$ respectively.  

We will now try to promote this identification of $E_1$ terms to a chain complex isomorphism 
\begin{equation}\label{kunnethidentification}
(E_1(\widehat{\C}(\hyperbox_1,\Lambda_1,\mathfrak{s}_1) \otimes_{\mathbb{F}} \widehat{\C}(\hyperbox_2,\Lambda_2,\mathfrak{s}_2)),d_1) \cong (E_1(\widehat{\C}(\hyperbox_{1,2},\Lambda,\mathfrak{s}_1 \# \mathfrak{s}_2)),d_1).
\end{equation}
The first observation is that we may apply the (algebraic) K\"unneth formula to the left-hand side of (\ref{kunnethidentification}) to see that this is isomorphic to the complex 
\begin{equation}\label{kunnethlhs}
(E_1(\widehat{\C}(\hyperbox_1,\Lambda_1,\mathfrak{s}_1)),d_1) \otimes_{\mathbb{F}} (E_1(\widehat{\C}(\hyperbox_2,\Lambda_2,\mathfrak{s}_2)),d_1).
\end{equation}
Therefore, it suffices to show that the $d_1$ differential on the right-hand side of (\ref{kunnethidentification}) respects the tensorial splitting in (\ref{kunnethlhs}).  

Fix a knot $K$ in $L_1$.  We can construct $\hyperbox_{1,2}$ such that the Heegaard moves involved in the destabilization of $K$ will have the following properties.  The first is that all of the Heegaard moves on any of the Heegaard multi-diagrams on $\Sigma_1 \# \Sigma_2$ will consist of moves on the $\Sigma_1$-side and small isotopies on the $\Sigma_2$-side of $\Sigma_1 \# \Sigma_2$.  Furthermore, we can require that these moves on the $\Sigma_1$-side will correspond exactly to the moves for destabilizing $K$ in $\hyperbox_1$.  Finally, by the properties of the basic systems for $L_1$ and $L_2$ (see Section~\ref{cobordismrephrasesection}) we can ensure that all $\hyperbox^M_{1,2}$ are the same Heegaard diagram modulo basepoints.  

By the same arguments as in Lemma~\ref{vertexkunneth}, the differential on $\widehat{\C}(\hyperbox_{1,2},\Lambda,\mathfrak{s}_1 \# \mathfrak{s}_2)$ counts no polygons in Sym$^k(\Sigma_1 \# \Sigma_2)$ with projections to $\Sigma_1 \# \Sigma_2$ passing over the connect-sum regions (or else $n_{w^2}(\psi) > 0$ and we will pick up a power of $U_m = 0$ from this free variable).  Also, if the isotopies on the $\Sigma_2$-side are small enough, there is a nearest-point map (this takes intersection points to the closest neighboring intersection point in the isotoped diagram) which will induce the identity on the $\Sigma_2$-side of the splitting of the $E_1$ terms in (\ref{kunnethidentification}) - this is a combination of the fact that the nearest-point map is a chain complex isomorphism (Lemma 6.2 of \cite{hflz}) and the fact that $\hyperbox_{1,2}^{M}$ and $\hyperbox_{1,2}^{M-K}$ have the same intersection points.  

Since on the $\Sigma_1$-side we are doing the same Heegaard moves we would for destabilizing $K$ in $\hyperbox_1$, we may apply the above discussion to see that the map $(\Phi^{\vec{K}})_*$ on $E_1(\widehat{\C}(\hyperbox_{1,2},\Lambda,\mathfrak{s}_1 \# \mathfrak{s}_2))$ corresponds to $(\Phi^{\vec{K}})_* \otimes id$ on $E_1(\widehat{\C}(\hyperbox_1,\Lambda_1,\mathfrak{s}_1)) \otimes_{\mathbb{F}} E_1(\widehat{\C}(\hyperbox_2,\Lambda_2,\mathfrak{s}_2))$.  The same argument applies for destabilizing components in $L_2$.  As the $d_1$ differential is comprised of the various $(\Phi^{\vec{K}})_*$, the complex $(E_1(\widehat{\C}(\hyperbox_{1,2},\Lambda,\mathfrak{s}_1 \# \mathfrak{s}_2)),d_1)$ is exactly the tensor complex as expected.  This completes the proof.  
\end{proof}

\section{Relating the Surgery Formula to the Invariance of Heegaard Floer Homology}\label{stabilizationsection}

While it is known that Heegaard Floer homology is a three-manifold invariant, it would be insightful to construct a new proof of this result via the link surgery formula.  In other words, we `define' the Heegaard Floer homology of $Y = S^3_{\Lambda}(L)$ to be $H_*(\C^\circ(\hyperbox,\Lambda))$, where $\hyperbox$ is a complete system for $L$.  If $\hyperbox_1$ and $\hyperbox_2$ are complete systems for oriented, framed links $(L_1,\Lambda_1)$ and $(L_2,\Lambda_2)$ representing homeomorphic three-manifolds, then one would need to prove that  $H_*(\C(\hyperbox_1,\Lambda_1)) \cong H_*(\C(\hyperbox_2,\Lambda_2))$.  Why try to do this?

\begin{theorem}[Manolescu-Ozsv\'ath-Thurston \cite{hflzcombo}]\label{combinatorialhf}
Fix a framed, oriented link $(L,\Lambda)$.  There is a complete system of hyperboxes, $\hyperbox$, for $L$ such that the hypercube of chain complexes $\C^\circ(\hyperbox,\Lambda)$ can be computed completely combinatorially for any $\Lambda$.  
\end{theorem}

If such a proof of the three-manifold invariance of Heegaard Floer homology could be made to be compatible with the constructions of Theorem~\ref{combinatorialhf}, then one could give a combinatorial proof of the invariance of Heegaard Floer homology.    

Once invariance under the choice of complete systems is shown (this covers link isotopies), one only needs to be able to construct an isomorphism if the oriented, framed links $(L_1,\Lambda_1)$ and $(L_2,\Lambda_2)$ are obtained by a sequence of handleslides, (de)stabilizations, and orientation-reversals of the components.  Recall that $(L,\Lambda)$ is a stabilization of $(L',\Lambda')$ if $(L,\Lambda)$ is obtained by adding a geometrically-split $\pm 1$-framed unknot to $(L',\Lambda')$.  

As we will see in Section~\ref{noninvariancesection}, Framed Floer homology is not a three-manifold invariant.  Therefore, any explicit quasi-isomorphism between the surgery formulas for $(L,\Lambda)$ and $(L',\Lambda')$ cannot preserve the $\varepsilon$-filtration and induce isomorphisms between the $\mathfrak{A}$-complexes simultaneously.  This indicates that a proof of three-manifold invariance will have to be somewhat unnatural in its construction.  While we don't currently have a complete proof of three-manifold invariance for Heegaard Floer homology from the link surgery formula, we will give a proof of stabilization-invariance. 

\begin{remark}
In the proof of Proposition~\ref{stabilizationinvariance}, we will choose a special complete system for $(L,\Lambda)$ given one for $(L',\Lambda')$.  The current proof that the homology of the link surgery formula is independent of the choice of complete system is actually implicitly using the assumption that Heegaard Floer homology is a three-manifold invariant (see the proof of Theorem 7.7 in \cite{hflz}).  Hopefully this problem can be dealt with when restricting to the combinatorial setting in Theorem~\ref{combinatorialhf}. 
\end{remark}

\begin{proof}[Proof of Proposition~\ref{stabilizationinvariance}]
Since the framings change by $\Lambda = \Lambda' \oplus \langle \pm 1 \rangle$, the identification between $\mathbb{H}(L)/H(L,\Lambda)$ and $\mathbb{H}(L')/H(L',\Lambda')$ is clear.    

We first do the case of a $+1$-stabilization.  $L'$ will have $n-1$ components.  $L$ will be $L' \coprod U$, ordered to agree with the original ordering on $L'$ and have $U$ be the unknotted $n$th component.  We will therefore denote elements of $\mathbb{H}(L)$ by $(\mathbf{s},s_n)$.

Fix a basic system $\widetilde{\hyperbox}$ for $L'$.  By utilizing the construction in Figure~\ref{stabilizedpicture} we can use $\widetilde{\hyperbox}$ to create a basic system $\hyperbox$ for $L$, similar to performing a free index 0/3 stabilization.  Let $\hyperbox' = \hyperbox|_{L'}$, as discussed in Section~\ref{surgerytheoremstatement}.  
\begin{figure}
\labellist
\large
\pinlabel $\Sigma$ at 67 95
\pinlabel $\Sigma$ at 510 95
\pinlabel $\alpha$ at 550 215
\pinlabel $\beta$ at 750 215
\pinlabel $w_n$ at 650 190
\pinlabel $z_n$ at 650 145
\endlabellist

\begin{center}
\includegraphics[scale=.36]{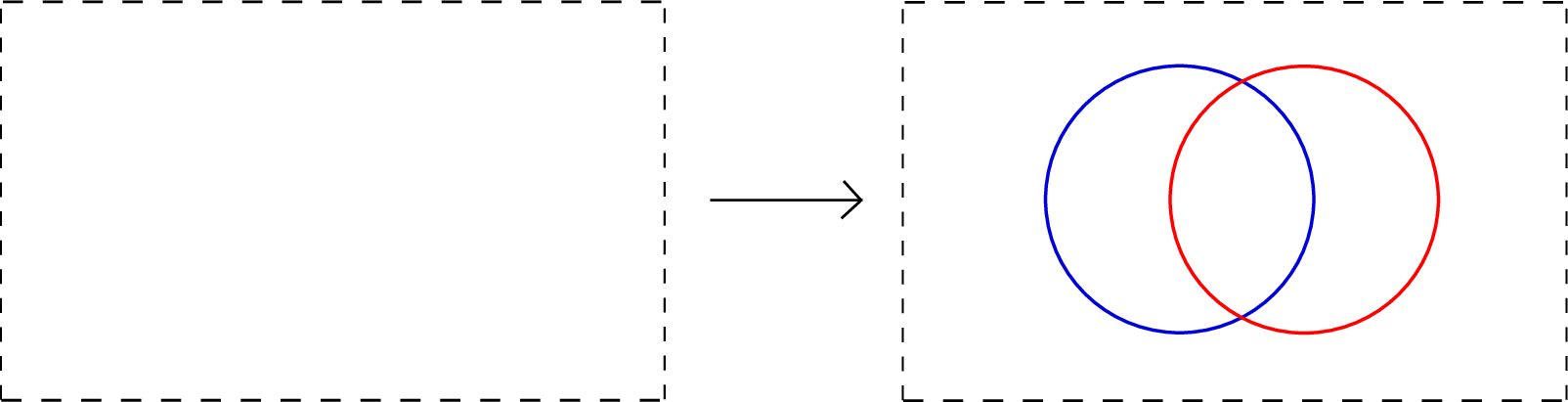}
\end{center}
\caption{Creating the unknot $U$ in the complete system}\label{stabilizedpicture}
\end{figure}

From these diagrams, one must always have that $A_{n}(\mathbf{x}) = A_{n}(\mathbf{y}) = 0$ for all $\mathbf{x}, \mathbf{y} \in \mathbb{T}_{\alpha^\varepsilon} \cap \mathbb{T}_{\beta^\varepsilon}$, since Whitney disks pass over $z_{n}$ if and only if they pass over $w_{n}$.  Therefore, if $s_{n} \geq 0$, then $\Phi^{+U}$ is a quasi-isomorphism between the $\mathfrak{A}^\circ$-complexes; similarly if $s_{n} \leq 0$, then $\Phi^{-U}$ is a quasi-isomorphism.  This is because for these $s_{n}$ we have that $\incl^{\pm U}$ is the identity, and thus $\Phi^{\pm U} = \destab^{\pm U}$, which is always a quasi-isomorphism.  We will use this observation to truncate nearly all of the complex $\C^\circ(\hyperbox,\Lambda,\mathfrak{s})$.      

We will compress our complexes in the $L'$-direction notationally by 
\[
\Link^i_{s_{n}} = \bigoplus_{\varepsilon_{n} = i} \prod_{\{\mathbf{s} \in \mathbb{H}(L'),[\mathbf{s}] = \mathfrak{s}\}} \varepsilon{(\mathbf{s},s_{n})}
\]
and
\[
\Gamma^{\pm U} = \sum_{\vec{N} \subseteq L} \Phi^{\pm U \cup \vec{N}}.
\]

Therefore, we can imagine a compressed picture for $\C^\circ(\hyperbox,\Lambda,\mathfrak{s})$ as 
\[
\xymatrix{
\ldots\!\!\!\!\!\!\!\!\!\!\!\!\!\!\!\!\!\!\!\!\!\! & \Link^0_{-1} \ar[d]^{\Gamma^{+ U}} \ar[dr]^{\Gamma^{- U}} & \Link^0_{0} \ar[d]^{\Gamma^{+ U}} \ar[dr]^{\Gamma^{- U}} & \Link^0_{1} \ar[d]^{\Gamma^{+ U}} \ar[dr]^{\Gamma^{- U}} & \qquad\ldots \\
\ldots \!\!\!\!\!\!\!\!\!\!\!\!\!\!\!\!\!\!\!\!\!\!& \Link^1_{-1} & \Link^1_{0} & \Link^1_{1} & \qquad\ldots
}
\]

Now, we perform what is called {\em horizontal truncation} (see Section 8.3 of \cite{hflz}) and see what remains after the dust settles.  This is a standard argument which is very useful for calculations with the link surgery formula.  

Consider the subcomplex, $\Mink_>$, defined by all $\Link^i_{s_{n}}$ with $s_{n} \geq 0$.  Define a filtration on $\Mink_>$ by 
\[
\F_>(x) = -s_{n} - \sum_{i \neq n} \varepsilon_i \text{ for } x \in \varepsilon_{(\mathbf{s},s_n)}.
\]  
While this filtration is not bounded below, this will not cause a problem in our setting.  The components of the differential that do not lower the filtration are $\Phi^{+U}$ and the vertex-differentials $\partial$.  The associated graded splits as a product of complexes 
\[
\xymatrix{
(\varepsilon_1\cdots\varepsilon_{n-1}0_{(\mathbf{s},s_{n})},\partial) \ar[r]^{\Phi^{+U}} & (\varepsilon_1\cdots\varepsilon_{n-1}1_{(\mathbf{s},s_{n})},\partial).
}
\]
Since $s_{n} \geq 0$, $\Phi^{+U}$ is a quasi-isomorphism and thus each of these complexes is acyclic.  Therefore, all of $\Mink_>$ is acyclic by Fact~\ref{filteredisoisiso}.  

Now, we consider the subcomplex, $\Mink_<$, generated by 
\[
\Link^0_{s_{n}} \text{ with } s_{n} < 0 \text{ and } \Link^1_{s_{n}} \text{ with } s_{n} \leq 0.
\]

We claim that this is acyclic as well.  Construct the filtration $\F_<(x) = s_{n} - \sum \varepsilon_i$ on $\Mink_<$.  This time the differentials preserving the filtration levels are $\partial$ and $\Phi^{-U}$.  Therefore, the associated graded splits as products
\[
\xymatrix{
(\varepsilon_1\cdots\varepsilon_{n-1}0_{(\mathbf{s},s_{n})},\partial) \ar[r]^{\Phi^{-U}} & (\varepsilon_1\cdots\varepsilon_{n-1}1_{(\mathbf{s},s_{n}+1)},\partial).
}
\]
Since $s_{n} \leq 0$, $\Phi^{-U}$ is a quasi-isomorphism and thus these complexes are acyclic.  Thus, $\Mink_<$ is acyclic.  Therefore, $\C^\circ(\hyperbox,\Lambda,\mathfrak{s})$ is quasi-isomorphic to the only piece that is left, $\Link^0_0$.  Thus, we just want to understand what this complex is.  Since $s_{n} = 0$, 
\[
\xymatrix{
(\varepsilon_1\cdots\varepsilon_{n-1}0_{(\mathbf{s},0)},\partial) \ar[r]^{\Phi^{+U}} & (\varepsilon_1\cdots\varepsilon_{n-1}1_{(\mathbf{s},0)},\partial)
}
\]
is a quasi-isomorphism.  Therefore, $\Gamma^{+U}$ induces a quasi-isomorphism from $\Link^0_0$ to $\Link^1_0$.  However, by construction we may apply Remark~\ref{restrictedsystems} to see that $\Link^1_0$ is exactly $\C^\circ(\hyperbox',\Lambda',\mathfrak{s})$.  We conclude that $\C^\circ(\hyperbox,\Lambda,\mathfrak{s})$ and $\C^\circ(\hyperbox',\Lambda',\mathfrak{s})$ are quasi-isomorphic.  
 
The proof for $-1$-stabilization now goes the same way except for a small change.  Instead of eliminating the two acyclic subcomplexes, we use two acyclic quotient-complexes.  The first consists of all $\Link^i_{s_{n}}$ with $s_{n} > 0$.  The other consists of $\Link^1_{s_{n}}$ with $s_{n} < 0$ and $\Link^0_{s_{n}}$ with $s_{n} \leq 0$.  By the same arguments these will be acyclic and we are left with $\Link^1_0$ instead of $\Link^0_0$.  This is $\C^\circ(\hyperbox',\Lambda',\mathfrak{s})$, completing the proof.      

Note that we can make the same truncations for $\widehat{FFC}(\hyperbox,\Lambda,\mathfrak{s})$ as for the complex $\widehat{\C}(\hyperbox,\Lambda,\mathfrak{s})$ itself, since the quasi-isomorphisms we were using to truncate were of the form $\Phi^{\pm U}$ which lower the $\varepsilon$-filtration by 1 (and thus are what we consider in the $d_1$ differential).  Therefore, since all of the explicit maps and identifications used in the argument are filtered and are quasi-isomorphisms on the $(E_1,d_1)$ pages, we see that $\widehat{FFH}(L,\Lambda,\mathfrak{s})$ and $\widehat{FFH}(L',\Lambda',\mathfrak{s})$ are also isomorphic.  
\end{proof}

\begin{remark}
Manolescu and Ozsv\'ath use horizontal truncation to show that for any oriented, framed link $(L,\Lambda)$ and $\mathfrak{s} \in \text{Spin}^c(Y)$, the complex $\C^-(\hyperbox,\Lambda,\mathfrak{s})$ is quasi-isomorphic to a complex which is finite-dimensional over $\mathbb{F}[[U_1,\ldots,U_n]]$ (Section 8.3 in \cite{hflz}).  A corollary of their construction is that the Framed Floer homology of any framed link in each Spin$^c$ structure is finite-dimensional over $\mathbb{F}$, analogous to $\widehat{HF}$.
\end{remark}

\begin{remark}
It seems likely that an effective way to prove the invariance of Heegaard Floer homology is to prove invariance of the link surgery formula under Fenn-Rourke moves \cite{FennRourke} (or some other set of local moves) instead of the standard Kirby moves.  With some extra work, one can prove invariance under the first non-trivial Fenn-Rourke move shown in Figure~\ref{onestrand}.  The details of the argument still prove the invariance of this move for Framed Floer homology as well, so the twisting seen in the other Fenn-Rourke moves is what will affect the Framed Floer homology.  
\end{remark}

\begin{figure}
\labellist
\large
\pinlabel $n\pm1$ at 430 250
\pinlabel $n$ at 172 250
\pinlabel $\pm1$ at 55 85
\pinlabel $\pm1$ at 425 85
\endlabellist
\begin{center}
\includegraphics[scale=.45]{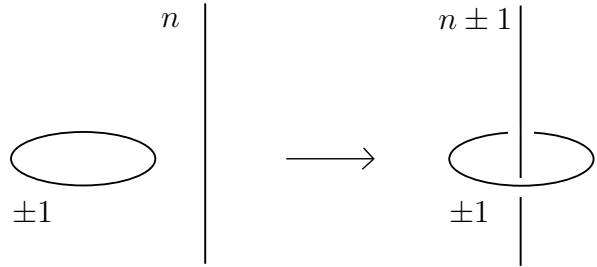}
\end{center}
\caption{The first non-trivial Fenn-Rourke move}\label{onestrand}
\end{figure} 

\begin{remark}
Note that it is easy to use the K\"unneth formula to prove stabilization invariance for the hat flavor.  The problem with doing this for other flavors is that in the proof of the K\"unneth formula for other flavors, one implicitly makes use of the fact that the Heegaard Floer homology of $S^2 \times S^1$ is independent of the choice of Heegaard diagram.  The trouble of avoiding this concern is more than the effort of using horizontal truncations in the proof of Proposition~\ref{stabilizationinvariance}.  
\end{remark}


\section{Three-Manifold Non-invariance}\label{noninvariancesection}

We now prove Theorem~\ref{noninvariancetheorem}.  In light of Proposition~\ref{stabilizationinvariance}, we will try to distinguish Framed Floer homology via handleslides.  It would be interesting to see a geometric explanation of why the higher polygon maps are required for the three-manifold invariance of Heegaard Floer homology.  We remark here that the choice of orientation of our links will not affect any of the computations in this section and in Section~\ref{mirrorsection} so we will not make mention of this choice.  

Consider the three-manifold $M=S^3_0(T_L) \# S^3_0(T_R)$, where $T_L$ and $T_R$ are the left-handed and right-handed trefoils respectively.  We would like two surgery presentations of $M$ as framed links in $S^3$ with different Framed Floer homologies.
First, there is the obvious presentation, $(L_1,\Lambda) = (T_L \coprod T_R, \mathbf{0})$.  Handlesliding $T_L$ over $T_R$ gives the new framed link $(L_2,\Lambda)$.  The components of $L_2$ are $T_L \# T_R$ and $T_R$.  The framing has not changed since we handeslid over a 0-framed, unlinked component.  Sometimes we will refer to these framings by $(n_1,n_2)$, where these are the surgery coefficients of the two components.  The proof of Theorem~\ref{noninvariancetheorem} is divided up over the next three subsections.  Recall that we have set things up so that there are no factors of $H^*(S^1)$ to keep track of.      

Fix basic systems for $L_1$ and $L_2$, $\hyperbox_1$ and $\hyperbox_2$ respectively.  Recall that $[\varepsilon_{\mathbf{s}}]$ is our notation for  $H_*(\varepsilon_{\mathbf{s}},\partial)$, which are the terms that make up the $E_1$ page of the $\varepsilon$-spectral sequence.  It will be clear from the context which complete system $\varepsilon_{\mathbf{s}}$ is coming from, where we allow $\hyperbox_1$, $\hyperbox_2$, or the complete system obtained by restricting some $\hyperbox_i$ to a single component.  By Theorem~\ref{largesurgeries}, for fixed $\varepsilon$ and $\mathbf{s}$, we have that $[\varepsilon_{\mathbf{s}}]$ is independent of the complete system for a given link.  Finally, for a knot $K$ we use the notation 
\[
\Psi^K_{\mathbf{s}} = (\Phi^{+K}_{\mathbf{s}})_*+(\Phi^{-K}_{\mathbf{s}})_*.
\] 

\begin{remark}\label{noninvariancerestrictedsystems} We may apply Remark~\ref{restrictedsystems} to see that studying the appropriate face of the surgery formula for $L$ really corresponds to studying sugery on a sublink.  In the case that $\Lambda$ is identically 0, the complexes in Remark~\ref{restrictedsystems} are especially simple.  In particular, if $L$ has only two components, say $K_1$ and $K_2$, and $\hyperbox$ is a complete system for $L$, then we have an $\varepsilon$-filtered chain complex isomorphism between the one-dimensional hypercubes of chain complexes
\[
\xymatrix{
([01_{(s_1,s_2)}] \ar[r]^{\Psi^{K_1}_{(s_1,s_2)}} & [11_{(s_1,s_2)}] )
}
\]
and 
\[
\xymatrix{ 
([0_{s_1}] \ar[r]^{\Psi^{K_1}_{s_1}} & [1_{s_1}]), 
}
\]
where $01_{(s_1,s_2)}$ and $11_{(s_1,s_2)}$ sit inside the complex  $\widehat{\C}(\hyperbox,\Lambda)$ and $0_{s_1}$ and $1_{s_1}$ sit inside $\widehat{\C}(\hyperbox|_{K_1},\Lambda|_{K_1})$.  A similar statement holds for $K_2$.  
\end{remark}

\subsection{Background Knot Calculations}\label{knotsurgerycalc}
As discussed, we would like to analyze the details of the surgery formula for each of the individual knots appearing as components of one of the $L_i$ - these are $T_L$, $T_R$, and $T_L \# T_R$.  For each component $K$, we will understand in detail the complexes 
\[
\xymatrix{
([0_s] \ar[r]^{\!\!\!\!\!\!\!\!\!\!\!\!\!\!\!\!\!\!\!\!\!\!\!\!\Psi_s^{K}} & [1_s]) = (H_*(\widehat{\mathfrak{A}}(\hyperbox_i^K,s)) \ar[r]^{\quad\quad\quad\Psi_s^{K}} & \widehat{\hfbold}(S^3)), 
}
\]
which calculate $\widehat{\hfbold}(S^3_0(K),s)$ for each $s \in \mathbb{Z} \cong \text{Spin}^c(S^3_0(K))$.  

We first need some background from \cite{hfalternating}.  Recall that a knot is {\em $HFK$-thin} if the value of the difference between the Alexander grading and the Maslov grading of each homological generator is the same.  Let $K$ have signature $\sigma$ and Alexander polynomial
\[
\Delta_K(T) = a_0 + \sum_{s>0} a_s(T^s + T^{-s}).
\]
Then, we set
\[
t_s(K) = \sum_{j=1}^\infty j a_{|s| + j} 
\text{ and }
\delta(\sigma,s) = 0 \vee \lceil \frac{|\sigma| - 2|s|}{4} \rceil.
\]
Note that the sum is always finite.  Finally, define $b_s$ for each $s \in \mathbb{Z}$ by 
\[
(-1)^{s+\frac{\sigma}{2}}b_s = \delta(\sigma,s) - t_s(K).
\]
We now are ready to give the hat-version of a result of Ozsv\'ath and Szab\'o.

\begin{theorem}[cf. Theorem 1.4 of \cite{hfalternating}] \label{thincalculation}
Suppose $K$ is an $HFK$-thin knot with $\sigma \leq 0$.  Then, for $s \in \mathbb{Z}$  
\begin{equation*}
\rk \widehat{HF}(S^3_0(K),s)  = \left\{
\begin{array}{rl}
2b_s + 2& \text{ if } s=0 \text{ or } \delta(\sigma,s) \neq 0\\
2b_s & \text{ if } s \neq 0 \text{ and } \delta(\sigma,s) = 0
\end{array} \right. 
\end{equation*}
and $\rk H_*(\widehat{\mathfrak{A}}(\hyperbox_i^K,s))  = 2b_s + 1$.
\end{theorem}

Theorem 1.3 of \cite{hfalternating} proves that alternating knots are $HFK$-thin.  Therefore, $T_L$ and $T_R$ are $HFK$-thin.  Finally, by the K\"unneth formula for knot Floer homology (Corollary 7.2 of \cite{hfk}), $T_L \# T_R$ is also thin.  

\begin{remark}\label{calculatepsi}
Because $\widehat{HF}(S^3_0(K),s)$ is calculated by the homology of the complex
\[
\xymatrix{
([0_s] \ar[r]^{\!\!\!\!\!\!\!\!\!\!\!\!\!\!\!\!\!\!\!\!\!\!\!\!\Psi_s^{K}} & [1_s]) = (H_*(\widehat{\mathfrak{A}}(\hyperbox_i^K,s)) \ar[r]^{\quad\quad\quad\Psi_s^{K}} & \widehat{\hfbold}(S^3)), 
}
\]
and the rank of $\widehat{HF}(S^3)$ is 1, we know that $\Psi_s^K$ will have rank 1 or 0 depending on whether the rank of $H_*(\widehat{\mathfrak{A}}(\hyperbox_i^K,s))$ is greater than or less than the rank of $\widehat{HF}(S^3_0(K),s)$.  
\end{remark}

In particular, for knots satisfying the hypotheses of Theorem~\ref{thincalculation}, 
\begin{equation*}
\rk \Psi_s^K = \left\{
\begin{array}{rl}
0 & \text{ if } s = 0 \text{ or } \delta(\sigma,s) \neq 0 \\
1 & \text{ if } s \neq 0 \text{ and } \delta(\sigma,s) = 0
\end{array}. \right.
\end{equation*}

We are now ready to calculate $H_*(\widehat{\mathfrak{A}}(\hyperbox_i^K,s))$ and the rank of $\Psi_s^K$ for the components of $L_1$ and $L_2$.      

\subsubsection{The Right-Handed Trefoil}\label{rht}
The Alexander polynomial for the right-handed trefoil is given by 
\[
\Delta_{T_R}(T) = T - 1 + T^{-1}.  
\]
Furthermore, $T_R$ has signature -2.  Applying Theorem~\ref{thincalculation}, we arrive at 
\begin{equation}\label{rhtcalc}
H_*(\widehat{\mathfrak{A}}(\hyperbox_i^{T_R},s)) \cong \mathbb{F} \text{ for all } s \;, \quad  
\rk \Psi_s^{T_R} = \left\{
\begin{array}{rl}
0 & \text{ if } s = 0 \\
1 & \text{ if } s \neq 0
\end{array} \right.
\end{equation}

\subsubsection{The Left-Handed Trefoil}\label{lht}
Since the signature of $T_L$ is 2, we cannot apply Theorem~\ref{thincalculation} for $T_L$.  We must use another approach.  

The adjunction inequality (Theorem 7.1 of \cite{hfpa}) implies that $\widehat{HF}(S^3_0(K),s) = 0$ for $|s| \geq g(K)$, where $g(K)$ is the Seifert genus of the knot $K$.  Therefore, for $|s| \geq g(K)$, both $H_*(\mathfrak{A}(\hyperbox^K,s))$ and $\Psi_s^K$ must have rank 1.  If it so happens that $H_*(\mathfrak{A}(\hyperbox^K,s))$ has rank 1 for {\em all} $s$, then $K$ must admit a positively-framed L-space surgery (see, for example, \cite{jhomlspace}).  However, $T_L$ is genus one, but does not admit a positively-framed L-space surgery (see Section 8.1 of \cite{absgraded}).
Finally, since $S^3_0(T_L)$ and $S^3_0(T_R)$ are homeomorphic (orientation reversing), we know that the rank of $\widehat{HF}(S^3_0(T_L),0)$ is two by Theorem~\ref{thincalculation} and (\ref{rhtcalc}).    

Putting together all of this information and again applying Remark~\ref{calculatepsi} we see that  
\begin{equation}\label{lhtcalc}
\rk H_*(\widehat{\mathfrak{A}}(\hyperbox_1^{T_L},s)) = \left\{
\begin{array}{rl}
3 & \text{ if } s = 0 \\
1 & \text{ if } s \neq 0
\end{array} \right.
, \quad \rk \Psi_s^{T_L} = 1 \text{ for all } s.
\end{equation}

\subsubsection{The Square Knot}
Finally, we need the calculation for $T_L \# T_R$.  This has vanishing signature and Alexander polynomial
\[
\Delta_{T_L \# T_R}(T) = T^2 - 2T + 3 - 2T^{-1} + T^{-2}.
\]
Again, applying Theorem~\ref{thincalculation} gives
\begin{equation}\label{squarecalc}
\rk H_*(\widehat{\mathfrak{A}}(\hyperbox_2^{T_L \# T_R},s)) = \left\{
\begin{array}{rl}
3 \text{ if } |s| = 1 \\
1 \text{ if } |s| \neq 1
\end{array} \right.
, \quad 
\rk \Psi_s^{T_L \# T_R} = \left\{
\begin{array}{rl}
0 \text{ if } s = 0 \\
1 \text{ if } s \neq 0
\end{array}. \right.
\end{equation}

We can now use the computations of this subsection to calculate $\widehat{HF}(M,s)$.  Recall that $M$ was defined as $S^3_0(T_L) \# S^3_0(T_R)$; the K\"unneth formula for Heegaard Floer homology (Theorem 1.5 of \cite{hfpa}) combined with Theorem~\ref{thincalculation}, (\ref{rhtcalc}), and (\ref{lhtcalc}) shows that
\begin{equation*}
\rk \widehat{\hfbold}(M,\mathfrak{s}) = \left\{
\begin{array}{rl}
4 & \text{ if } \mathfrak{s} = 0 \\
0 & \text{ if } \mathfrak{s} \neq 0
\end{array}. \right.
\end{equation*}

\subsection{The Framed Floer Homology of $(L_1,\Lambda)$}
\begin{proposition}\label{L1calc}
For all $\mathfrak{s}$, $\widehat{FFH}(L_1,\Lambda,\mathfrak{s}) \cong \widehat{HF}(M,\mathfrak{s})$.
\end{proposition}
\begin{proof}
We apply the Framed Floer homology K\"unneth formula (Proposition~\ref{kunnethformula}), Remark~\ref{knotsvanish}, and the Heegaard Floer homology K\"unneth formula to see that 
\begin{align*}
\widehat{FFH}(L_1,\Lambda,\mathfrak{s}_1 \# \mathfrak{s}_2) & \cong \widehat{FFH}(T_L,0,\mathfrak{s}_1) \otimes \widehat{FFH}(T_R,0,\mathfrak{s}_2) \\
& \cong \widehat{HF}(S^3_0(T_L),\mathfrak{s}_1) \otimes \widehat{HF}(S^3_0(T_R),\mathfrak{s}_2) \\
& \cong \widehat{HF}(M,\mathfrak{s}_1 \# \mathfrak{s}_2). \qedhere
\end{align*}
\end{proof}

\subsection{The Framed Floer Homology of $(L_2,\Lambda)$}
Here, we will see that there is exactly one Spin$^c$ structure where $\widehat{HF}(M,\mathfrak{s})$ differs from $\widehat{FFH}(L_2,\Lambda,\mathfrak{s})$.  

The key thing to observe is that for any $n$, the framed link $(L_1,(n,0))$ can still be handleslid to obtain $(L_2,(n,0))$; this only requires that the framing on $T_R$ remain fixed at 0 to still obtain $L_2$ upon handlesliding.  Let's first study $\widehat{HF}(S^3_{(n,0)}(L_2))$ for sufficiently large $n$.  Given $\mathbf{s} \in \mathbb{H}(L)$ we will use $[\mathbf{s}]$ to represent Spin$^c$ structures on various framed surgeries on $L$.  The context will always be clear which surgery these Spin$^c$ structures are living in.

\begin{lemma}\label{largetruncs} Let $L = K_1 \cup K_2$ and fix $\mathbf{s} \in \mathbb{H}(L)$.  Consider the framing $\Gamma = \begin{pmatrix} n_1 & 0 \\ 0 & 0 \end{pmatrix}$.  For sufficiently large $n_1$, we have that   
\[
\widehat{\hfbold}(Y_{\Gamma}(L),[\mathbf{s}]) \cong 
\xymatrix{
H_*([00_{\mathbf{s}}] \ar[r]^{\Psi^{K_2}} & [01_{\mathbf{s}}]).
}
\]
\end{lemma}
\begin{proof}
First, by Theorem~\ref{largesurgeries}, we may identify each $[00_{(s_1,s_2)}]$ with the Heegaard Floer homology of large surgery on both components of $L$, say $\widetilde{\Gamma} = (n_1,n_2)$, in the Spin$^c$ structure $\mathfrak{s}_{\widetilde{\Gamma}} = [(s_1,s_2)] \in \mathbb{H}(L)/H(L,\widetilde{\Gamma})$.  Similarly, $[01_{(s_1,s_2)}]$ is the Heegaard Floer homology of $n_1$-surgery on $K_1$ in the Spin$^c$ structure $\mathfrak{s}_1 = [s_1] \in \mathbb{H}(K_1)/H(K_1,\widetilde{\Gamma}|_{K_1})$ (here we have applied $\psi^{K_2}$ to $(s_1,s_2)$).  Recall that the maps $\Phi^{\pm K_2}$ correspond to cobordism maps for certain Spin$^c$ structures, say $\mathfrak{t}_{\pm}$, on the reversed 2-handle addition $W$ from $S^3_{n_1}(K_1)$ to $S^3_{\widetilde{\Gamma}}(L)$ by Proposition~\ref{mappingconeidentification}.  

Thus we have a quasi-isomorphism between 
\[
\xymatrix{
[00_{\mathbf{s}}] \ar[r]^{\Psi^{K_2}} & [01_{\mathbf{s}}]
}
\]
and 
\[
\xymatrix{
\widehat{HF}(S^3_{\widetilde{\Gamma}}(L),\mathfrak{s}_{\widetilde{\Gamma}}) \ar[r]^{F_{W,\mathfrak{t}_+} + F_{W,\mathfrak{t}_-}} & \widehat{HF}(S^3_{n_1}(K_1),\mathfrak{s}_1).
}
\]

Working from the other side, we can try to rewrite $S^3_\Gamma(L)$.  Let's study this via the link surgery formula, but from a different perspective.  We can also use the surgery formula to study 0-surgery on the knot $K_2$ where the ambient manifold is  $S^3_{n_1}(K_1)$ (as discussed in Remark~\ref{linksurgerygeneralization}, the surgery formula still works in arbitrary manifolds as long as the link is nullhomologous).  This sequence of surgeries will also give $S^3_{\Gamma}(L)$.  Again, by Proposition~\ref{mappingconeidentification}, but this time for the alternate surgery presentation, the Heegaard Floer homology of $S^3_{\Gamma}(L)$ in the Spin$^c$ structure $\mathfrak{s}$ will be given by the homology of the mapping cone 
\[
\xymatrix{
\widehat{HF}(S^3_{\widetilde{\Gamma}}(L),\mathfrak{s}_{\widetilde{\Gamma}}) \ar
[r]^{F_{W,\tilde{\mathfrak{t}}_+} + F_{W,\tilde{\mathfrak{t}}_-}} & \widehat
{HF}(S^3_{n_1}(K_1),\mathfrak{s}_1),
}
\]    
for two Spin$^c$ structures $\tilde{\mathfrak{t}}_+$ and $\tilde{\mathfrak{t}}_-$ on the reversed 2-handle addition $W$.  However, by the construction of the Spin$^c$ structures on $W$ (see Section 10 in \cite{hflz}), the $\tilde{\mathfrak{t}}_{\pm}$ are actually the same as $\mathfrak{t}_{\pm}$ and we are done.  
\end{proof}
  
For the framing $\Lambda = \mathbf{0}$, each $[\mathbf{s}]$ in $\mathbb{H}(L_2)/H(L_2,\Lambda)\cong \text{Spin}^c(S^3_{\Lambda}(L_2))$ has a single representative $\mathbf{s}$ in  $\mathbb{H}(L_2)$.  Combining Proposition~\ref{L1calc} with the following completes the proof of Theorem~\ref{noninvariancetheorem}.

\begin{proposition}\label{L2calc} a) For all $\mathfrak{s} \neq [(0,0)]$, $\widehat{FFH}(L_2,\Lambda,\mathfrak{s}) \cong \widehat{HF}(M,\mathfrak{s})$.  b) For $\mathfrak{s} = [(0,0)]$, we have $\rk \widehat{FFH}(L_2,\Lambda,\mathfrak{s})  \neq \rk \widehat{HF}(M,\mathfrak{s})$.  
\end{proposition}
\begin{proof}
We set $K_1 = T_L \# T_R$ and $K_2 = T_R$.  We first study the claim for $\mathfrak{s} \neq [(0,0)]$.  This follows from the calculations for $T_R$ and $T_L \# T_R$; for either knot, if $s_i \neq 0$, then $\Psi_{s_i}^{K_i}$ gives a surjection from $H_*(\widehat{\mathfrak{A}}(\hyperbox_2^{K_i},s_i))$ onto $\widehat{HF}(S^3)$.  By Remark~\ref{noninvariancerestrictedsystems},  $\Psi^{K_i}_{\mathbf{s}}$ is a surjection from $[01_{\mathbf{s}}]$ or $[10_{\mathbf{s}}]$ onto $[11_{\mathbf{s}}]$.  The $(E_1,d_1)$ complex for the $\varepsilon$-spectral sequence now has that $[11_\mathbf{s}]$ is in the image of one of the two possible $\Psi_{\mathbf{s}}^K$ maps.  Therefore, $E^0_2$ will be trivial.  This tells us $d_2$ must be 0.  Since the depth of the filtration is 2, the spectral sequence must collapse at $\widehat{FFH}$.  

It remains to show that $\widehat{FFH}(L_2,\Lambda,[(0,0)])$ is not isomorphic to $\widehat{HF}(M,[(0,0)])$.  Undoing the handleslide of $T_L$ over $T_R$ gives the framed link $(T_L \coprod T_R, (n,0))$.  Lemma~\ref{largetruncs} shows that each $\widehat{HF}(S^3_n(T_L) \# S^3_0(T_R) , \mathfrak{s})$ is isomorphic to the homology of the complex 
\[
\xymatrix{
[00_{\mathbf{s}}] \ar[r]^{\Psi^{K_2}} & [01_{\mathbf{s}}].
}
\]
Thus, there exists some $(s_1,s_2)$ such that
\begin{equation}\label{largeandzerosurgery}
\xymatrix{
H_*([00_{(s_1,s_2)}] \ar[r]^{\Psi^{K_2}} & [01_{(s_1,s_2)}])
}
\end{equation}
calculates $\widehat{HF}(S^3_n(T_L) \# S^3_0(T_R),\mathfrak{s}_1 \# \mathfrak{s}_2)$ where $\mathfrak{s}_1$ and $\mathfrak{s}_2$ are such that $\widehat{HF}(S^3_n(T_L),\mathfrak{s}_1)$ has rank 3 and  $\widehat{HF}(S^3_0(T_R),\mathfrak{s}_2)$ has rank 2.  Thus, (\ref{largeandzerosurgery}) must have rank 6 by the K\"unneth formula.

The first claim is that this must happen at $(s_1,s_2) = (0,0)$.  If not, then we have that $\widehat{FFH}(L_2,\Lambda,[(s_1,s_2)])=0$, since this agrees with  $\widehat{HF}(M,\mathfrak{s})$ when $(s_1,s_2) \neq (0,0)$.  However, both $[10_{(s_1,s_2)}]$ and $[11_{(s_1,s_2)}]$ have rank 1 from Section~\ref{rht} and Remark~\ref{noninvariancerestrictedsystems}.  Combining this with (\ref{largeandzerosurgery}), we must have that $\widehat{FFH}(L_2,\Lambda,[(s_1,s_2)])$ has rank at least 4, which is a contradiction.  

Therefore, $(s_1,s_2) = (0,0)$.  Based on our calculations, 
$\widehat{FFH}(L_2,\Lambda,[(0,0)])$ can be seen as the homology of the complex in Figure~\ref{L2complex}.

\begin{figure}[h]
\[
\xymatrix{
& [00_{(0,0)}] \ar[dl]_{\Psi_{(0,0)}^{T_L \# T_R}} \ar[dr]^{\Psi_{(0,0)}^{T_R}} \\
[10_{(0,0)}] \ar[dr]^0 & & [01_{(0,0)}] \ar[dl]_0 \\
& [11_{(0,0)}]
}
\]
\caption{The complex to calculate $\widehat{FFH}(L_2,\Lambda,[(0,0)])$}\label{L2complex}
\end{figure}

Recall that the ranks of $[10_{(0,0)}] = H_*(\widehat{\mathfrak{A}}(\hyperbox_2^{T_R},0))$ and $[11_{(0,0)}] \cong \widehat{HF}(S^3)$ are each one.  Thus, it is easy to see that the rank of this complex is at least 6.  Since the rank of $\widehat{HF}(M,[(0,0)])$ is 4, the proof is complete.  
\end{proof}

\begin{remark}\label{not4is6}
It actually follows that $\widehat{FFH}(L_2,\Lambda,[(0,0)])$ must in fact have  rank 6.  By depth arguments, $d_2$ is the only nontrivial higher differential in the $\varepsilon$-spectral sequence.  The only possibility is that $d_2$ must be non-zero from $[00_{(0,0)}]$ to $[11_{(1,1)}]$.  In particular, the rank of $d_2$ must be exactly 1.  Since $E_3 \cong E_\infty$ has rank 4, $\widehat{FFH}$ is rank 6.  
\end{remark}

\section{Framed Floer Homology and Mirrors}\label{mirrorsection}
An essential element in the proof of Theorem~\ref{noninvariancetheorem} was that the $d_2$ differential could be non-zero on $\widehat{FFH}(L_2,\Lambda,(0,0))$ because the maps $\Psi^{T_R}$ and $\Psi^{T_L \# T_R}$ were zero into $E^0_1$.  On the other hand, for $T_L$ the corresponding maps surjected onto $E^0_1$, making $E^0_2$ trivial.  By depth arguments there could be no higher differentials in the $\varepsilon$-spectral sequence.  With this in mind, we are ready to build our counterexample to mirror invariance.  

\begin{proof}[Proof of Proposition~\ref{nonmirrortheorem}]
Suppose $L = L_2$ is the link from the previous section, where we have handleslid the left-handed trefoil over the right-handed trefoil.  We keep $\Lambda = 0$ to have $\Lambda = -\Lambda$.  Therefore, $\bar{L}$ consists of $K_1 = T_L \# T_R$ and $K_2 = T_L$.

The key step is to establish that $\widehat{FFH}(\bar{L},-\Lambda,\mathfrak{s}) \cong \widehat{HF}(S^3_{-\Lambda}(\bar{L}),\mathfrak{s})$.  Fix any $\mathfrak{s} \in \mathbb{H}(\bar{L})$.  In the complex for the Framed Floer homology in this Spin$^c$ structure, there is a subcomplex 
\[
\xymatrix{
[10_{\mathfrak{s}}] \ar[r]^{\Psi^{T_L}} & [11_{\mathfrak{s}}].
}
\]
We know that $\Psi^{T_L}$ is always a surjection from the calculations in Section~\ref{lht}.  As we have seen, this implies that $\widehat{FFH}(\bar{L},-\Lambda,\mathfrak{s}) \cong \widehat{HF}(S^3_{-\Lambda}(\bar{L}),\mathfrak{s})$ for each Spin$^c$ structure.  Since we are ignoring gradings, this is isomorphic to $\widehat{HF}(S^3_\Lambda(L),\mathfrak{s})$ (Proposition 2.5 in \cite{hfpa}).  We compare this calculation to Proposition~\ref{L2calc} to see that this group is not isomorphic to $\widehat{FFH}(L,\Lambda,\mathfrak{s})$ if and only if $\mathfrak{s} = (0,0)$.
\end{proof}

\section{Surgery Exact Sequences}\label{surgeryexactsection}
In this section, we give a proof of the surgery exact triangle, analogous to the one found in Theorem 1.7 of \cite{hfpa}, but using the link surgery formula.  The surgery exact triangle will come from a short exact sequence of chain complexes where there are no counts of holomorphic polygons in the maps; in fact, everything is defined purely in terms of multiplications by $U$.  For simplicity, we will only present the proof for the hat-flavor; the proof for the other flavors is very similar and only requires adding in more multiplications by $U$.  

Let $(L',\Lambda')$ be an $(n-1)$-component framed link.  Suppose that $K$ is a nullhomologous knot in $Y_{\Lambda'}(L')$.  Define $L = L' \cup K$, so $K$ is the $n$th knot; take $\Lambda^r$ to be the framing for $L$ such that $\Lambda|_{L'} = \Lambda'$ and $\Lambda|_K$ is $r$.  If $L' = \emptyset$, then $\Lambda^r = r$.  

\subsection{A Single Knot}
To illustrate the proof of Theorem~\ref{sestheorem}, we will first do a simple case which outlines how this should go in general.  This is due to Manolescu.

\begin{proposition}\label{knotexact}
Let $K$ be a knot in $Y$ and fix a complete system of hyperboxes, $\hyperbox$, for $K$.  There is a short exact sequence of $\varepsilon$-filtered chain complexes
\[
0 \longrightarrow \widehat{\C}(\hyperbox|_{\emptyset}) \longrightarrow \widehat{\C}(\hyperbox,0) \longrightarrow \widehat{\C}(\hyperbox,1) \longrightarrow 0
\] which induces a short exact sequence on the $\widehat{FFC}$-complexes.
\end{proposition}

We first need some background.  There are two types of complexes in $\widehat{\C}(\hyperbox,\Lambda)$ when working solely with a knot.  There are the $0_s$ and the $1_s$; the $0_s$ are the $\widehat{\mathfrak{A}}(\hyperbox^K,s)$-complexes, while the $1_s$ are simply copies of $\widehat{\mathfrak{A}}(\hyperbox^\emptyset) = \widehat{\C}(\hyperbox|_\emptyset)$.  From the perspective of the integer surgeries formula for knots, the $0_s$ are the $A_s$ and the $1_s$ are the $B_s$.  We choose $U_i = 0$, where $K$ has been assigned color $i$.  

Recall that $A(x)$ denotes the Alexander grading on $x$, satisfying $A(x) - A(y) = n_z(\phi) - n_w(\phi)$ for any $\phi \in \pi_2(x,y)$.  Define maps $\proj_s^{\pm} :0_s \longrightarrow 0_{s \pm 1}$ on the chain group generators as follows.

\begin{equation*}
\proj_s^+(x)  = \left\{
\begin{array}{rl}
x & \text{if } A(x) \leq s \\
0 & \text{if } A(x) > s
\end{array} \right. \text{ and } 
\proj_s^-(x)  = \left\{
\begin{array}{rl}
x & \text{if } A(x) \geq s \\
0 & \text{if } A(x) < s
\end{array} \right. 
\end{equation*}

It is not hard to see that these are chain maps.  In fact, these are just the maps $(\incl_{s \pm 1}^{\pm K})^{-1} \circ \incl_s^{\pm K}$ (while $\incl_{s \pm 1}^{\pm K}$ isn't necessarily invertible, an inverse can be defined on the image of $\incl_s^{\pm K}$).  For notation, we will use $_r \varepsilon_s$ to mean the complex $\varepsilon_s$ in $\widehat{\C}(\hyperbox,\Lambda^r)$.  It is important to remember that the chain groups $\varepsilon_s$ are the same for all $s$ - it is the differential that distinguishes them.  Furthermore, we will often not distinguish when $\Phi^{\pm K}$ has domain given by $_0 0_s$ or $_1 0_s$.  This will be clear from the context.    

\begin{remark}\label{projcompsvanish}
The $\proj$ maps are defined to behave well with the $\Phi$ maps.  For a generator $x \in 0_s$, we have that $\Phi_{s\pm 1}^{\pm K} \circ \proj_s^{\pm}(x) = \Phi^{\pm K}_s(x)$.  Here, the domains and ranges don't necessarily match up, but we can still compare the generators since all $0_s$ share the same generators and all $1_s$ share the same generators.    For the opposite signs, 
\[
\Phi_{s \pm 1}^{\mp K} \circ \proj_s^{\pm} = \incl_{s \pm 1}^{\mp K} \circ \proj_s^{\pm} = 0 .
\]
The reason for these strange comparisons will make sense shortly.
\end{remark}

\begin{proof}[Proof of Proposition~\ref{knotexact}]
We first construct $f: \widehat{\C}(\hyperbox|_\emptyset) \longrightarrow \widehat{\C}(\hyperbox,0)$.  In $\widehat{\C}(\hyperbox,0)$, each $_01_s$ is a copy of $\widehat{\C}(\hyperbox|_\emptyset) = \widehat{\mathfrak{A}}(\hyperbox^\emptyset)$, so we choose $f_s$ to be the map which is simply the identity from $\widehat{\C}(\hyperbox|_\emptyset)$ to $_0 1_s$.  We define $f = \prod_{s \in \mathbb{H}(K)} f_s$.  This is clearly an injective chain map which respects the $\varepsilon$-filtration.  Note that by Theorem 11.1 of \cite{hflz}, this corresponds to the usual sum of cobordism maps used in constructing surgery exact triangles.

We now need to define $g:\widehat{\C}(\hyperbox,0) \longrightarrow \widehat{\C}(\hyperbox,1)$.  This will split as the sum of two maps 
\[
g^0: \prod_{s \in \mathbb{H}(K)} {_0} 0_s \longrightarrow \prod_{s \in \mathbb{H}(K)} {_1} 0_s, 
\]
which we call the {\em top} map, and 
\[
g^1: \prod_{s \in \mathbb{H}(K)} {_0} 1_s \longrightarrow \prod_{s \in \mathbb{H}(K)} {_1} 1_s, 
\]
which is the {\em bottom} map.  In the notation of the integer surgery formula for knots, this says $A_s$'s go to $A_s$'s and $B_s$'s go to $B_s$'s.  
Each $g^i$ will be defined componentwise over $\mathbb{H}(K)$, so we refer the reader to Figure~\ref{surgerychainmap} for a visualization of the setup.  We will not keep track of the components of elements which are 0.  For example, if $x \in {_0\varepsilon_{s_0}}$, let $\overline{x}$ be the element of $\widehat{C}(\hyperbox,0)$ which is $x$ in $_0 \varepsilon_{s_0}$ and 0 elsewhere.  We will use $g(x)$ to really mean $g(\overline{x})$; furthermore, we would not write the components of $g(\overline{x})$ which are zero.  

\begin{figure}
\[
\xymatrix{
{_00_s} \ar@/_1pc/[d]_{+} \ar@/^1pc/[d]^{-} \ar@/^1pc/[rr]^{\proj_s^-} \ar@/^3pc/[rrr]_{id} 
\ar@/^4pc/[rrrr]^{\proj_s^+}  & 
& {_10_{s-1}} \ar[d]_{+} \ar[dr]^{-} & {_10_s} \ar[d]_{+} \ar[dr]^{-} & {_10_{s+1}} \ar[d]_{+} \ar[dr]^{-} & \\
{_01_s} \ar@/_1pc/[rrr]_{id} \ar@/_3pc/[rrrr]_{id}  &&  {_11_{s-1}} & {_11_s} & {_11_{s+1}} &
}
\]
\caption{A local picture for $g$ at a fixed $s$ in $\mathbb{H}(K)$}\label{surgerychainmap}
\end{figure}
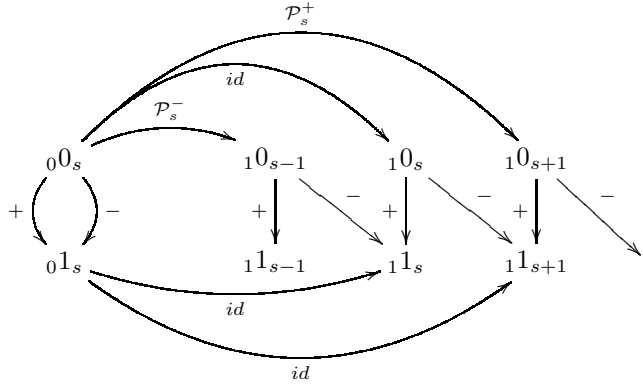

Let's first do the bottom map, $g^1$.  For each $s$, the map $g^1_s$ will be 
\[
g^1_s: {_0} 1_s \longrightarrow {_1}1_s \oplus {_1}1_{s+1} \text{ where } g^1_s(x)=(x,x).
\]
This is clearly a chain map, since $_r1_s = \widehat{\mathfrak{A}}(\hyperbox^\emptyset)$ for all $s$ and $r$.  Again, define $g^1 = \prod_{s \in \mathbb{H}(K)} g^1_s$.  From this construction, the kernel of $g^1$ is exactly the image of $f$.  Let's see that $g^1$ is surjective.  Suppose $x$ is in $_1 1_s$; define a sequence in the domain of $g^1$ by 
\begin{equation*}
\tilde{x}^{(s')}  = \left\{
\begin{array}{rl}
x & \text{if } s' \geq s\\
0 & \text{if } s' < s
\end{array} \right. 
\end{equation*}
This clearly satisfies $g^1(\tilde{x}) = x$.  

Next, we would like to construct the top map $g^0$ between the $0_s$ complexes.  For each $s$, define
\[
g_s^0: {_00_s} \longrightarrow {_10_{s-1}} \oplus {_10_s} \oplus {_10_{s+1}}, \text{ where } g_s^0(x) = (\proj_s^-(x),x,\proj_s^+(x)).
\]

We must check to see that the map $g = g^0 + g^1$ is now a chain map.  
We have already considered the case where $x \in {_01_s}$.  Now suppose that $x \in {_00_s}$.  Let's use $_rD$ for the differentials on $\widehat{\C}(\hyperbox,r)$, while $\partial$ will still represent the standard differential on $\varepsilon_s$.  
This tells us that 
\[
_0D(x) = (\partial(x),\Phi_s^{+K}(x) + \Phi_s^{-K}(x)) \in {_00_s} \oplus {_01_s} \text{ for } x \in {_0\varepsilon_s}.
\]
Therefore, 
\begin{align*}
g(_0D(x)) = & \:(\proj_s^- \partial(x),\partial(x),\proj_s^+ \partial(x), \\
&\qquad \Phi_s^{+K}(x) + \Phi_s^{-K}(x),\Phi_s^{+K}(x) + \Phi_s^{-K}(x)) \\ 
&\qquad \qquad \in {_10_{s-1}} \oplus {_10_s} \oplus {_10_{s+1}} \oplus {_11_s} \oplus {_11_{s+1}}.
\end{align*}
Now, $g_s(x) = (\proj_s^-(x),x,\proj_s^+(x))$ and thus
\begin{align*}
_1D(g(x)) = & \: (\partial\proj_s^-(x),\partial(x),\partial\proj_s^+(x), \\
& \qquad \Phi_{s-1}^{+K}\proj_s^-(x),\Phi_{s-1}^{-K}\proj_s^-(x) + \Phi_s^{+K}(x),\\
& \qquad \qquad \Phi_s^{-K}(x) + \Phi_{s+1}^{+K}\proj_s^+(x), \Phi_{s+1}^{-K}\proj_s^+(x)) \\
& \in {_10_{s-1}} \oplus {_10_s} \oplus {_10_{s+1}} \oplus {_11_{s-1}} \oplus {_11_s} \oplus {_11_{s+1}} \oplus {_11_{s+2}}.
\end{align*}
By Remark~\ref{projcompsvanish} and the fact that $\proj^{\pm}$ are chain maps, $g(_0D(x)) = {_1}D(g(x))$.

We will now prove injectivity of $g^0$ to show that the image of $f$ is in fact the kernel of $g$.  Let $(x_s)_{s \in \mathbb{H}(K)}$ be in the kernel of $g^0$.  Without loss of generality, all of the $x_s$ satisfy $A(x_s) = k$, for some $k$ - this is because $g^0$ preserves the Alexander filtration of generators by definition, so it suffices to check injectivity on each Alexander filtration level.  Then, we have the following equation for each $s$,
\begin{equation}\label{surgeryinjection}
g^0_s(x) = \proj_{s-1}^+(x_{s-1}) + x_s + \proj_{s+1}^-(x_{s+1}) = 0.
\end{equation}
If $s > k$, then (\ref{surgeryinjection}) becomes $x_{s-1} + x_s = 0$.  Thus, $x_s = x_k$ for all $s \geq k$ (remember that we can relate these since the chain {\em groups} $0_s$ do not depend on $s$).  A similar argument shows that $x_s = x_k$ for all $s \leq k$.     
Since $\proj_{k \pm 1}^{\mp}(x_{k \pm 1}) = 0$, (\ref{surgeryinjection}) shows that $x_k = x_s = 0$ for all $s$.

Finally, we would like to show that $g^0$ is surjective; this will complete the construction of the short exact sequence.  Again, let $x$ be a generator of ${_1}0_s$ with $A(x) = k$.  First, suppose $s < k$; the case of $s > k$ is similar. We will construct an element $\tilde{x}$ with $g^0(\tilde{x}) = x$.    For notation, $x^{(s')}$ means the generator $x$, but now sitting in ${_0}0_{s'}$.  Consider the sequence in $\prod_{s} {_0}0_s$ given by 
\begin{equation*}
\tilde{x}_{s'} = \left\{
\begin{array}{rl}
x^{(s')} & \text{if } s' \leq s\\
0 & \text{if } s' > s
\end{array}. \right. 
\end{equation*}
Note that $\proj_{s'}^+(x^{(s')}) = 0$ for $s' \leq s$, since $s < k$.  Similarly, $\proj_{s'}^-(x^{(s')}) = x^{(s'-1)}$ for $s' \leq s$.  Therefore, we must have that $g^0_{s'}(\tilde{x}) = 0$ for $s' > s$ and at each $s' < s$, there are two components of $g^0$,  $\proj_{s'+1}^-(x^{(s'+1)})$ and $x^{(s')}$.  By construction, their sum is 0.  Since the only component of $g^0$ which is non-zero into $_10_s$ is the identity on $x^{(s)}$, we have attained $g^0(\tilde{x}) = x$.       
The final case is when the generator of ${_1}0_s$ has $A(x) = k = s$.  For this, we use the constant sequence $\tilde{x}_s = x^{(s)}$ in $\prod_{s} {_0}0_s$.  A similar argument shows that $g^0(\tilde{x}) = x$.  

By construction, $f$ and $g$ are $\varepsilon$-filtered, but do not lower the filtration level.  Therefore, we obtain a short exact sequence 
\[
\xymatrix{
0 \ar[r] & (E_0(\widehat{\C}(\hyperbox^\emptyset)),d_0) \ar[r]^{f_0} & (E_0(\widehat{\C}(\hyperbox,0)),d_0) \ar[r]^{g_0} & (E_0(\widehat{\C}(\hyperbox,1)),d_0) \ar[r] & 0.
}
\]
This induces a long exact sequence between the $E_1$ pages of the spectral sequences.  However, by construction, the induced map $f_1$ on $E_1$ pages is still injective, since it is the identity componentwise.  Thus, the long exact sequence on $E_1$ pages is actually short exact; in other words, this establishes a short exact sequence for $\widehat{FFC}$.  This completes the proof.
\end{proof}

\begin{remark}
The short exact sequence in Proposition~\ref{knotexact} automatically gives a surgery exact triangle for Heegaard Floer homology when taking homology.  However, this simple short exact sequence will not quite fit into our combinatorial framework unless the ambient manifold in this construction is $S^3$; it also does not cover the Framed Floer homology exact triangle we would like.  Therefore, we will need the more general construction which follows.  We should also note that since $\widehat{FFH} \cong \widehat{HF}$ for knots, the last part of the proof of Proposition~\ref{knotexact} was unnecessary; this was done so as to elucidate the argument for the general case.      
\end{remark}

\subsection{The General Surgery Exact Sequence}
Using the framework from above, we would like to prove Theorem~\ref{sestheorem}.  The idea is simple.  In the surgery formula, we would like to simplify notation by compressing everything that does not involve $K$.  Since $\mathbb{H}(L)/H(L,\Lambda|_{L-K}) \cong \mathbb{Z} \cong \mathbb{H}(K)$, we should try to recreate the picture from the proof of Proposition~\ref{knotexact}.  In fact, this compressing idea is exactly what we used in the proof of Proposition~\ref{stabilizationinvariance}, where we summed over the complexes which did not involve the unknot $U$.  Similarly in the current case, the {\em bottom} (objects with $\varepsilon_n=1$) and {\em top} (objects with $\varepsilon_n=0$) will split into $1$'s and $0$'s corresponding to whether or not $K$ has been destabilized.  When we introduce the differentials with $\Lambda^0$ and $\Lambda^1$, the shape of these compactified complexes will now be that from Figure~\ref{surgerychainmap}.   We then define the maps in the short exact sequence by the analogous $\proj$-maps.  Now for some details.  

Recall that for a complete system $\hyperbox$ and a sublink $L'$, we can construct the complete system $\hyperbox|_{L'}$ for $L'$.  First fix $\mathbf{s} \in \mathbb{H}(L)$.  Let $\langle \mathbf{s} \rangle$ be the set of all $\mathbf{s'}$ that are attainable from $\mathbf{s}$ by adding linear combinations of $\Lambda_i$ for $i \neq n$.  Now, we can consider 
\[
\C^1_{\langle \mathbf{s} \rangle} = \prod_{\mathbf{s'} \in \langle \mathbf{s} \rangle} \bigoplus_{\varepsilon_n = 1} \varepsilon_{\mathbf{s'}} 
\]
We can think of this as beginning with some vertex complex at $\mathbf{s}$ where $K$ has been destabilized and compressing together all complexes in the `$L-K$ direction'.  Note that by Remark~\ref{restrictedsystems}, there is a natural identification of this complex with $\C(\hyperbox',\Lambda')$, which we will think of as the identity map (similar to in Proposition~\ref{knotexact}).  

We may also compress terms in the $L-K$ direction which do not yet have $K$ destabilized.  These complexes are given by 
\[
\C^0_{\langle \mathbf{s} \rangle} = \prod_{\mathbf{s'} \in \langle \mathbf{s} \rangle} \bigoplus_{\varepsilon_n = 0} \varepsilon_{\mathbf{s'}}.
\]
We will again use a subscript $_r \star$ to indicate that the object $\star$ lives in $\widehat{C}(\hyperbox,\Lambda^r)$.

\begin{remark}
If $\langle \mathbf{s} \rangle$ and $\langle \mathbf{s'} \rangle$ are different, but sit in the same Spin$^c$ structure on $Y_{\Lambda^1}(L)$, then they must be related by some multiple of $(0,\ldots,0,1)$.  Note that $(0,\ldots,0,1)$ is actually the difference between the framing vectors $\Lambda^1_n$ and $\Lambda^0_n$, since $K$ is the $n$th component. To compactify notation, let's write $(0,\ldots,0,1) = e_n$.  We also use the notation
\[
\Phi^{*} = \sum_{\vec{M} \subseteq L'} \Phi^{\vec{M}} \text{ and } \Phi^{\vec{K} \cup *} = \sum_{\vec{M} \subseteq L'} \Phi^{\vec{K} \cup \vec{M}}.
\]
Therefore, we have the familiar picture in Figure~\ref{generalsurgeryfigure} to help visualize these complexes.
\begin{figure}
\[ 
\xymatrix{
{_1}\C^0_{\langle \mathbf{s-e_n} \rangle} \ar[dd]_{\textstyle{\widehat{C}(\hyperbox,\Lambda^1) \: = \:} \scriptstyle{+K \cup *}} \ar[ddr]_{-K \cup *} & {_1\C^0_{\langle \mathbf{s} \rangle}} \ar[dd]_{+K \cup *} \ar[ddr]_{-K \cup *}& {_1\C^0_{\langle \mathbf{s+e_n} \rangle}} \ar[dd]_{+K \cup *} \ar[ddr]^{-K \cup *}\\\\
{_1\C^1_{\langle \mathbf{s-e_n} \rangle}} & {_1\C^1_{\langle \mathbf{s} \rangle}} & {_1\C^1_{\langle \mathbf{s+e_n} \rangle}} &
}
\]
\caption{Link surgery complex recast like the knot surgery formula}\label{generalsurgeryfigure}
\end{figure}
\end{remark}

\begin{proof}[Proof of Theorem~\ref{sestheorem}] We are ready to construct the appropriately filtered short exact sequence.  To visualize the setup, refer to Figure~\ref{generalsurgeryfigure}.  As before, the first map $f$ is given by the product of $f_{\langle s \rangle}$, where each $f_{\langle s \rangle}$ is the identity map from  $\widehat{\C}(\hyperbox',\Lambda')$ to $\C^1_{\langle\mathbf{s}\rangle}$.  This is still clearly injective.  Furthermore, $f$ will still preserve the filtration level and induce injections on the $E_0$ pages and $E_1$ pages.  As the maps $f$ and $g$ will not lower the filtration levels, it suffices to establish that $f$ and $g$ form a short exact sequence; the result for the $E_1$ pages will immediately follow by the same argument as in Proposition~\ref{knotexact}.  

The next step is to construct $g$.  For each $\varepsilon$ with $\varepsilon_n = 1$, define $g^{\varepsilon}_{\mathbf{s}} : {_0\varepsilon_{\mathbf{s}}} \longrightarrow {_1}\varepsilon_{\mathbf{s}} \oplus {_1}\varepsilon_{\mathbf{s}+e_n}$ by $x \mapsto (x,x)$.
It is again easy to see that 
\[
g^1 = \prod_{\mathbf{s} \in \mathbb{H}(L)} \bigoplus_{\varepsilon_n = 1} g^{\varepsilon}_{\mathbf{s}}
\]
surjects onto the bottom of $\widehat{\C}(\hyperbox,\Lambda^1)$ and has the image of $f$ as its kernel.  

Now, when $\varepsilon_n = 0$, define 
\begin{align*}
g^{\varepsilon}_{\mathbf{s}}\!: \; &{_0}\varepsilon_{\mathbf{s}} \longrightarrow {_1}\varepsilon_{\mathbf{s}-e_n} \oplus {_1}\varepsilon_{\mathbf{s}} \oplus {_1}\varepsilon_{\mathbf{s}+e_n}\\
& x \mapsto (\proj_{\mathbf{s}}^-(x),x,\proj_{\mathbf{s}}^+(x)).
\end{align*}
Here, we still have the same $\proj$ as before.  By this, we mean the following: first, choose $U_n$ to be the variable set to 0.  Now, define  
\[
\proj_{\mathbf{s}}^{\pm}(x) = (\incl_{\mathbf{s}\pm e_n}^{\pm K})^{-1} \circ \incl_{\mathbf{s}}^{\pm K}(x) \text{ for } x \in {_0}\varepsilon_{\mathbf{s}}.
\]
It is not hard to see that the $\proj^{\pm}$ are still chain maps and commute with all other $\Phi^{\vec{M}}$ when $K$ is not in $M$.  This follows from (\ref{destabinclcommute}).  We take $g^0$ to be the necessary product over $\mathbf{s}$ and $\varepsilon$.  

Following suit with the single-component case, it suffices to show two things.  One is that $g = g^0 + g^1$ is an authentic chain map, while the other is that $g^0$ is a bijection.  The fact that $g^0$ is a bijection follows from the same arguments as for Proposition~\ref{knotexact}.  Thus, the proof is complete after we establish the following lemma. 
\end{proof}
\begin{lemma} $g$ is a chain map. 
\end{lemma}
Recall that if we are applying a map that goes out of the hypercube (say $\Phi^{+K}$ applied to $x$ in $\widehat{\mathfrak{A}}(\hyperbox^\emptyset)$) then the convention is that this must be zero.  From now on, instead of $\mathbf{s}$, we will only keep track of the class $\langle \mathbf{s} \rangle$ as to avoid worrying about terms that go in the $L-K$ direction; therefore, we will abuse notation with terms like $\Phi^{+K \cup *}_{\langle \mathbf{s} \rangle}$.    

\begin{proof}
Fix $x \in { }_0\varepsilon_{\mathbf{s}}$.  
We apply the differential to $x$ to obtain 
\[
_0D(x) = (\Phi_{\langle \mathbf{s} \rangle}^*(x), \Phi_{\langle \mathbf{s} \rangle}^{+K \cup *}(x) + \Phi_{\langle \mathbf{s} \rangle}^{-K \cup *}(x)) \in {_0}\C^0_{\langle \mathbf{s} \rangle} \oplus {_0}\C^1_{\langle \mathbf{s} \rangle}.
\]
This gives 
\begin{align*}
g({_0}D(x)) = (\proj_{\langle \mathbf{s} \rangle}^-&\Phi_{\langle \mathbf{s} \rangle}^*(x),\Phi_{\langle \mathbf{s} \rangle}^*(x),\proj_{\langle \mathbf{s} \rangle}^+\Phi_{\langle \mathbf{s} \rangle}^*(x), \\ &
 \qquad\qquad\qquad\Phi_{\langle \mathbf{s} \rangle}^{+K \cup *}(x) + \Phi_{\langle \mathbf{s} \rangle}^{-K \cup *}(x), \Phi_{\langle \mathbf{s} \rangle}^{+K \cup *}(x) + \Phi_{\langle \mathbf{s} \rangle}^{-K \cup *}(x)) \\
& \in {_1}\C^0_{\langle \mathbf{s} - e_n \rangle} 
\oplus {_1}\C^0_{\langle \mathbf{s} \rangle}
\oplus {_1}\C^0_{\langle \mathbf{s} + e_n \rangle} 
\oplus {_1}\C^1_{\langle \mathbf{s} \rangle}
\oplus {_1}\C^1_{\langle \mathbf{s} + e_n \rangle}.  
\end{align*}
Now, we compare with 
\begin{align*}
{_1}D(g(x)) &= {_1}D(\proj_{{\langle \mathbf{s} \rangle}}^-(x),x,\proj_{{\langle \mathbf{s} \rangle}}^+(x)) \\
	      &= (    \Phi_{{\langle \mathbf{s} - e_n \rangle}}^*\proj_{{\langle \mathbf{s} \rangle}}^-(x), \Phi_{{\langle \mathbf{s} \rangle}}^*(x), \Phi_{{\langle \mathbf{s}+ e_n \rangle}}^*\proj_{{\langle \mathbf{s} \rangle}}^+(x), \\
& \qquad\qquad \Phi_{{\langle \mathbf{s} - e_n \rangle}}^{+K \cup *}\proj_{{\langle \mathbf{s} \rangle}}^-(x), \Phi_{{\langle \mathbf{s} - e_n\rangle}}^{-K \cup *}\proj_{{\langle \mathbf{s} \rangle}}^-(x) + \Phi_{{\langle \mathbf{s} \rangle}}^{+K \cup *}(x), \\
& \qquad \qquad\qquad  \Phi_{{\langle \mathbf{s} \rangle}}^{-K \cup *}(x) + \Phi_{{\langle \mathbf{s} + e_n \rangle}}^{+K \cup *}\proj_{\langle \mathbf{s} \rangle}^+(x), \Phi_{{\langle \mathbf{s} + e_n \rangle}}^{-K \cup *}\proj_{{\langle \mathbf{s} \rangle}}^+(x)   )  \\    
& \in {_1}\C^0_{\langle \mathbf{s} - e_n \rangle} 
\oplus {_1}\C^0_{\langle \mathbf{s} \rangle}
\oplus {_1}\C^0_{\langle \mathbf{s} + e_n \rangle} 
\oplus {_1}\C^1_{\langle \mathbf{s} - e_n\rangle}
\oplus {_1}\C^1_{\langle \mathbf{s}  \rangle}
\oplus {_1}\C^1_{\langle \mathbf{s} + e_n \rangle}
\oplus {_1}\C^1_{\langle \mathbf{s} + 2 e_n \rangle}
.
\end{align*}
Again, we have that $\Phi_{\mathbf{s} \pm e_n}^{\mp K} \circ \proj_{\mathbf{s}}^{\pm K}= 0$.  Studying the inclusion maps $\incl$ shows that if $\incl^{\vec{M'}}(x) = 0$ and $\vec{M'}$ is a compatibly oriented sublink of $\vec{M}$, then $\incl^{\vec{M}}(x) = \incl^{\vec{M}-M'} \circ \incl^{\vec{M'}}(x) = 0$ as well.  Therefore, the maps $\Phi_{\langle \mathbf{s} \pm e_n \rangle}^{\mp K \cup *} \circ \proj_{\langle \mathbf{s} \rangle}^{\pm}$ vanish.  Here, we can apply (\ref{destabinclcommute}) to make the remaining commutations between $\proj^{\pm}$ and $\Phi^*$, since we are comparing inclusions for $K$ and destabilizations for sublinks of $L-K$.  This completes the proof.  
\end{proof}

\subsection{Combinatorial Surgery Exact Triangles}
\begin{proof}[Proof of Corollary~\ref{combinatorialles}]
In the construction of Theorem~\ref{sestheorem}, we only used multiplications by powers of $U$ which depended on the Alexander gradings.  These can be calculated by an explicit formula from a grid diagram (see Section 1 in \cite{hflcombo}).  In the proof, the only additional information we needed was the relation (\ref{destabinclcommute}).  This is  proved in Lemma 7.4 of \cite{hflz} by simply counting more powers of $U$.  These observations combined with Theorem~\ref{combinatorialhf} show that everything (construction and proof) is combinatorial.       
\end{proof}



\section{Framed Floer Homology and Property 2R}\label{property2rsection}
In this section, we make use of the non-invariance of Framed Floer homology to study Property 2R.  Let's again recall the definition of this property.  

\begin{definition}
We will say that a two-component link $L$ in $S^3$ has {\em Property 2R} if whenever $\Lambda$-surgery on $L$ yields $S^2 \times S^1 \# S^2 \times S^1$, the pair $(L,\Lambda)$ can be related to the 0-framed, two-component unlink, $V$, by only handleslides (no stabilizations allowed).
\end{definition} 

A link is {\em algebraically split} if all pairwise linking numbers are 0.  As discussed in the introduction, if a two-component link surgers to $S^2 \times S^1 \# S^2 \times S^1$, then it must be algebraically split and the framing must be identically 0.  

In order to prove Proposition~\ref{property2r}, we must review how to calculate four-dimensional cobordism maps using the link surgery formula.  We will restrict our attention to the setting required for Property 2R.  In particular, given any two-component, algebraically split link, $L$, we construct a four-manifold $X_L$ by attaching 0-framed two-handles to $S^3$ along the components of $L$.  We equip $X_L$ with the Spin$^c$ structure $\mathfrak{t}_L$ which restricts to the unique torsion Spin$^c$ structure $\mathfrak{s}_0$ on $S^3_{\mathbf{0}}(L)$.  

Recall the notation $11_{(0,0)}$ - this is the complex $\widehat{\mathfrak{A}}(\hyperbox^\emptyset,\psi^L((0,0)))$, or in other words, $\widehat{\mathbf{CF}}(S^3)$.  Furthermore, $[11_{(0,0)}]$ is the homology of the vertex complex $11_{(0,0)}$, which is $\widehat{HF}(S^3)$.  Theorem 11.2 of \cite{hflz} tells us that 
\[
\xymatrix{
Cone(\iota: 11_{(0,0)} \hookrightarrow \widehat{\C}(L,\mathbf{0},\mathfrak{s}_0))
}
\]
is quasi-isomorphic to 
\[
Cone(\widehat{f}_{X_L,\mathfrak{t}_L}: \widehat{\mathbf{CF}}(S^3) \longrightarrow \widehat{\mathbf{CF}}(S^3_{\mathbf{0}}(L),\mathfrak{s}_0))
\] 
via the identifications of Theorem~\ref{surgerytheorem}.  Since the rank of the induced map $\widehat{F}_{X_L,\mathfrak{t}_0}$ on homology is a diffeomorphism invariant \cite{hfsmooth4}, we must have that the rank of the induced map $\iota_*:[11_{(0,0)}] \longrightarrow H_*(\widehat{\C}(L,\mathbf{0},\mathfrak{s}_0))$ is an invariant as well.  This is because for any $L$, 
\[
[11_{(0,0)}] \cong \widehat{HF}(S^3) \cong \mathbb{F}.
\]    
 
\begin{remark}
The inclusion of $11_{(0,0)}$ into $\widehat{\C}(L,\mathbf{0},\mathfrak{s}_0)$ is $\varepsilon$-filtered, so we may study the induced maps $\iota_i$ on the $i$th pages of the respective $\varepsilon$-spectral sequences. 
\end{remark}

\begin{proof}[Proof of Proposition~\ref{property2r}]
Suppose that $\widehat{FFH}(L,\mathbf{0},\mathfrak{s}_0)$ is not rank 4 and $L$ has Property 2R.  Consider the four-manifold, $X_L$, obtained by attaching 0-framed two-handles to $S^3$ along the components of $L$.  Since $L$ can be converted to $V$ by only handleslides, $X_L$ and $X_{V}$ are diffeomorphic.  As mentioned above, we must have that the ranks of the respective inclusion maps must agree on homology.  The contradiction is given by the following two lemmas.        
\end{proof}

While the first lemma is a well-known fact, we include it for practice with the link surgery formula and cobordism maps.  It will always be assumed that we are restricting our spectral sequence to the subcomplexes for $\mathfrak{s}_0$.  
\begin{lemma}\label{unlinkisinjective}
The map $\iota_*:[11_{(0,0)}] \longrightarrow H_*(\widehat{\C}(V,\mathbf{0},\mathfrak{s}_0))$ is non-zero.  
\end{lemma}
\begin{proof}
We proceed page by page through the $\varepsilon$-spectral sequence using prior knowledge.  The first page is 
\[
E_1(\widehat{C}(V,\mathbf{0},\mathfrak{s}_0))  = [00_{(0,0)}] \oplus [01_{(0,0)}] \oplus [10_{(0,0)}] \oplus [11_{(0,0)}].
\]
We may calculate all of these terms simply by appealing to Theorem~\ref{largesurgeries}.  This says that the $[{\varepsilon_1 \varepsilon_2}_{\mathbf{s}}]$ are given by the Heegaard Floer homologies of large surgeries on various sublinks of $V$ in certain Spin$^c$ structures.  However, any non-zero surgery on a sublink of $V$ is $S^3$, a lens space, or a connect-sum of lens spaces.  All of these manifolds will have rank 1 for their Heegaard Floer homologies in each Spin$^c$ structure.  Therefore, the rank of the $E_1$ page in the torsion Spin$^c$ structure is 4.  However, this spectral sequence is converging to the Heegaard Floer homology of $S^2 \times S^1 \# S^2 \times S^1$, which itself has rank 4.  Therefore, all higher differentials must vanish and $E_1(\widehat{C}(V,\mathbf{0},\mathfrak{s}_0)) \cong E_\infty(\widehat{C}(V,\mathbf{0},\mathfrak{s}_0))$.  

We can also consider the $\varepsilon$-spectral sequence on the 0-dimensional hypercube of chain complexes $11_{(0,0)}$.  Clearly, we have that $E_1 \cong E_\infty \cong [11_{(0,0)}]$.  Therefore, the rank of the induced map $\iota_\infty: E_\infty(11_{(0,0)}) \longrightarrow E_\infty(\widehat{\C}(V,\mathbf{0},\mathfrak{s}_0))$ is equal to the rank of $\iota_1$.  By Fact~\ref{filteredisoisiso}, $\iota_*$ will be non-zero if we show that $\iota_1$ is non-zero.  However, this is now clear as the map $\iota_1$ is the obvious inclusion of $[11_{(0,0)}]$ into $[00_{(0,0)}]  \oplus [01_{(0,0)}] \oplus [10_{(0,0)}] \oplus [11_{(0,0)}]$.   
\end{proof}

\begin{lemma}\label{linkis0}
If the rank of $\widehat{FFH}(L,\mathbf{0},\mathfrak{s}_0)$ is not 4, then the inclusion $\iota_*:[11_{(0,0)}] \longrightarrow H_*(\widehat{\C}(L,\mathbf{0},\mathfrak{s}_0))$, is identically 0.
\end{lemma}
\begin{proof}
By assumption, we have that the Framed Floer homology is not isomorphic to the Heegaard Floer homology of $S^2 \times S^1 \# S^2 \times S^1$.  This tells us that there must be higher differentials in the $\varepsilon$-spectral sequence.  Since the depth of the $\varepsilon$-filtration for $L$ is only two, we must have a non-trivial $d_2$ differential.  If this happens, this means that 
\[
d_2:E^{2}_2(\widehat{C}(L,\mathbf{0},\mathfrak{s}_0)) \longrightarrow E^0_2(\widehat{C}(L,\mathbf{0},\mathfrak{s}_0)) 
\]
is non-zero (remember that $E^j_i$ refers to terms in the $j$th filtration level).  Note that $E^0_2(\widehat{C}(L,\mathbf{0},\mathfrak{s}_0))$ has rank at most 1, as $E^0_1(\widehat{C}(L,\mathbf{0},\mathfrak{s}_0)) = [11_{(0,0)}]$.  Thus, $d_2$ is surjective, and $E^0_3(\widehat{C}(L,\mathbf{0},\mathfrak{s}_0)) \cong E^0_\infty(\widehat{C}(L,\mathbf{0},\mathfrak{s}_0))$ is 0.  However, as the $E_3$ page of the $\varepsilon$-spectral sequence for the 0-dimensional hypercube of chain complexes $(11_{(0,0)},\partial)$ is $[11_{(0,0)}]$ (and again also the $E_\infty$ page), we must have that $\iota_3 = \iota_\infty = 0$.  Therefore, $\iota_*$ is 0 by Fact~\ref{filteredisoisiso}.       
\end{proof}

An argument similar to that in Remark~\ref{not4is6} implies that if the rank of $\widehat{FFH}(L,\Lambda,\mathfrak{s})$ is not 4, it must be 6.  

If Framed Floer homology was an invariant of $S^3_\Lambda(L)$, then we could never have the setting described in Proposition~\ref{property2r}, since the Framed Floer homology for the 0-framed two-component unlink has rank 4.  Therefore, the non-invariance of Framed Floer homology is what allows us the possibility to distinguish different surgery presentations.  

However, it is not true in general that the Framed Floer homology is an invariant of the four-manifold attained by attaching two-handles along $L$ according to $\Lambda$.  In fact, the framed links $(L_1,\mathbf{0})$ and $(L_2,\mathbf{0})$ in Theorem~\ref{noninvariancetheorem} yield diffeomorphic four-manifolds.  The difference is that in this case the cobordism map on Heegaard Floer homology in the torsion Spin$^c$ structure is trivial.  For a similar reason, Framed Floer homology will not be useful for Property 2R in the non-torsion Spin$^c$ structures.  

More generally, it would be interesting to bring Heegaard Floer homology to the table to study this problem.  For example, if an $n$-component link surgers to $\#_n S^2 \times S^1$, then the link bounds smooth disks in a homotopy 4-ball (see Theorem 2 of \cite{hillmanideals} for the topological case or Proposition 2.3 of \cite{gompfscharlemannthompson} for the smooth case).  Therefore, by Theorem 1.1 of \cite{hfkfourballgenus}, the knot Floer homology concordance invariant $\tau$ must be zero for each component of the link.  On the other hand, if an $n$-component link $L$ which surgers to $\#_n S^2 \times S^1$ can be related to the unlink by handleslides only, then there are additional constraints on the knot Floer complexes of each component.  In particular, for each component $K$, the map $\Psi^K_0 = (\Phi_0^{+K})_* + (\Phi_0^{-K})_*$ must be identically zero on $H_*(\widehat{A}(K,0))$.  This is because the component of the $d_1$ differential into $E^0_\infty(\widehat{C}(L,\mathbf{0}))$ must vanish in order for the cobordism map induced by the two-handle attachments to be non-zero (this is analogous to the proof of Lemma~\ref{unlinkisinjective}).

\section{Future Directions for $\widehat{FFH}$}\label{futuresection}

It is natural to expect there to be some relation between the reduced Khovanov homology of a link $L$ and the Framed Floer homology of some appropriately chosen surgery presentation of $\Sigma(\overline{L})$.  The most obvious choice would be the framed link determined by a one-circle resolution (see Figure 19 in \cite{jbloomsurgery} for an illustration of the construction); let's denote the result by $(\widetilde{L},\widetilde{\Lambda})$.  Ideally, one would hope that $\widehat{FFH}(\widetilde{L},\widetilde{\Lambda})$ and $\widetilde{Kh}(L)$ are in fact isomorphic.  This would give some new insight into how Khovanov homology and Heegaard Floer homology are related.  A similar idea would be to attempt to recreate the argument of Oszv\'ath and Szab\'o in \cite{hfbranched}, but incorporating the ideas from Section~\ref{surgeryexactsection} to give a spectral sequence from $\widetilde{Kh}(L)$ to $\widehat{HF}(\Sigma_2(\bar{L}))$ which could also be computed combinatorially.  Theorem~\ref{sestheorem} would have to be extended to knots that are not necessarily nullhomologous.

As of now, there is not a clear understanding of what Framed Floer homology actually represents.  It would be interesting to see a definition which did not require the surgery formula or hypercubes of chain complexes.  It is known that Heegaard Floer homology can detect many intrinsic topological properties, such as the Thurston norm \cite{hfgenus}; it would be interesting to see if $\widehat{FFH}$ can detect specific structures, especially in light of Proposition~\ref{nonmirrortheorem}.  As in the work of Baldwin \cite{higherdiffs}, one should also try to study the subsequent pages and higher differentials in the $\varepsilon$-spectral sequence.  For example, the $d_3$ differential on the $\infty$ flavor contains information about the cup product structure of the surgered manifold \cite{hftriple}.   

Another direction would be to explore further the natural cobordism maps arising in Framed Floer homology as discussed in Section~\ref{property2rsection}.  

Finally, we discuss the computational advantages of Framed Floer homology.  Suppose $Y$ is presented by surgery on an $n$-component link.  To compute Heegaard Floer homology from the link surgery formula, one must make many holomorphic polygon counts, from bigons up to $(n+2)$-gons.  However, Framed Floer homology can be calculated with only bigons and triangles, making it significantly less involved.  This avoids many of the difficulties that the full theory of the surgery formula faces, especially when dealing with the combinatorial setting.  One would therefore like to see it used in an application where the Heegaard Floer homology is too difficult to explicitly calculate using the link surgery formula because of the higher polygon maps.

\bibliographystyle{plain}

\bibliography{biblio}

\end{document}